\RequirePackage{fix-cm}
\documentclass[smallextended]{svjour3}       
\smartqed  
\usepackage{graphicx}
\usepackage{lipsum}
\usepackage{amsfonts}
\usepackage{graphicx}
\usepackage{epstopdf}
\usepackage{algorithmic}
\usepackage{microtype}
\usepackage{subcaption}
\usepackage{booktabs}
\usepackage{hyperref}
\usepackage{amsfonts,mathrsfs}
\usepackage{verbatim}
\usepackage{acronym,color}
\usepackage{pifont}
\usepackage[normalem]{ulem}
\usepackage{yhmath}
\usepackage{amsmath}
\usepackage{amssymb}
\usepackage{mathtools}
\usepackage{enumerate}
\usepackage{cite}
\usepackage{caption}
\usepackage{xcolor}

\newtheorem{assumption}[theorem]{Assumption}

\def\la{\langle}
\def\ra{\rangle}

\def\cF{{\cal F}}

\def\1{{\bf 1}}

\def\e{\epsilon}

\def\R{\mathbb{R}}

\def\e{\epsilon}

\newcommand{\EXP}[1]{\mathsf{E}\!\left[#1\right] }

\def\argmin{\mathop{\rm argmin}}

\def\be{\begin{equation}}
\def\ee{\end{equation}}

\def\bit{\begin{itemize}}
\def\eit{\end{itemize}}

\providecommand{\ip}[2]{\langle #1, #2 \rangle} 
\newcommand{\cmark}{\ding{51}}
\newcommand{\xmark}{\ding{55}}
\newcommand{\remove}[1]{}

\begin{document}

\title{Popov Mirror-Prox Method for Variational Inequalities\thanks{This work was funded by the National Science Foundation awards CCF-2106336 and CIF-2134256.} \thanks{Some of the preliminary results have appeared in \cite{chakraborty2024popov} (accepted in IEEE CDC 2024) for deterministic settings but under more restrictive assumptions on the VI mapping being monotone and Lipschitz continuous. Here, we work on mappings having a polynomial growth (eg. H\"older growth) on its variation (possibly having discontinuous jumps) and extend our results for both stochastic and deterministic VI scenario.}
}



\author{Abhishek Chakraborty         \and
        Angelia Nedi\'c 
}


\institute{Abhishek Chakraborty \at
              School of Electrical Computer and Energy Engineering, Arizona State University, Tempe, Arizona, USA\\
              \email{achakr61@asu.edu}           
           \and
           Angelia Nedi\'c \at
              School of Electrical Computer and Energy Engineering, Arizona State University, Tempe, Arizona, USA\\
              \email{Angelia.Nedich@asu.edu}
}

\date{Received: date / Accepted: date}

\maketitle

\begin{abstract}
This paper establishes the convergence properties of the Popov mirror-prox algorithm for solving stochastic and deterministic variational inequalities (VIs) under a polynomial growth condition on the mapping variation. Unlike existing methods that require prior knowledge of problem-specific parameters, we propose step-size schemes that are entirely parameter-free in both constant and diminishing forms. For stochastic and deterministic monotone VIs, we establish optimal convergence rates in terms of the dual gap function over a bounded constraint set. Additionally, for deterministic VIs with H\"older continuous mapping, we prove convergence in terms of the residual function without requiring a bounded set or a monotone mapping, provided a Minty solution exists. This allows our method to address certain classes of non-monotone VIs. However, knowledge of the H\"older exponent is necessary to achieve the best convergence rates in this case. By extending mirror-prox techniques to mappings with arbitrary polynomial growth, our work bridges an existing gap in the literature. We validate our theoretical findings with empirical results on matrix games, piecewise quadratic functions, and image classification tasks using ResNet-18.
\end{abstract}

\section{Introduction}\label{sec:Intro}

Variational inequalities (VIs) provide a fundamental framework for modeling equilibrium problems~\cite{nagurney1999network, ralph1999equilibrium}, as well as applications in economics, optimization, and game theory~\cite{facchinei2003finite, kinderlehrer2000introduction}. Owing to their broad applicability, extensive research has focused on developing efficient algorithms for solving VIs. Classical methods target Lipschitz continuous mappings~\cite{korpelevich1976extragradient, popov1980modification} or mappings with bounded variation~\cite{nesterov2007dual}. More recently, motivated by applications such as smooth multi-armed bandits~\cite{liu2021smooth} and heart rate variability analysis~\cite{nakamura2005local}, H\"older continuous mappings have gained attention. Adaptive universal step-size methods have since been proposed to handle such settings in VIs without prior knowledge of the H\"older parameters~\cite{ablaev2022some, stonyakin2022generalized, klimza2024universal}. However, these methods still require knowledge of the diameter of the constraint set and are not completely parameter-free. In this work, we extend the analysis under a weaker assumption on the mapping $F$, requiring that
\begin{align}
    \|F(x) - F(y) \|_* \leq L_\nu \|x-y\|^{\nu} + M_\nu,  \qquad \forall x,y \in X, \label{mapping_growth}
\end{align}
with constants $L_\nu >0$, $M_\nu \geq 0$, and $\nu \in [0,\infty)$. The class of mappings satisfying ~\eqref{mapping_growth} generalizes the growth conditions studied in the literature, encompassing H\"older continuity ($\nu \in [0,1]$ and $M_\nu = 0$), Lipschitz continuity ($\nu = 1$ and $M_\nu = 0$), and bounded variation ($\nu = 0$). To the best of our knowledge, no prior work has studied VIs for such a general class of mappings, which can model a wide range of real-world problems in optimization, machine learning, and game theory.

The mirror-prox algorithm for VIs was originally developed for deterministic settings with Lipschitz continuous mappings in \cite{nemirovski2004prox}, and later extended to stochastic settings in \cite{juditsky2011solving} under the growth condition~\eqref{mapping_growth} with $\nu = 1$. Both works employ Korpelevich-style updates. However, mirror-prox can also be implemented using Popov-style updates which we initiated in~\cite{chakraborty2024popov}, where we established some preliminary results on the convergence rates
for a deterministic VI with a Lipschitz continuous and monotone mapping. In this current paper, we analyze the Popov mirror-prox algorithm for solving \textit{stochastic monotone} VIs, where the mapping $F$ satisfies the generalized growth condition~\eqref{mapping_growth}. We further extend the analysis to \textit{deterministic} VIs with \textit{H\"older continuous (not necessarily monotone)} mappings. These contributions distinguish this current paper from our initial work~\cite{chakraborty2024popov}. Our results apply to this broader class of mappings under both constant and diminishing, non-adaptive, parameter-free step sizes, achieving optimal convergence rates. However, in the special case where $M_\nu = 0$ and the mapping is deterministic, knowledge of the H\"older exponent is required to select an appropriate step size for optimal performance.

A variety of algorithms has been developed for solving VIs in both deterministic and stochastic settings. One of the earliest and most well-known algorithms for solving strongly monotone VIs is the projection method \cite{bertsekas2009projection}, which iteratively updates a candidate solution by descending along the direction of the VI mapping evaluated at the current iterate and then projecting it back onto the constraint set. However, it has been observed that for certain classes of monotone problems, the projection method fails to converge and can even diverge \cite{gidel2018variational, mokhtari2020unified}, which has prompted the study of alternative approaches, such as the Korpelevich (also known as Extragradient) method \cite{korpelevich1976extragradient} and the Popov method \cite{popov1980modification}. These methods introduce an intermediate step that refines the iterative process, leading to convergence guarantees in a broader range of problem settings. Korpelevich method uses two VI mapping evaluations for updating a single iterate. In the regime of stochastic problems, work \cite{mishchenko2020revisiting} improves the stochastic Korpelevich method by reducing per-iteration oracle queries for estimating the stochastic VI mapping using a single random stochastic sample per iteration, while
still requiring two mapping evaluations. 
The Popov method cuts the computational cost of the mapping evaluations to half by reusing an old mapping value. Single step Popov style updates, also known as optimistic gradient descent ascent, have been studied in the literature \cite{mokhtari2020unified, jiang2022generalized}, as well as operator extrapolation techniques~\cite{kotsalis2022simple}. However, none of these methods have been studied under the growth 
condition~\eqref{mapping_growth}.

Mirror-prox algorithms generalize classical methods by using Bregman distances instead of the Euclidean distance. These methods can be more effective depending on the structure of the constraint set and the VI mapping. Although mirror-prox is conceptual and involves solving a subproblem at each iteration, \cite{nemirovski2004prox} showed that the Korpelevich-style updates can solve this subproblem in the deterministic setting, leading to the development of the Korpelevich mirror-prox method. Earlier works on deterministic VIs assumed monotone and Lipschitz continuous mappings. The paper~\cite{juditsky2011solving} extended the Korpelevich mirror-prox to the stochastic setting, assuming the monotonicity and a Lipschitz-type growth with bounded variation on the mapping, i.e., $\nu = 1$ and $M_\nu = M_1$ in relation~\eqref{mapping_growth}. The convergence rate in~\cite{juditsky2011solving} depends on $M_1$, the step size, the diameter of the constraint set, and the noise variance. In practice, these parameters are often unknown, making the method difficult to implement. Work \cite{dang2015convergence} studied the Korpelevich method under the H\"older continuity of the mapping but required knowledge of problem constants and the H\"older exponent for step-size selection. Motivated by such limitations, universal algorithms have been proposed \cite{babanezhad2020geometry, bach2019universal, rodomanov2024universality, klimza2024universal} to address Nesterov's question \cite{nesterov2007dual} — whether methods can adapt to the continuity level of the VI mapping. However, even these universal methods are not entirely parameter-free, as they still require knowledge of the constraint set's diameter for step size selection.

The Popov variant of the mirror-prox algorithm was introduced in~\cite{semenov2017version}, where the convergence of the iterates was shown for Lipschitz continuous and pseudo-monotone mappings. The method was further applied to min-max GAN problems in~\cite{gidel2018variational}. Work~\cite{azizian2021last} studied the method under Lipschitz continuity for quasi-sharp stochastic VIs and established convergence rates. In the deterministic setting, the paper~\cite{azizian2024rate} analyzes the method under the assumption that the mapping is Lipschitz continuous and strongly monotone within a convex neighborhood, and examines last iterate convergence and the role of different Bregman distance functions. More recently, \cite{chakraborty2024popov} showed that for deterministic monotone VIs with Lipschitz continuous mappings, Popov-style updates using proximal mappings can approximate the conceptual mirror-prox method~\cite{nemirovski2004prox}, with the accumulated approximation error vanishing under proper initialization~\cite[Lemma 3]{chakraborty2024popov}. 
These results provide a strong theoretical foundation for extending the Popov mirror-prox method to more general settings, including both stochastic and deterministic VIs under relation~\eqref{mapping_growth}. To the best of our knowledge, no prior work has addressed convergence analysis for such VIs using parameter-free schemes with both diminishing and constant step sizes, as done in this paper.

\begin{table}[t]
    \centering
    \begin{tabular}{ccccc}
         \toprule
         VI Type & Monotone & Parameters in~\eqref{mapping_growth} & Bounded Set & Merit Function\\
         \midrule
         Stochastic& \cmark & $M_\nu \neq 0$, $\nu \in [0,\infty)$ & \cmark & 
         Expected Dual Gap\\
         \hline
         Deterministic & \xmark & $M_\nu = 0$, $\nu \in (0,1]$ & \xmark & Residual\\
         \bottomrule
    \end{tabular}
    \caption{A summary of our assumptions and the merit function used to analyze the convergence of Popov mirror-prox for stochastic and deterministic VIs.}
    \label{tab:my_work}
\end{table}

\noindent \textbf{Contributions and flow of the paper:} For the first time, we study the convergence rates of stochastic and deterministic Popov mirror-prox algorithm with the VI mapping satisfying relation~\eqref{mapping_growth}. We employ constant and diminishing step-size schemes that are completely problem-parameter-free, but non-adaptive in the stochastic case. We establish the convergence of the method in terms of the (expected) dual gap function over a bounded constraint set. Additionally, for H\"older continuous deterministic VIs with exponent $\nu \in (0,1]$ and $M_\nu = 0$ in~\eqref{mapping_growth}, we prove the convergence of the algorithm in terms of the residual function~\cite{dang2015convergence}, without requiring a bounded set or a monotone VI mapping, as long as a Minty solution exists. Consequently, our approach can also handle non-monotone VIs; however, achieving the optimal rates requires knowledge of problem parameters, particularly the H\"older exponent $\nu$. Table~\ref{tab:my_work} concisely summarizes the assumptions on the mapping and the associated merit functions used in this work for showing the convergence rates of the method.

The flow of the paper is as follows. 
We formulate the problem in Section~\ref{sec:problem} along with the related assumptions. In Section~\ref{sec:stoc_popov}, we introduce the stochastic Popov mirror-prox algorithm and state some auxiliary results on scalar sequences. The convergence rates of the stochastic case are studied in Section~\ref{sec:stochastic_popov_conv}. The deterministic Popov mirror-prox algorithm is provided in Section~\ref{sec:deterministic_popov} along with its convergence rates. We demonstrate the performance of the algorithm in Section~\ref{sec:simulations} through simulations for a noisy matrix game, minimization of a piecewise quadratic function, and MNIST and CIFAR-10 data classification using ResNet-18 model. We conclude the paper in Section~\ref{sec:conclusion}.

\smallskip

\noindent \textbf{Notations:} We consider the Euclidean space $\R^n$ and use $\la \cdot, \cdot \ra$ to denote the standard inner product. The primal space is the pair ($\R^n,\|\cdot\|$) equipped with some norm $\|\cdot\|$. The dual space is the pair ($\R^n, \|\cdot\|_*$), where 
 $\|\cdot\|_*$ is the conjugate norm associated with $\|\cdot\|$, defined by $\|x\|_* = \max_{\|z\| \leq 1} \la x,z \ra$.
\section{Problem Formulation and Assumptions}\label{sec:problem}
Let $X \subseteq \R^n$ be a given set in the primal space ($\R^n, \|\cdot\|$). Let $F(\cdot): X \rightarrow \R^n$ be a mapping, where the range of the mapping $F$ belongs to the dual space ($\R^n, \|\cdot\|_*$). The classic variational inequality problem VI$(X,F)$ (Stampacchia \cite{stampacchia1970variational}) is the problem of determining 
a point $x^* \in X$ such that
\begin{align}
    \la F(x^*), x-x^* \ra \geq 0, \quad \forall x \in X . \label{strong-vi}
\end{align} Such a point $x^*$ is a classic solution to the  VI$(X,F)$, which is also referred to as a ``strong" solution to differentiate it from a different concept of a solution. Namely, given a set $X \subseteq \R^n$ and a mapping $F(\cdot): X \rightarrow \R^n$, one seeks to find a point $x^*\in X$ such that 
\begin{align}
    \la F(x), x-x^* \ra \geq 0, \quad \forall x \in X. \label{minty-vi}
\end{align}
Such a point is referred to as a ``weak" solution to VI$(X,F)$, or 
a Minty solution~\cite{minty1967generalization, kinderlehrer2000introduction}. We will refer to such a point as a Minty solution, while a point satisfying~\eqref{strong-vi} will be referred to as a solution.

We consider Stochastic Variational Inequality (SVI) problem, where the mapping $F(\cdot)$ is given as the expected value, i.e., $F(x) = \EXP{\widehat F(x, \xi)}$ for all $x \in X$, with $\EXP{\cdot}$ being the expectation operator with respect to the random variable $\xi$ taking values in a set $\Omega$.
The solution concepts for the SVIs are the same as those for the determistic VIs.

Next, we provide some assumption that will be used in the sequel. Not all of them will be always used.
\begin{assumption}\label{asum-set1}
    The constraint set $X\subseteq\mathbb{R}^n$ is non-empty, closed, and convex.
\end{assumption}
We implicitly assume that the set $X$ is easy to project on.
\begin{assumption}\label{asum-holder}
    The mapping $F: X \rightarrow \R^n$ satisfies the following condition
    for some $\nu \in [0,\infty)$, $L_\nu > 0$, and $M_\nu \geq 0$,
    \begin{align*}
        \|F(x) - F(y) \|_* \leq L_\nu \|x-y\|^{\nu} + M_\nu , \qquad \forall x,y \in X.
    \end{align*}
\end{assumption}
    Assumption~\ref{asum-holder} allows for a discontinuous mapping $F$. 
    When $M_\nu = 0$ and $\nu \in [0,1]$, Assumption~\ref{asum-holder} reduces to the H\"older continuity~\cite{ablaev2022some}:
    \begin{align}
        \|F(x) - F(y) \|_* \leq L_\nu \|x-y\|^{\nu} , \qquad \forall x,y \in X. 
        \label{holder-smooth}
    \end{align}
    Moreover, for $\nu \in (0,\infty)$, the preceding relation implies that the mapping $F$ is continuous, and it is Lipschitz continuous when $\nu = 1$. 
    When $\nu = 0$  in Assumption~\ref{asum-holder}, 
    the condition reduces to the bounded variation for the mapping. 

%

%
%


We will work with the Bregman divergence $B_\psi(\cdot,\cdot)$ induced by a continuously differentiable and strongly convex function $\psi:X \rightarrow \R$. Specifically, we assume that the function $\psi(\cdot)$ is $\alpha$-strongly convex, i.e., for $\alpha>0$,
    \begin{align}
        \psi(z) \geq \psi(x) + \la \nabla \psi(x), z-x \ra + \frac{\alpha}{2} \| z - x \|^2 , \quad \forall z,x \in X. \nonumber
    \end{align}
The Bregman divergence $B_\psi(\cdot,\cdot)$ is given by
\begin{align}
    B_{\psi}(z,x) = \psi(z) - \psi(x) - \ip{\nabla \psi(x)}{z-x}, \quad \forall z,x \in X . \label{breg-dist}
\end{align}
By the strong convexity of the function $\psi$, we have
%
%
\begin{align}
    B_{\psi}(z,x) \ge \frac{\alpha}{2} \|z-x\|^2, \quad \forall z,x \in X. \label{breg-lb}
\end{align}
When $\psi(x) = \frac{1}{2} \|x\|^2$ for all $x \in X$, relation~\eqref{breg-lb} holds as equality with $\alpha = 1$.



\section{Stochastic Popov Mirror-Prox}\label{sec:stoc_popov}

Here, we consider an SVI problem VI$(X,F)$ with 
$F(x) = \EXP{\widehat F(x, \xi)}$ for all $x \in X$, where the random variable $\xi$ takes values in a sample set $\Omega$.
For such a problem, we present the stochastic Popov mirror-prox algorithm, some results related to step size selection, and some basic relations for  
the iterates.

\subsection{Stochastic Popov Mirror-prox Algorithm}
The stochastic Popov mirror-prox algorithm uses stochastic evaluations of the mapping $\EXP{\widehat F(x, \xi)}$ at a given $x \in X$ by drawing a random sample $\xi(x) \in \Omega$ (according to the true distribution of $\xi$). 
The algorithm maintains two sequences $\{y_t\}$ and $\{x_t\}$, defined as:
\begin{align}
    &y_{t+1} = \argmin_{z \in X} [\la \gamma_t \widehat F(y_{t}, \xi_{t}) - \nabla \psi(x_{t}), z \ra + \psi(z)], \label{stoc-pov1}\\
    &x_{t+1} = \argmin_{z \in X} [\la \gamma_t \widehat F(y_{t+1}, \xi_{t+1}) - \nabla \psi(x_{t}), z \ra + \psi(z)] , \label{stoc-pov2}
\end{align}
where $\gamma_t > 0$ is a stepsize and $\xi_t \in \Omega$ is a random sample drawn at the iterate $y_t$.
Note that only one random sample $\xi_t \in \Omega$ and one mapping evaluation $\widehat F(y_t,\xi_t)$ is used to compute $x_t$ and $y_{t+1}$. The method is initiated randomly by choosing $x_0\in X$ according to some distribution with the support on the set $X$, and setting $y_0 = x_0$.
We assume that $\EXP{\|x_0\|^2} < \infty$.
The output of the algorithm is the weighted average $y^{(t+1)}$ of $\{y_t\}$, given by
\begin{align}
    y^{(t+1)} = \frac{\sum_{\tau=\tilde t}^t \omega_\tau y_{\tau+1}}{\sum_{\tau=\tilde t}^t \omega_\tau}, \qquad\hbox{$t\ge 0$,} \label{avg-itr}
\end{align}
where $t\ge \tilde t$, with some choice of $\tilde t$ and weights $\omega_\tau>0$ for all $\tau = \tilde t, \ldots, t$ 
to be discussed later. Compared to the Korpelevich version of the stochastic mirror-prox algorithm~\cite{juditsky2011solving}, the stochastic Popov mirror-prox uses a single mapping computation and, hence, can be computationally more effective.

\subsection{Preliminary Results}

Here, we first present some preliminary results on the upper and lower bounds of the sums of scalar sequences, which will be applied to the step sizes later on.

%
%
\begin{lemma}\label{lem-step-bound}
    Given any $T \geq 1$, for $\gamma_t = \frac{c}{\sqrt{t+1}}$, with $c>0$, we have
        \begin{align}
            \sum_{t=\lceil \frac{T}{2} \rceil}^T \gamma_t^2 \leq c^2 \ln 4 , \quad \text{and } \quad \sum_{t=\lceil \frac{T}{2} \rceil}^T \gamma_t \geq \frac{c \sqrt{T+1}}{2} . \nonumber
        \end{align}
\end{lemma}
The proof of Lemma~\ref{lem-step-bound} is given in Appendix~\ref{lem-step-bound-proof}.
\begin{lemma}\label{lem-step-bd2}
    For $\gamma_t = \frac{c}{(t+1)^a}$, $t\ge0$, with constants $c>0$ and $0<a<1$, we have the following relations, for any $T \geq 0$:
    \begin{enumerate}[i)]
        \item The quantity $\sum_{t=0}^T \gamma_t$ is lower bounded as
        \begin{align}
            \sum_{t=0}^T \gamma_t \geq c(T+1)^{1-a}. \nonumber
        \end{align}
        \item The quantity $\sum_{t=0}^T \gamma_t^{\frac{2}{1-p}}$, for any $p \in [0,1)$, is upper bounded as
        \begin{align}
            \sum_{t=0}^T \gamma_t^{\frac{2}{1-p}} \leq \begin{cases}
                \frac{(1-p) c^{\frac{2}{1-p}}}{1-p-2a} (T+1)^{\frac{1-p-2a}{1-p}}  \quad &\text{when $0<a<\frac{1-p}{2}$,}\\
                c^{\frac{2}{1-p}} \left( 1 + \ln (T+1) \right)  \quad &\text{when $a=\frac{1-p}{2}$,} \\
                \frac{ 2a c^{\frac{2}{1-p}}}{2a+p-1}  \quad &\text{when $\frac{1-p}{2}<a<1$} .
            \end{cases} \nonumber
        \end{align}
        \item The quantity $\sum_{t=0}^T \gamma_t^{\frac{2p}{1-p}}$, for any $p \in (0,1)$, is upper bounded as
        \begin{align}
            \sum_{t=0}^T \gamma_t^{\frac{2p}{1-p}} \leq \begin{cases}
                \frac{(1-p) c^{\frac{2p}{1-p}}}{1-p-2p a} (T+1)^{\frac{1-p-2p a}{1-p}} \quad &\text{when $0<a<\frac{1-p}{2 p}$,}\\
                c^{\frac{2 p}{1-p}} (1+ \ln (T+1))  \quad &\text{when $a=\frac{1-p}{2 p}$,} \\
                \frac{ 2pa c^{\frac{2 p}{1-p}}}{2p a+p-1}  \quad &\text{when $\frac{1-p}{2 p}<a<1$} .
            \end{cases} \nonumber
        \end{align}
        \item The quantity $\sum_{t=0}^T \frac{1}{\gamma_t}$ is lower bounded as
        \begin{align}
            \sum_{t=0}^T \frac{1}{\gamma_t} \geq \frac{T^{1+a}}{c(1+a)} . \nonumber
        \end{align}
        \item The quantity $\sum_{t=0}^T \gamma_t^2$, for $0<a<\frac{1}{2}$, is lower bounded as
        \begin{align}
            \sum_{t=0}^T \gamma_t^2 \geq c^2 (T+1)^{1-2a} . \nonumber
        \end{align}
    \end{enumerate}
\end{lemma}
The proof of Lemma~\ref{lem-step-bd2} is provided in Appendix~\ref{lem-step-bd2-proof}.
%
%
Next, we provide some basic relations for the iterates of the  algorithm. We start with the well known three point lemma for Bregman diverence~\cite[Lemma 3.1]{chen1993convergence} that will be useful in the convergence analysis of the algorithm. 

%
%
\begin{lemma}[\cite{chen1993convergence, beck2017first}]\label{lem-3pt}
Let $\psi: X \rightarrow \R$ be a proper, convex, and differentiable function. Given any three points $x,y,z \in X$, we have
\begin{align}
    B_\psi(z,x) + B_\psi(x,y) - B_\psi(z,y) = \la \nabla \psi(y) - \nabla \psi(x), z-x \ra . \nonumber
\end{align}
\end{lemma}

Using Lemma~\ref{lem-3pt}, we establish a basic relation for the iterates generated by the stochastic Popov mirror-prox algorithm. In what follows, for notational convenience, we use $b_t$ and $\e_{t}$ to define the sample-error and its norm respectively, 
\begin{align}
    b_t = F(y_t) - \widehat F(y_t, \xi_{t}) , \qquad \epsilon_t = \| b_t \|_* . \label{epsilon-def}
\end{align}
We also use the following relation, which is an immediate consequence of the definition of the iterates $y_{t+1}$ and $x_{t+1}$ of the stochastic Popov mirror-prox algorithm in~\eqref{stoc-pov1}--\eqref{stoc-pov2}
and  the optimality principle:
\begin{align}
    &\ip{\gamma_t \widehat F(y_t, \xi_{t}) - \nabla \psi(x_t) + \nabla \psi(y_{t+1})}{z-y_{t+1}} \geq 0 , \quad \forall z \in X , \label{opt_y_pop}\\
    &\ip{\gamma_t \widehat F(y_{t+1}, \xi_{t+1}) - \nabla \psi(x_t) + \nabla \psi(x_{t+1})}{z-x_{t+1}} \geq 0 , \quad \forall z \in X . \label{opt_x_pop}
\end{align}

\begin{lemma}\label{lem-gen-stoc}
    Let Assumption~\ref{asum-set1} and Assumption~\ref{asum-holder}, with $\nu>0$, hold. Then, 
    for the iterates of the stochastic Popov mirror-prox, we have for any $z \in X$,
    \begin{align}
       &B_{\psi}(z,x_{t+1}) \leq B_{\psi}(z,x_t) - \left(1- \frac{2 w_1+w_2+w_3+w_4}{\alpha} \right) B_{\psi}(x_{t+1},y_{t+1}) \nonumber\\
       &+\frac{\gamma_t^2 L_{\nu}^2}{2 w_2} \left(\frac{2}{\alpha} B_\psi(x_t,y_t) \right)^\nu - B_{\psi}(y_{t+1},x_{t}) + \gamma_t \la F(y_{t+1}), z - y_{t+1}\ra \nonumber\\
    & + \gamma_t \la b_{t+1}, y_{t+1} - z \ra + \frac{\gamma_t^2 L_{\nu}^2}{2 w_3} \left(\frac{2}{\alpha} B_\psi(y_{t+1},x_t) \right)^\nu + \frac{\gamma_t^2}{2 w_1} (\epsilon_{t}^2 + \epsilon_{t+1}^2 ) + \frac{2 \gamma_t^2 M_{\nu}^2}{w_4} , \nonumber
   \end{align}
   where $w_1$, $w_2$, $w_3$, and $w_4$ are some positive constants, $\gamma_t > 0$ is the step size, and $b_{t+1}$, $\epsilon_{t}$ and $\epsilon_{t+1}$ are the errors as defined in~\eqref{epsilon-def} . 
   If Assumption~\ref{asum-holder} holds with $\nu=0$ instead, we then have: for any $z \in X$,
   \begin{align}
        B_{\psi}(z,x_{t+1}) \leq &B_{\psi}(z,x_t) - \left(1 - \frac{2w_1 + w_2}{\alpha} \right) B_{\psi}(x_{t+1},y_{t+1}) - B_{\psi}(y_{t+1},x_{t}) \nonumber\\
    & + \gamma_t \la F(y_{t+1}), z - y_{t+1} \ra + \gamma_t \la b_{t+1}, y_{t+1} - z \ra \nonumber\\
    &+\frac{\gamma_t^2}{2 w_1} (\epsilon_{t}^2 + \epsilon_{t+1}^2 ) + \frac{\gamma_t^2 (L_0+M_0)^2}{2 w_2}. \nonumber
    \end{align}
\end{lemma}
\begin{proof}
    By using Lemma~\ref{lem-3pt} with $x=z_{t+1}$ and $y=x_t$, we obtain for any $z\in X$,
    \begin{align}
        B_{\psi}(z,x_{t+1}) &= B_{\psi}(z,x_t) - B_{\psi}(x_{t+1},x_{t}) + \la \nabla \psi(x_{t+1}) - \nabla \psi(x_t), x_{t+1} - z \ra . \nonumber
    \end{align}
    Applying Lemma~\ref{lem-3pt} again with $x=y_{t+1}$, $y=x_t$ and $z=x_{t+1}$ for the quantity $B_{\psi}(x_{t+1},x_{t})$ of the preceding relation, we obtain
    \begin{align}
        &B_{\psi}(z,x_{t+1}) = B_{\psi}(z,x_t) - B_{\psi}(x_{t+1},y_{t+1}) - B_{\psi}(y_{t+1},x_{t}) \nonumber\\
        & - \la \nabla \psi(y_{t+1}) - \nabla \psi(x_t), x_{t+1} - y_{t+1} \ra - \la \nabla \psi(x_{t+1}) - \nabla \psi(x_t), z - x_{t+1} \ra . \label{thm-conv-eq1}
    \end{align}
    The fourth and the fifth terms on the right hand side of \eqref{thm-conv-eq1} can estimated using the optimality conditions \eqref{opt_y_pop} and \eqref{opt_x_pop}, respectively, as follows:
    \begin{align}
        & - \la \nabla \psi(y_{t+1}) - \nabla \psi(x_t), x_{t+1} - y_{t+1} \ra \leq \gamma_t \la \widehat F(y_t, \xi_{t}), x_{t+1} - y_{t+1} \ra \nonumber\\
        & - \la \nabla \psi(x_{t+1}) - \nabla \psi(x_t), z - x_{t+1} \ra \leq \gamma_t \la \widehat F(y_{t+1}, \xi_{t+1}), z -x_{t+1} \ra . \nonumber
    \end{align}
    Using the preceding two inequalities in relation \eqref{thm-conv-eq1}, we have
    \begin{align}
        B_{\psi}(z,x_{t+1}) &\leq B_{\psi}(z,x_t) - B_{\psi}(x_{t+1},y_{t+1}) - B_{\psi}(y_{t+1},x_{t}) \nonumber\\
        &+ \gamma_t \la \widehat F(y_t, \xi_{t}), x_{t+1} - y_{t+1} \ra + \gamma_t \la \widehat F(y_{t+1}, \xi_{t+1}), z -x_{t+1} \ra . \nonumber
    \end{align}
Next, we add and subtract the quantity $\gamma_t \la \widehat F(y_{t+1}, \xi_{t+1}), x_{t+1} - y_{t+1} \ra$ on the right hand side of the preceding relation to obtain
\begin{align}\label{holder-eq1}
    &B_{\psi}(z,x_{t+1}) \leq B_{\psi}(z,x_t) - B_{\psi}(x_{t+1},y_{t+1}) - B_{\psi}(y_{t+1},x_{t}) \\
    &\ + \gamma_t \la \widehat F(y_t, \xi_{t}) - \widehat F(y_{t+1}, \xi_{t+1}), x_{t+1} - y_{t+1} \ra + \gamma_t \la \widehat F(y_{t+1}, \xi_{t+1}), z - y_{t+1} \ra . \nonumber
\end{align}
The last term on the right hand side of relation~\eqref{holder-eq1} can be written as
\begin{align}
    \gamma_t \la \widehat F(y_{t+1}, \xi_{t+1}), z - y_{t+1} \ra = \gamma_t \la F(y_{t+1}), z - y_{t+1} \ra + \gamma_t \la b_{t+1}, y_{t+1} - z \ra , \nonumber
\end{align}
where the quantity $b_{t+1} = F(y_{t+1}) - \widehat F(y_{t+1}, \xi_{t+1})$, as per definition~\eqref{epsilon-def}. The preceding relation when used in \eqref{holder-eq1} yields
\begin{align}
    & B_{\psi}(z,x_{t+1}) \leq B_{\psi}(z,x_t) \!-\! B_{\psi}(x_{t+1},y_{t+1}) \!-\! B_{\psi}(y_{t+1},x_{t}) + \gamma_t \la F(y_{t+1}), z \!-\! y_{t+1} \ra  \nonumber\\
    & \ + \gamma_t \la \widehat F(y_t, \xi_{t}) - \widehat F(y_{t+1}, \xi_{t+1}), x_{t+1} - y_{t+1} \ra + \gamma_t \la b_{t+1}, y_{t+1} - z \ra . \label{holder-eq1-next}
\end{align}
For the term $\gamma_t \la \widehat F(y_t, \xi_{t}) - \widehat F(y_{t+1}, \xi_{t+1}), x_{t+1} - y_{t+1} \ra$ of relation~\eqref{holder-eq1-next}, we have
\begin{align}
    &\gamma_t \la \widehat F(y_t, \xi_{t}) - \widehat F(y_{t+1}, \xi_{t+1}), x_{t+1} - y_{t+1} \ra = \gamma_t \la(F(y_t) - F(y_{t+1})) , x_{t+1} - y_{t+1} \ra \nonumber\\
        &\qquad+\gamma_t \la (\widehat F(y_t, \xi_{t}) - F(y_t)) - (\widehat F(y_{t+1}, \xi_{t+1})- F(y_{t+1})), x_{t+1} - y_{t+1} \ra . \nonumber
\end{align}
Next, using the Cauchy-Schwarz inequality in the preceding relation we obtain
    \begin{align}
        &\gamma_t \la \widehat F(y_t, \xi_{t}) - \widehat F(y_{t+1}, \xi_{t+1}), x_{t+1} - y_{t+1} \ra \leq \gamma_t \Big( \|\widehat F(y_t, \xi_{t}) - F(y_t)\|_* \nonumber\\
         &+ \|\widehat F(y_{t+1}, \xi_{t+1}) - F(y_{t+1})\|_* + \|F(y_t) - F(y_{t+1})\|_* \Big) \|x_{t+1} - y_{t+1}\| . \label{holder-eq3}
    \end{align}
    Using the definitions of $\epsilon_{t}$ and $\epsilon_{t+1}$ (see~\eqref{epsilon-def}), and using Young's inequality, we further have
    \begin{align}\label{holder-eq4}
        &\gamma_t \la \widehat F(y_t, \xi_{t}) - \widehat F(y_{t+1}, \xi_{t+1}), x_{t+1} - y_{t+1} \ra \\
        &\leq \frac{\gamma_t^2}{2 w_1} (\epsilon_{t}^2 + \epsilon_{t+1}^2 ) + w_1 \|x_{t+1} - y_{t+1}\|^2 + \gamma_t \|F(y_t) - F(y_{t+1})\|_* \|x_{t+1} - y_{t+1}\| , \nonumber
    \end{align}
    where $w_1$ is a positive constant, and we use  Assumption~\ref{asum-holder} to estimate the term $\|F(y_t) - F(y_{t+1})\|_*$. Now, we consider the cases $\nu>0$ and $\nu=0$, separately.
    
    \noindent {\it Case $\nu>0$:}
    To estimate the term $\|F(y_t) - F(y_{t+1})\|_*$ on the right hand side of~\eqref{holder-eq4}, we add and subtract $F(x_t)$, and use triangle inequality, thus giving us
    \begin{align*}
    \|F(y_t) - F(y_{t+1})\|_* 
    &\le \|F(y_t) - F(x_t)\|_* + \|F(y_{t+1})-F(x_t)\|_*\cr
    &\le L_\nu \|y_t-x_t\|^{\nu} + L_\nu \|y_{t+1}-x_t\|^{\nu} + 2M_\nu,    
    \end{align*}
    where in the last inequality we use Assumption~\ref{asum-holder}. Using the preceding relation and 
    Young's inequality, we have that
\begin{align}
        \gamma_t \|F(y_t) - F(y_{t+1})\|_* &\|x_{t+1} - y_{t+1}\| \leq \frac{\gamma_t^2 L_{\nu}^2}{2 w_2} \|y_t-x_t\|^{2\nu} +  \frac{\gamma_t^2 L_{\nu}^2}{2 w_3} \|y_{t+1}-x_t\|^{2 \nu} \nonumber\\
        &+ \frac{2 \gamma_t^2 M_{\nu}^2}{w_4} + \frac{w_2+w_3+w_4}{2} \, \|x_{t+1} - y_{t+1}\|^2, \nonumber
    \end{align}
    where $w_2$, $w_3$ and $w_4$ are positive constants. 
By the strong convexity of the Bregman distance (see~\eqref{breg-lb}),  it follows that 
    \begin{align}\label{holder-eq5}
        &\gamma_t \|F(y_t) - F(y_{t+1})\|_* \|x_{t+1} - y_{t+1}\| \leq \frac{\gamma_t^2 L_{\nu}^2}{2 w_2} \left(\frac{2}{\alpha} B_\psi(x_t,y_t) \right)^{\nu} \\
        &+ \frac{\gamma_t^2 L_{\nu}^2}{2 w_3} \left(\frac{2}{\alpha} B_\psi(y_{t+1},x_t) \right)^{\nu} + \frac{w_2+w_3+w_4}{\alpha} B_\psi(x_{t+1},y_{t+1}) + \frac{2 \gamma_t^2 M_{\nu}^2}{w_4} . \nonumber
    \end{align}
    Combining relations \eqref{holder-eq4} and \eqref{holder-eq5} with relation~\eqref{holder-eq1-next}, we obtain
    \begin{align*}
        &B_{\psi}(z,x_{t+1}) \leq B_{\psi}(z,x_t) - \left(1- \frac{w_2+w_3+w_4}{\alpha} \right) B_{\psi}(x_{t+1},y_{t+1}) \nonumber\\
    &+ w_1 \|x_{t+1} - y_{t+1}\|^2 - B_{\psi}(y_{t+1},x_{t}) + \gamma_t \la F(y_{t+1}), z - y_{t+1} \ra \nonumber\\
    & + \gamma_t \la b_{t+1}, y_{t+1} - z \ra +\frac{\gamma_t^2 L_{\nu}^2}{2 w_2} \left(\frac{2}{\alpha} B_\psi(x_t,y_t) \right)^\nu + \frac{\gamma_t^2 L_{\nu}^2}{2 w_3} \left(\frac{2}{\alpha} B_\psi(y_{t+1},x_t) \right)^\nu \nonumber\\
    &+ \frac{\gamma_t^2}{2 w_1} (\epsilon_{t}^2 + \epsilon_{t+1}^2 ) + \frac{2 \gamma_t^2 M_{\nu}^2}{w_4} . 
    \end{align*}
   The norm term on the right hand side of the preceding inequality can be bounded from above using the strong convexity of the Bregman divergence, which finally gives us the first relation of the lemma.

\noindent{\it Case $\nu=0$:}
   To obtain the second relation of the lemma, we use Assumption~\ref{asum-holder} for the last term on the right hand side of relation~\eqref{holder-eq4} and, thus, obtain
    \begin{align}
        \gamma_t \|F(y_t) - F(y_{t+1})\|_* \|x_{t+1} - y_{t+1}\| &\leq \gamma_t (L_0+M_0) \|x_{t+1} - y_{t+1}\| \nonumber\\
        &\leq \frac{\gamma_t^2 (L_0+M_0)^2}{2 w_2} + \frac{w_2}{2} \|x_{t+1} - y_{t+1}\|^2 . \nonumber
    \end{align}
    We combine the preceding estimate with relation~\eqref{holder-eq4}. Then, we follow the same line of analysis as in the case $\nu >0$ and obtain the stated relation. \hfill$\square$
\end{proof}

Lemma~\ref{lem-gen-stoc} is the key for the convergence rate analysis of the algorithm. It will be repeatedly used with  suitable selection of the step size $\gamma_t$ and the positive constants $w_1$, $w_2$, $w_3$ and $w_4$.

Next, we provide a lemma that will be useful for the convergence analysis of the stochastic version of our algorithm.
\begin{lemma}[\cite{juditsky2011solving}]\label{lem-juditsky}
    Let the sequence $\{h_t\}_{t=0}^T$, for any $T \geq 1$, be defined as
    \begin{align*}
        h_{t+1} = \argmin_{z \in X} [\la \gamma_{t}  b_{t+1} - \nabla \psi(h_{t}), z \ra + \psi(z)] ,\qquad \forall t=0,1,\ldots, T-1,
    \end{align*}
    with $h_0 \in X$. Then, the following relation is satisfied for all $t=0,1,\ldots, T-1$,
    \begin{align*}
        \gamma_t \la b_{t+1}, h_{t} - z \ra \leq B_\psi(z,h_{t}) - B_\psi(z,h_{t+1}) + \frac{\gamma_t^2 \e_{t+1}^2}{2 \alpha} ,
    \end{align*}
    where $b_t$ and $\e_t$ are defined in relation~\eqref{epsilon-def}.
\end{lemma}
The proof of Lemma~\ref{lem-juditsky} can be found in~\cite[Lemma 3 and Corollary 2]{juditsky2011solving}. For the sake of completeness, we provide our version of the proof in Appendix~\ref{lem-juditsky-proof}, which uses Lemma~\ref{lem-3pt} and the optimality conditions for the iterates $\{h_{t}\}_{t=0}^T$.

\section{Convergence Rates of the Algorithm}\label{sec:stochastic_popov_conv}
In this section, we deal with a monotone mapping $F$.
\begin{assumption}\label{asum-monotone}
    The mapping $F: X \rightarrow \R^n$ is monotone over the set $X$, i.e.,
    \begin{align}
        \la F(x) - F(y), x-y \ra \geq 0 , \qquad \forall x,y \in X . \nonumber
    \end{align}
\end{assumption}

%
%

%
%
Using Lemma~\ref{lem-juditsky}, for a monotone mapping  
 we have a refinement of Lemma~\ref{lem-gen-stoc}.

\begin{lemma}\label{lem-error-control}
Let Assumptions~\ref{asum-set1}--\ref{asum-monotone} hold. Then, for the iterates of the stochastic Popov mirror-prox method for the case when $\nu>0$, we have
for any step size $\gamma_t>0$, $t\ge0$, and $z\in X$,
\begin{align*}
    &\gamma_t \left\la F(z), y_{t+1} - z \right\ra \leq B_{\psi}(z,x_t) - B_{\psi}(z,x_{t+1}) + B_\psi(z,h_{t}) - B_\psi(z,h_{t+1}) \cr 
       & +\frac{2\gamma_t^2 L_{\nu}^2}{\alpha} \left(\frac{2}{\alpha} B_\psi(x_t,y_t) \right)^\nu + \gamma_t \left\la b_{t+1}, y_{t+1} - h_t \right\ra \cr
    &+ \frac{2\gamma_t^2 L_{\nu}^2}{\alpha} \left(\frac{2}{\alpha} B_\psi(y_{t+1},x_t) \right)^\nu + \frac{\gamma_t^2}{\alpha} \left(4\epsilon_{t}^2 + \frac{9}{2}\epsilon_{t+1}^2 \right) + \frac{8 \gamma_t^2 M_{\nu}^2}{\alpha},
\end{align*}
where $b_t$ and $\e_t$ are defined in relation~\eqref{epsilon-def}, and $h_t$ is defined in Lemma~\ref{lem-juditsky}. When $\nu = 0$, the following relation holds
\begin{align*}
    &\gamma_t \left\la F(z), y_{t+1} - z \right\ra \leq B_{\psi}(z,x_t) - B_{\psi}(z,x_{t+1}) + B_\psi(z,h_{t}) - B_\psi(z,h_{t+1}) \cr 
    & + \gamma_t \left\la b_{t+1}, y_{t+1} - h_t \right\ra 
    +\frac{\gamma_t^2}{\alpha} \left(2\epsilon_{y_t}^2 + \frac{5}{2}\epsilon_{y_{t+1}}^2 \right) + \frac{\gamma_t^2 (L_0+M_0)^2}{\alpha} .
\end{align*}
\end{lemma}

\begin{proof}
First, we consider the case $\nu>0$ and
     use Lemma~\ref{lem-gen-stoc},  where we
     set $w_1 = \alpha/8$ and $w_2 = w_3 = w_4 = \alpha/4$, so that 
     $1- \frac{2 w_1+w_2+w_3+w_4}{\alpha} = 0$ and obtain
     \begin{align}
         &B_{\psi}(z,x_{t+1}) \leq B_{\psi}(z,x_t) 
       +\frac{2\gamma_t^2 L_{\nu}^2}{\alpha} \left(\frac{2}{\alpha} B_\psi(x_t,y_t) \right)^\nu \!-\! B_{\psi}(y_{t+1},x_{t}) \cr 
       & + \gamma_t \left\la F(y_{t+1}), z - y_{t+1} \right\ra + \gamma_t \left\la b_{t+1}, y_{t+1} - z \right\ra 
    + \frac{2\gamma_t^2 L_{\nu}^2}{\alpha} \left(\frac{2}{\alpha} B_\psi(y_{t+1},x_t) \right)^\nu \cr
    &+ \frac{4\gamma_t^2}{\alpha} \left(\epsilon_{t}^2 + \epsilon_{t+1}^2 \right) + \frac{8 \gamma_t^2 M_{\nu}^2}{\alpha} . \nonumber
     \end{align}
     We use the monotonicity of the mapping $F$
     (Assumption~\ref{asum-monotone}) and by re-arranging the terms in the preceding relation, we obtain 
     \begin{align*}
         &\gamma_t \left\la F(z), y_{t+1} - z \right\ra \leq B_{\psi}(z,x_t) - B_{\psi}(z,x_{t+1}) +\frac{2\gamma_t^2 L_{\nu}^2}{\alpha} \left(\frac{2}{\alpha} B_\psi(x_t,y_t) \right)^\nu \cr 
       & - B_{\psi}(y_{t+1},x_{t}) + \gamma_t \left\la b_{t+1}, y_{t+1} - z \right\ra 
    + \frac{2\gamma_t^2 L_{\nu}^2}{\alpha} \left(\frac{2}{\alpha} B_\psi(y_{t+1},x_t) \right)^\nu \cr
    &+ \frac{4\gamma_t^2}{\alpha} \left(\epsilon_{t}^2 + \epsilon_{t+1}^2 \right) + \frac{8 \gamma_t^2 M_{\nu}^2}{\alpha} .
     \end{align*}
    Next, in the fifth term on the right-hand side of the preceding relation, we add add and subtract the quantity $\gamma_t \la b_{t+1}, h_t \ra$ to arrive at
     \begin{align}
         &\gamma_t \left\la F(z), y_{t+1} - z \right\ra \leq B_{\psi}(z,x_t) - B_{\psi}(z,x_{t+1}) +\frac{2\gamma_t^2 L_{\nu}^2}{\alpha} \left(\frac{2}{\alpha} B_\psi(x_t,y_t) \right)^\nu \cr 
       & - B_{\psi}(y_{t+1},x_{t}) + \gamma_t \left\la b_{t+1}, h_t - z \right\ra + \gamma_t \left\la b_{t+1}, y_{t+1} - h_t \right\ra \cr
    &+ \frac{2\gamma_t^2 L_{\nu}^2}{\alpha} \left(\frac{2}{\alpha} B_\psi(y_{t+1},x_t) \right)^\nu + \frac{4\gamma_t^2}{\alpha} \left(\epsilon_{t}^2 + \epsilon_{t+1}^2 \right) + \frac{8 \gamma_t^2 M_{\nu}^2}{\alpha} , \nonumber
     \end{align}
     where $h_t$ is defined in Lemma~\ref{lem-juditsky}. The fifth quantity on the right-hand side of the preceding relation can be upper estimated using Lemma~\ref{lem-juditsky}. Upper estimating the fourth quantity on the right-hand side of the preceding relation by $0$, we finally obtain the required relation of the lemma for the case when $\nu>0$.

     For the case when $\nu=0$, we use the second relation of Lemma~\ref{lem-gen-stoc}  with $w_1 = \alpha/4$ $w_2 = \alpha/2,$ and obtain  
    \begin{align*}
        &B_{\psi}(z,x_{t+1}) 
        \leq B_{\psi}(z,x_t) - B_{\psi}(y_{t+1},x_{t}) 
    +\frac{2\gamma_t^2}{\alpha} (\epsilon_{y_t}^2 + \epsilon_{y_{t+1}}^2 ) \cr 
    &+ \frac{\gamma_t^2 (L_0+M_0)^2}{\alpha} + \gamma_t \la F(y_{t+1}), z - y_{t+1} \ra + \gamma_t \la b_{t+1}, y_{t+1} - z \ra. \nonumber
    \end{align*}
    The rest of the proof proceeds analogously to the case when $\nu > 0$. \hfill$\square$
\end{proof}

There is a dependence of $\gamma_t^2$ on the right-hand side of the relation in Lemma~\ref{lem-error-control}, whereas the dependence on the left-hand side is on $\gamma_t$. Proper selection of the step size is required to obtain the best convergence guarantees which we analyze in the sequel. 
We use the notion of a merit function to quantify the convergence rate of the algorithm. As a merit function, we consider the dual gap function \cite{marcotte1998weak, nemirovski2004prox}, defined as follows:
\begin{align}
    G(x) = \sup_{z \in X} \la F(z), x-z \ra, \quad \forall x \in X. \label{dg-gap}
\end{align}
The dual gap function is always non-negative and becomes zero only when $x=x^*$ is a \textit{Minty} solution (cf. \eqref{minty-vi}) of the VI$(X,F)$ \cite{chakraborty2024popov, juditsky2016solving}. Now, going back to the solution concepts of VI$(X,F)$, we know that under Assumption~\ref{asum-monotone}, the solution set is a subset of the Minty solution set \cite[Lemma 2.2]{huang2023beyond}. In addition to Assumption~\ref{asum-monotone}, when $\nu \in (0,\infty)$ and $M_\nu = 0$ in Assumption~\ref{asum-holder}, the mapping $F$ is continuous. In this case, the solutions (cf. \eqref{strong-vi}) and the Minty solutions (cf. \eqref{minty-vi}) of VI$(X,F)$ coincide~\cite[Lemma 1.5]{kinderlehrer2000introduction}. 

For the stochastic version of the algorithm, we assume a compact set $X$.
\begin{assumption}\label{asum-set2}
    The constraint set $X$ is bounded.
\end{assumption}
Since the Bregman divergence is a continuous function, under Assumption~\ref{asum-set2}, there exists a scalar $D>0$ such that
    \begin{align}
        \max_{x,z \in X} B_\psi(z,x) \leq D. \label{bounded-set}
    \end{align}

We let $\cF_t$ denote the random realization of the initial iterate $x_0$ and the realizations of the random variables $\xi_t$ up to time $t$, inclusively, i.e., 
\begin{align}
    \cF_t = \{x_0\} \cup \{\xi_\tau\}_{\tau=0}^t\quad\hbox{for  $t\ge0$}, \qquad \hbox{and}\qquad
    \cF_{-1} = \{x_0\}. \label{sigma-alg}
\end{align}
We make the following assumptions on the bias and variance of the random samples of the mapping $\widehat F(\cdot,\xi)$.

\begin{assumption}\label{asum-stoc-bias-var}
   Given the sigma-algebra $\mathcal{F}_{t-1}$, for all $z \in X$ and $t \geq 0$,
   the sample mapping $\widehat{F}(z, \xi_{t})$ is conditionally unbiased
   and its conditional variance is bounded by $\sigma^2$, i.e.,
   for all $z \in X$ and $t \geq 0$,
    \[\EXP{\widehat F(z, \xi_{t}) \mid \cF_{t-1}} = F(z),
        \qquad\quad
        \EXP{\|\widehat F(z, \xi_{t}) - F(z) \|_*^2 \mid \cF_{t-1}} \leq \sigma^2 . \]
\end{assumption}

\subsection{Weighted-averages of the iterates using step sizes as weights}
We analyze the convergence rate of the method using the weighted averages of the iterates with weights $\omega_\tau = \gamma_\tau$ in~\eqref{avg-itr}.
The following lemma is used to estimate the expected value of the dual gap function in terms of the step sizes.

\begin{lemma}\label{lem-avg}
    Let Assumptions~\ref{asum-set1}--\ref{asum-stoc-bias-var} hold. Then, for the iterates of the stochastic Popov mirror-prox we have for any $T\ge \tilde t\ge0$,
    \begin{align}
        \EXP{G \left( \frac{\sum_{t=\Tilde{t}}^T \gamma_t y_{t+1}}{\sum_{t=\Tilde{t}}^T \gamma_t} \right)} \leq \frac{2 D + \widehat D \sum_{t=\Tilde{t}}^T \gamma_t^2 }{\sum_{t=\Tilde{t}}^T \gamma_t} , \nonumber
    \end{align}
    where $D\ge \max_{x,z \in X} B_\psi(z,x)$ (see~\eqref{bounded-set}) and the constant $\widehat D$ is defined by
    \begin{align}
        \widehat D = \begin{cases}
        \frac{2^{2+\nu} L_{\nu}^2 D^\nu}{\alpha^{1+\nu}} + \frac{17\sigma^2 + 16 M_{\nu}^2}{2 \alpha} \quad &\text{when $\nu > 0$},\\
         \frac{9\sigma^2 + 2(L_0+M_0)^2}{2\alpha} \quad &\text{when $\nu = 0$}.
        \end{cases} \label{D_hat}
    \end{align}
\end{lemma}
\begin{proof}
    \noindent {\it Case $\nu>0$:} 
    Under Assumption~\ref{asum-set2}, 
    from Lemma~\ref{lem-error-control}, where we bound the fifth and the seventh term on the right-hand side of the relation given in the lemma, we obtain for any $z\in X$,
    \begin{align*}
        &\gamma_t \left\la F(z), y_{t+1} - z \right\ra \leq B_{\psi}(z,x_t) - B_{\psi}(z,x_{t+1}) + B_\psi(z,h_{t}) - B_\psi(z,h_{t+1}) \cr 
       & +\frac{2^{2+\nu} L_{\nu}^2 D^\nu}{\alpha^{1+\nu}} \gamma_t^2 + \gamma_t \left\la b_{t+1}, y_{t+1} - h_t \right\ra + \frac{\gamma_t^2}{\alpha} \left(4\epsilon_{t}^2 + \frac{9}{2}\epsilon_{t+1}^2 \right) + \frac{8 \gamma_t^2 M_{\nu}^2}{\alpha} .
    \end{align*}
    We sum the preceding relations for $t = \Tilde{t}, \ldots, T$, for some $T\ge \Tilde{t} \geq 0$  to obtain
    \begin{align}
        &\sum_{t=\Tilde{t}}^T \gamma_t \left\la F(z), y_{t+1} - z \right\ra \leq B_{\psi}(z,x_{\Tilde{t}}) - B_{\psi}(z,x_{T+1}) + B_\psi(z,h_{\Tilde{t}}) - B_\psi(z,h_{T+1}) \cr 
       & + \sum_{t=\Tilde{t}}^T \gamma_t \left\la b_{t+1}, y_{t+1} - h_t \right\ra + \left(\frac{2^{2+\nu} L_{\nu}^2 D^\nu}{\alpha^{1+\nu}} + \frac{4\epsilon_{t}^2 + \frac{9}{2}\epsilon_{t+1}^2 + 8 M_{\nu}^2}{\alpha} \right) \sum_{t=\Tilde{t}}^T \gamma_t^2 . \nonumber 
    \end{align}
    Wee upper estimate the second and the fourth terms on the right-hand side of the preceding relation with $0$ and, thus, we have
    \begin{align}
        \sum_{t=\Tilde{t}}^T \gamma_t \left\la F(z), y_{t+1} - z \right\ra \leq B_{\psi}(z,x_{\Tilde{t}}) + B_\psi(z,h_{\Tilde{t}}) + \sum_{t=\Tilde{t}}^T \gamma_t \left\la b_{t+1}, y_{t+1} - h_t \right\ra \cr
        + \left(\frac{2^{2+\nu} L_{\nu}^2 D^\nu}{\alpha^{1+\nu}} + \frac{4\epsilon_{t}^2 + \frac{9}{2}\epsilon_{t+1}^2 + 8 M_{\nu}^2}{\alpha} \right) \sum_{t=\Tilde{t}}^T \gamma_t^2. \label{hol-lem-smpl}
    \end{align}
    By linearity, the left-hand side of \eqref{hol-lem-smpl} can be expressed as
    \begin{align*}
        \sum_{t=\Tilde{t}}^T \gamma_t \la F(z), y_{t+1}- z \ra
        &=\left( \sum_{t=\Tilde{t}}^T \gamma_t \right) \left\la F(z), \frac{\sum_{t=\Tilde{t}}^T \gamma_t y_{t+1}}{\sum_{t=\Tilde{t}}^T \gamma_t} - z \right\ra.
    \end{align*}
    We use the preceding relation back in~\eqref{hol-lem-smpl},  simplify the terms, take the maximum over $z \in X$ on both sides of the relation, and obtain the following upper bound on the dual gap function (cf.~\eqref{dg-gap})
    \begin{align*}
        &G \left(\frac{\sum_{t=\Tilde{t}}^T \gamma_t y_{t+1}}{\sum_{t=\Tilde{t}}^T \gamma_t} \right) \leq \frac{1}{\left( \sum_{t=\Tilde{t}}^T \gamma_t \right)} \Bigg[\max_{z \in X} \left[ B_{\psi}(z,x_{\Tilde{t}}) + B_\psi(z,h_{\Tilde{t}}) \right] \cr
        &+ \sum_{t=\Tilde{t}}^T \gamma_t \left\la b_{t+1}, y_{t+1} - h_t \right\ra + \left(\frac{2^{2+\nu} L_{\nu}^2 D^\nu}{\alpha^{1+\nu}} + \frac{4\epsilon_{t}^2 + \frac{9}{2}\epsilon_{t+1}^2 + 8 M_{\nu}^2}{\alpha} \right) \sum_{t=\Tilde{t}}^T \gamma_t^2 \Bigg].
    \end{align*}
    We use Assumption~\ref{asum-set2} to bound the first two quantities on the right-hand side of the preceding relation as $\max_{z \in X} \left[ B_{\psi}(z,x_{\Tilde{t}}) + B_\psi(z,h_{\Tilde{t}}) \right] \leq 2D$. Next, we take total expectation on both sides of the preceding relation and use the linearity of expectation to write
    \begin{align}
        \EXP{G \left(\frac{\sum_{t=\Tilde{t}}^T \gamma_t y_{t+1}}{\sum_{t=\Tilde{t}}^T \gamma_t} \right)} \leq \frac{1}{\left( \sum_{t=\Tilde{t}}^T \gamma_t \right)} \Bigg[2D + \EXP{\sum_{t=\Tilde{t}}^T \gamma_t \left\la b_{t+1}, y_{t+1} - h_t \right\ra} \cr
        + \left(\frac{2^{2+\nu} L_{\nu}^2 D^\nu}{\alpha^{1+\nu}} + \frac{4 \EXP{\epsilon_{t}^2} + \frac{9}{2} \EXP{\epsilon_{t+1}^2} + 8 M_{\nu}^2}{\alpha} \right) \sum_{t=\Tilde{t}}^T \gamma_t^2 \Bigg] . \label{hol-lem-smpl2}
    \end{align}
    The second quantity on the right-hand side of relation~\eqref{hol-lem-smpl2} can simplified using the law of iterated expectation, i.e.,
    \begin{align*}
        \EXP{\sum_{t=\Tilde{t}}^T \gamma_t \left\la b_{t+1}, y_{t+1} - h_t \right\ra} = \EXP{\sum_{t=\Tilde{t}}^T \gamma_t \EXP{\left\la b_{t+1}, y_{t+1} - h_t \right\ra \mid \cF_t}} = 0.
    \end{align*}
    Using the preceding relation in~\eqref{hol-lem-smpl2} and Assumption~\ref{asum-stoc-bias-var} to estimate the error terms by $\EXP{\epsilon_{t}^2} \leq \sigma^2$ and $\EXP{\epsilon_{t+1}^2} \leq \sigma^2$, we obtain the desired relation for $\nu>0$.

    \noindent {\it Case $\nu=0$:} We use Lemma~\ref{lem-error-control} for the case $\nu=0$ and follow the same analysis as done for the case $\nu > 0$. 
    \hfill$\square$
\end{proof}

Now, we establish the convergence rates for different step sizes assuming that the total number $T$ of iterations is given in advance.
\begin{theorem}\label{thm-holder-known-itr}
    Let Assumptions~\ref{asum-set1}--\ref{asum-stoc-bias-var} hold 
    and the number $T\ge 0$ of iterations  be given.
    Then,  for the stochastic Popov mirror-prox algorithm, we have
    \begin{align}
        &\EXP{G \left( \frac{1}{T+1} \sum_{t=0}^T y_{t+1} \right)} \leq \frac{2D + c^2 \widehat D}{c \sqrt{T+1}} \ \hbox{when $\gamma_t \!=\! \gamma \!=\! \frac{c}{\sqrt{T+1}}$ for all $t=0,\ldots T$}, \nonumber\\
        &\EXP{G \left( \frac{\sum_{t= \lceil \frac{T}{2} \rceil}^T \gamma_t y_{t+1}}{\sum_{t= \lceil \frac{T}{2} \rceil}^T \gamma_t} \right)} \leq \frac{4D + 2  \widehat D c^2\ln 4 }{c \sqrt{T+1}} 
        \ \hbox{when $\gamma_t =\frac{c}{\sqrt{t+1}}$ for $t=0,\ldots T$}, \nonumber
    \end{align}
    where $c>0$, $\max_{x,z \in X} B_\psi(z,x)\le D$ (see~\eqref{bounded-set}), and $\widehat D$ is as in Lemma~\ref{lem-avg}.
\end{theorem}
\begin{proof}
    We use Lemma~\ref{lem-avg}, with $\tilde t=0$ and a fixed step size $\gamma_t = \gamma = \frac{c}{\sqrt{T+1}}$. Thus, we obtain
    \begin{align}
        \EXP{G \left( \frac{1}{T+1} \sum_{t=0}^T y_{t+1} \right)} \leq \frac{2 D + \widehat D (T+1) \gamma^2 }{(T+1)\gamma} =\frac{2D + \widehat D c^2 }{c\sqrt{T+1}}. \nonumber
    \end{align}
    For the second relation, we use Lemma~\ref{lem-avg} with 
    $\Tilde{t} = \lceil T/2 \rceil$ and $\gamma_t = \frac{c}{\sqrt{t+1}}$. Then, we use Lemma~\ref{lem-step-bound} to bound the stepsize related sums.
    \hfill$\square$
\end{proof}
Theorem~\ref{thm-holder-known-itr} shows that the expected value of the dual gap function evaluated at suitably averaged iterates converges at the rate of $O(\frac{1}{\sqrt{T+1}})$, which is a standard rate for stochastic methods. 
The above theorem also shows that, instead of using constant step sizes $O(\frac{c}{\sqrt{T+1}})$ which could be small, one can use larger step sizes $\gamma_t = \frac{c}{\sqrt{t+1}}$ 
and still achieve the same rate of $O(\frac{1}{\sqrt{T+1}})$ but with a different averaging window. Instead of $\left\lceil \frac{T}{2} \right\rceil$ we can choose any generic $\left\lceil \frac{T}{m} \right\rceil$, with $m>2$, and Theorem~\ref{thm-holder-known-itr} will hold with a slight modification of the constants. For simplicity, we used $m=2$.

When the number of iterations is not chosen in advance, then we can select diminishing step sizes for which we have the following result.
\begin{theorem}\label{thm-hol-conv-unknown-itr}
    Let Assumptions~\ref{asum-set1}--\ref{asum-stoc-bias-var} hold.
    Then, for the stochastic Popov mirror-prox with a diminishing step size $\gamma_t = \frac{c}{(t+1)^a}$, where the constants $c>0$ and $0<a<1$, we have the following rate result, for all $T\ge0$,
    \begin{align}
        \EXP{G \left( \frac{\sum_{t=0}^T \gamma_t y_{t+1}}{\sum_{t=0}^T \gamma_t} \right)} \leq
        \begin{cases}
             \frac{2D}{c(T+1)^{1-a}} + \frac{\widehat D c}{(1-2a) (T+1)^a} \quad &\text{when $0<a<\frac{1}{2}$},\\
             \frac{2D}{c(T+1)^{1-a}} + \frac{\widehat D c (1+\ln(T+1)) }{(T+1)^{1-a}}  \quad &\text{when $a=\frac{1}{2}$},\\
             \frac{2D}{c(T+1)^{1-a}} + \frac{ 2a \widehat D c }{(2a-1) (T+1)^{1-a}}  \quad &\text{when $\frac{1}{2}<a<1$} ,
        \end{cases} \nonumber
    \end{align}
    where $\max_{x,z \in X} B_\psi(z,x)\le D$ and $\widehat D$ is as in Lemma~\ref{lem-error-control} (cf.~\eqref{D_hat}).
\end{theorem}
\begin{proof}
    Using Lemma~\ref{lem-avg}, with $\Tilde{t} = 0$, we obtain
    \begin{align}
        \EXP{G \left( \frac{\sum_{t=0}^T \gamma_t y_{t+1}}{\sum_{t=0}^T \gamma_t} \right)} \leq \frac{2D + \widehat D \sum_{t=0}^T \gamma_t^2 }{\sum_{t=0}^T \gamma_t} . \label{hol-thm2-eq1}
    \end{align}
    For the step size $\gamma_t = \frac{c}{(t+1)^a}$, with $c>0$ and $0<a<1$, we use Lemma~\ref{lem-step-bd2}(i) to bound the denominator and Lemma~\ref{lem-step-bd2}(ii) with $p=0$ to bound the numerator on the right-hand side of relation~\eqref{hol-thm2-eq1}. Thus, we have
    \begin{align*}
        \EXP{G \left( \frac{\sum_{t=0}^T \gamma_t y_{t+1}}{\sum_{t=0}^T \gamma_t} \right)} &\leq
        \begin{cases}
             \frac{2D}{c(T+1)^{1-a}} + \frac{\widehat D c (T+1)^{1-2a} }{(1-2a) (T+1)^{1-a}}  \quad &\text{when $0<a<\frac{1}{2}$,}\\
             \frac{2D}{c(T+1)^{1-a}} + \frac{\widehat D c (1+\ln(T+1)) }{(T+1)^{1-a}}  \quad &\text{when $a=\frac{1}{2}$,}\\
             \frac{2D}{c(T+1)^{1-a}} + \frac{2a \widehat D c }{(2a-1) (T+1)^{1-a}}  \quad &\text{when $\frac{1}{2}<a<1$} .
        \end{cases} \nonumber\\
    \end{align*}
    %
    Simplifying the preceding relation yields the desired result of the theorem.
    \hfill$\square$
\end{proof}
%
%
%
%
    From Theorem~\ref{thm-hol-conv-unknown-itr} we can see that $a = \frac{1}{2}$ is not the best choice since the rate has an $\ln(T+1)$ term in the numerator which can deteriorate the convergence rate. Next, focusing on the case $0<a<\frac{1}{2}$, we see that the order of convergence is $O(1/T^a)$. Now, setting the value of $a = \frac{m-1}{2m}$, where $m>>1$ is a constant, then we get the rate as
    \begin{align}
        \EXP{G\left(\frac{\sum_{t=0}^T \gamma_t y_{t+1}}{\sum_{t=0}^T \gamma_t} \right)} \leq \frac{2D}{c(T+1)^{\frac{m+1}{2m}}} + \frac{ m \widehat D c}{(T+1)^{\frac{m-1}{2m}}} . \nonumber
    \end{align}
    If we choose for example $m=50$, then the rate is $O(1/T^{0.49})$ and with a larger values of $m$, we can get rate arbitrarily close to $O(T^{0.5})$. On the other hand, the second quantity in the preceding relation has a dependence on $m$ in the numerator so we should not choose extremely large $m$. Now, looking into the case when $\frac{1}{2}<a<1$, the convergence rate is of the order $O(1/T^{1-a})$, which is good for $a$ close to $1/2$. If we choose $a = \frac{m+1}{2m}$, we obtain 
    \begin{align}
        \EXP{G\left(\frac{\sum_{t=0}^T \gamma_t y_{t+1}}{\sum_{t=0}^T \gamma_t} \right)} \leq \frac{2D}{c(T+1)^{\frac{m-1}{2m}}} + \frac{m \widehat D c}{(T+1)^{\frac{m-1}{2m}}} . \nonumber
    \end{align}
    Similar rates to  the preceding case can be obtained with the selection of some $m>>1$, but not too large values of $m$.


However, an optimal convergence rate of $O(1/\sqrt{T+1})$ can be obtained using different weights in the averaging, as seen in the next subsection.

\subsection{Weighted-averages of the iterates using inverse step sizes as weights}
Here, we consider the weighted averages of the iterates (cf.~\eqref{avg-itr}) with the weights $\omega_\tau = \frac{1}{\gamma_\tau}$. In such averages,  we use a diminishing step size which gives larger weights to the more recent iterates than to those computed further in the past. 
Ultimately, this is beneficial to obtain the optimal convergence rate for the stochastic case, as seen in the following result.
\begin{theorem}\label{thm-inv-avg}
    Let Assumptions~\ref{asum-set1}--\ref{asum-stoc-bias-var} hold. Then, for the stochastic Popov mirror-prox algorithm with a diminishing step size $\gamma_t = \frac{c}{(t+1)^a}$, where $c>0$ and $0<a<1$, we have for any $T\ge1$,
    \begin{align}
        \EXP{G \left(\frac{\sum_{t=0}^T \gamma_t^{-1} y_{t+1}}{\sum_{t=0}^T \gamma_t^{-1}} \right)} \leq \frac{2 \times 4^{a} (1+a) D}{c T^{1-a}} + \frac{2 \widehat D c (1+a)}{T^a} , \nonumber
    \end{align}
    with the constant $\widehat D$ given in \eqref{D_hat} and $\max_{x,z \in X} B_\psi(z,x)\le D$.
\end{theorem}
\begin{proof}
    \noindent \textit{Case $\nu > 0$:} We use Lemma~\ref{lem-error-control}, where 
    we multiply both sides of the given relation by $\gamma_t^{-2}$ and obtain for any $t \geq 0$,
    \begin{align}
        &\gamma_t^{-1} \left\la F(z), y_{t+1} - z \right\ra \leq \gamma_t^{-2} B_{\psi}(z,x_t) - \gamma_t^{-2} B_{\psi}(z,x_{t+1}) + \gamma_t^{-2} B_\psi(z,h_{t}) \cr 
       & - \gamma_t^{-2} B_\psi(z,h_{t+1}) +\frac{2^{1+\nu} L_{\nu}^2}{\alpha^{1+\nu}} \left[\left(B_\psi(x_t,y_t) \right)^\nu + \left(B_\psi(y_{t+1},x_t) \right)^\nu \right] \cr
    & + \gamma_t^{-1} \left\la b_{t+1}, y_{t+1} - h_t \right\ra + \frac{4\epsilon_{t}^2 + \frac{9}{2}\epsilon_{t+1}^2}{\alpha} + \frac{8 M_{\nu}^2}{\alpha} . \nonumber
    \end{align}
    We bound the fifth quantity on the right-hand side of the preceding relation using Assumption~\ref{asum-set2} (compact set $X$), giving us $\left(B_\psi(x_t,y_t) \right)^\nu + \left(B_\psi(y_{t+1},x_t) \right)^\nu \leq 2D^\nu$, which when substituted in the preceding relation yields for all $t \geq 0$,
    \begin{align}
        &\gamma_t^{-1} \left\la F(z), y_{t+1} - z \right\ra \leq \gamma_t^{-2} B_{\psi}(z,x_t) - \gamma_t^{-2} B_{\psi}(z,x_{t+1}) + \gamma_t^{-2} B_\psi(z,h_{t}) \cr 
       & - \gamma_t^{-2} B_\psi(z,h_{t+1}) + \gamma_t^{-1} \left\la b_{t+1}, y_{t+1} - h_t \right\ra +\frac{2^{2+\nu} L_{\nu}^2 D^\nu}{\alpha^{1+\nu}} \cr
       & + \frac{4\epsilon_{t}^2 + \frac{9}{2}\epsilon_{t+1}^2}{\alpha} + \frac{8 M_{\nu}^2}{\alpha} . \label{inv-avg-basic-rel}
    \end{align}
    We add and subtract the quantities $\gamma_{t-1}^{-2} \EXP{B_{\psi}(z,x_t)}$ and $\gamma_{t-1}^{-2} \EXP{B_{\psi}(z,h_t)}$ on the right-hand side of relation~\eqref{inv-avg-basic-rel}, and obtain for any $z\in X$ and $t\ge1$,
    \begin{align}
        &\gamma_t^{-1} \left\la F(z), y_{t+1} - z \right\ra \leq \gamma_{t-1}^{-2} B_{\psi}(z,x_t) - \gamma_{t}^{-2} B_{\psi}(z,x_{t+1}) + \gamma_{t-1}^{-2} B_\psi(z,h_{t}) \cr 
       & - \gamma_{t}^{-2} B_\psi(z,h_{t+1}) + \left( \gamma_t^{-2} - \gamma_{t-1}^{-2} \right) \EXP{B_{\psi}(z,x_t)} + \left( \gamma_t^{-2} - \gamma_{t-1}^{-2} \right) \EXP{B_{\psi}(z,h_t)} \cr
       &+ \gamma_t^{-1} \left\la b_{t+1}, y_{t+1} - h_t \right\ra +\frac{2^{2+\nu} L_{\nu}^2 D^\nu}{\alpha^{1+\nu}} + \frac{4\epsilon_{t}^2 + \frac{9}{2}\epsilon_{t+1}^2}{\alpha} + \frac{8 M_{\nu}^2}{\alpha} . \nonumber
    \end{align}
    Since $\gamma_t <\gamma_{t-1}$, we have $\gamma_{t}^{-2} - \gamma_{t-1}^{-2} > 0$. Moreover, $\EXP{B_{\psi}(z,x_t)}\le D$ and $\EXP{B_{\psi}(z,h_t)}\le D$ by the boundedness of $X$ (Assumption~\ref{asum-set2}), implying that for any $z\in X$ and $t\ge1$,
    \begin{align*}
        &\gamma_t^{-1} \left\la F(z), y_{t+1} - z \right\ra \leq \gamma_{t-1}^{-2} B_{\psi}(z,x_t) - \gamma_{t}^{-2} B_{\psi}(z,x_{t+1}) + \gamma_{t-1}^{-2} B_\psi(z,h_{t}) \cr 
       & - \gamma_{t}^{-2} B_\psi(z,h_{t+1}) + 2 \left( \gamma_t^{-2} - \gamma_{t-1}^{-2} \right) D + \gamma_t^{-1} \left\la b_{t+1}, y_{t+1} - h_t \right\ra \cr
       &+\frac{2^{2+\nu} L_{\nu}^2 D^\nu}{\alpha^{1+\nu}} + \frac{4\epsilon_{t}^2 + \frac{9}{2}\epsilon_{t+1}^2}{\alpha} + \frac{8 M_{\nu}^2}{\alpha} .
    \end{align*}
    Summing the preceding relation from $t = 1$ to $T$ for any $T \geq 1$, we obtain
    \begin{align}
        &\sum_{t=1}^T \gamma_t^{-1} \left\la F(z), y_{t+1} - z \right\ra \leq \gamma_{0}^{-2} B_{\psi}(z,x_1) - \gamma_{T}^{-2} B_{\psi}(z,x_{T+1}) + \gamma_{0}^{-2} B_\psi(z,h_{1}) \cr 
       & - \gamma_{T}^{-2} B_\psi(z,h_{T+1}) + 2 \left( \gamma_T^{-2} - \gamma_{0}^{-2} \right) D + \sum_{t=1}^T \gamma_t^{-1} \left\la b_{t+1}, y_{t+1} - h_t \right\ra \cr
       &+ \left( \frac{2^{2+\nu} L_{\nu}^2 D^\nu}{\alpha^{1+\nu}} + \frac{4\epsilon_{t}^2 + \frac{9}{2}\epsilon_{t+1}^2}{\alpha} + \frac{8 M_{\nu}^2}{\alpha} \right) T . \label{inv-avg-basic-rel2}
    \end{align}
    Letting $t=0$ in relation~\eqref{inv-avg-basic-rel}, we obtain for $t =0$,
    \begin{align*}
        &\gamma_0^{-1} \left\la F(z), y_{1} - z \right\ra \leq \gamma_0^{-2} B_{\psi}(z,x_0) - \gamma_0^{-2} B_{\psi}(z,x_{1}) + \gamma_0^{-2} B_\psi(z,h_{0}) \cr 
       & - \gamma_0^{-2} B_\psi(z,h_{1}) + \gamma_0^{-1} \left\la b_{1}, y_{1} - h_0 \right\ra +\frac{2^{2+\nu} L_{\nu}^2 D^\nu}{\alpha^{1+\nu}} \cr
       & + \frac{4\epsilon_{t}^2 + \frac{9}{2}\epsilon_{t+1}^2}{\alpha} + \frac{8 M_{\nu}^2}{\alpha} .
    \end{align*}
    Summing the preceding relation and the relation in~\eqref{inv-avg-basic-rel2}, we obtain
    \begin{align}
        &\sum_{t=0}^T \gamma_t^{-1} \left\la F(z), y_{t+1} - z \right\ra \leq \gamma_{0}^{-2} B_{\psi}(z,x_0) - \gamma_{T}^{-2} B_{\psi}(z,x_{T+1}) + \gamma_{0}^{-2} B_\psi(z,h_0) \cr 
       & - \gamma_{T}^{-2} B_\psi(z,h_{T+1}) + 2 \left( \gamma_T^{-2} - \gamma_{0}^{-2} \right) D + \sum_{t=0}^T \gamma_t^{-1} \left\la b_{t+1}, y_{t+1} - h_t \right\ra \cr
       &+ \left( \frac{2^{2+\nu} L_{\nu}^2 D^\nu}{\alpha^{1+\nu}} + \frac{4\epsilon_{t}^2 + \frac{9}{2}\epsilon_{t+1}^2}{\alpha} + \frac{8 M_{\nu}^2}{\alpha} \right) (T+1) . \nonumber
    \end{align}
    Using the compactness of set $X$ (Assumption~\ref{asum-set2}) we estimate the first and third terms,  and dropping the non-positive terms on the right-hand side of the preceding relation, after simplifying the resulting relation, we obtain
    \begin{align}
        &\sum_{t=0}^T \gamma_t^{-1} \left\la F(z), y_{t+1} - z \right\ra \leq 2D \gamma_T^{-2} + \sum_{t=0}^T \gamma_t^{-1} \left\la b_{t+1}, y_{t+1} - h_t \right\ra \cr
       &+ \left( \frac{2^{2+\nu} L_{\nu}^2 D^\nu}{\alpha^{1+\nu}} + \frac{4\epsilon_{t}^2 + \frac{9}{2}\epsilon_{t+1}^2}{\alpha} + \frac{8 M_{\nu}^2}{\alpha} \right) (T+1) . \label{inv-avg-basic-rel3}
    \end{align}
    The left-hand side of the preceding relation can be simplified as
    \begin{align}
        \sum_{t=0}^T \gamma_t^{-1} \left\la F(z), y_{t+1} - z \right\ra = \left(\sum_{t=0}^T \gamma_t^{-1} \right) \left\la F(z), \frac{\sum_{t=0}^T \gamma_t^{-1} y_{t+1}}{\sum_{t=0}^T \gamma_t^{-1}} - z \right\ra . \nonumber
    \end{align}
    Using the preceding relation back in~\eqref{inv-avg-basic-rel3} and taking the  maximum over $z \in X$, we get a bound for the dual gap function value (cf.~\eqref{dg-gap}), which after taking the total expectation yields
    \begin{align}
       &\EXP{G \left(\frac{\sum_{t=0}^T \gamma_t^{-1} y_{t+1}}{\sum_{t=0}^T \gamma_t^{-1}} \right)} \leq \frac{1}{\left(\sum_{t=0}^T \gamma_t^{-1} \right)} \Bigg[2D \gamma_T^{-2} + \EXP{\sum_{t=0}^T \gamma_t^{-1} \!\left\la b_{t+1}, y_{t+1} \!-\! h_t \right\ra \!} \cr
       &+ \left( \frac{2^{2+\nu} L_{\nu}^2 D^\nu}{\alpha^{1+\nu}} + \frac{4 \EXP{\epsilon_{t}^2} + \frac{9}{2}\EXP{\epsilon_{t+1}^2}}{\alpha} + \frac{8 M_{\nu}^2}{\alpha} \right) (T+1) \Bigg] . \nonumber
    \end{align}
By law of iterated expectation, the second term on the right-hand side of the preceding relation is zero. Using Assumption~\ref{asum-stoc-bias-var} we have $\EXP{\epsilon_{t}^2} \leq \sigma^2$ and $\EXP{\epsilon_{t+1}^2} \leq \sigma^2$, thus leading to
\begin{align}
    \EXP{G \left(\frac{\sum_{t=0}^T \gamma_t^{-1} y_{t+1}}{\sum_{t=0}^T \gamma_t^{-1}} \right)} \leq \frac{2D \gamma_T^{-2} + \widehat D (T+1)}{\left(\sum_{t=0}^T \gamma_t^{-1} \right)} , \nonumber
\end{align}
where the quantity $\widehat D$ for $\nu>0$ is defined in Lemma~\ref{lem-avg} (cf. \eqref{D_hat}).
We take a diminishing step size $\gamma_t = \frac{c}{(t+1)^a}$ with some constant $c>0$ and $0<a<1$. Hence, $\frac{1}{\gamma_T^2} = \frac{(T+1)^{2a}}{c^2}$. Using these relations and Lemma~\ref{lem-step-bd2}(iv), we estimate the denominator on the right-hand side of the preceding relation and obtain
\begin{align}
    \EXP{G \left(\frac{\sum_{t=0}^T \gamma_t^{-1} y_{t+1}}{\sum_{t=0}^T \gamma_t^{-1}} \right)} \leq \frac{2D (1+a) \left(1+\frac{1}{T}\right)^{2a}}{c T^{1-a}} + \frac{\widehat D c (1+a) \left(1+\frac{1}{T}\right)}{T^{a}} . \nonumber
\end{align}
Since $T \geq 1$, we have  $1+\frac{1}{T} \leq 2$ and the stated relation follows.\\
\noindent \textit{Case $\nu = 0$:} We use Lemma~\ref{lem-error-control} with  $\nu = 0$ and follow the same line of analysis as for the case $\nu > 0$. \hfill$\square$
\end{proof}
Theorem~\ref{thm-inv-avg} provides the convergence rate of the dual gap function with respect to the number of iterates $T$, which depends on the terms $O(1/T^{1-a})$ and $O(1/T^a)$. The convergence rate of the first term increases when the constant $a$ decreases and vice versa for the second term. Next, we state a remark showing the optimal convergence rate and the step size selection to achieve it. 
\begin{remark}\label{rem-inv-avg}
    Letting $a = \frac{1}{2}$ in Theorem~\ref{thm-inv-avg}, the step size reduces to $
        \gamma_t = \frac{c}{\sqrt{t+1}}$ with $c > 0$. With such a step size selection, the convergence rate is
    \begin{align}
        \EXP{G \left(\frac{\sum_{t=0}^T \frac{y_{t+1}}{\gamma_t}}{\sum_{t=0}^T \frac{1}{\gamma_t}} \right)} \leq \frac{6 D}{c T^{\frac{1}{2}}} + \frac{3 \widehat D c}{T^{\frac{1}{2}}} , \nonumber
    \end{align}
    where $\max_{x,z \in X} B_\psi(z,x)\le D$ (cf. \eqref{bounded-set}) and the constant $\widehat D$ is given in relation~\eqref{D_hat}. Compared to \cite{klimza2024universal} and \cite{dang2015convergence}, which analyze the Korpelevich method, our algorithm employs the Popov method and addresses scenarios where the mapping $F$ satisfies relation~\eqref{mapping_growth}, yet maintains the same convergence rate $O(1/\sqrt{T})$ with parameter-free diminishing step sizes. Our analysis accommodates arbitrary growth conditions (including discontinuities) of the mapping $F$ over a compact set $X$. 
    \hfill$\square$
\end{remark}
The impact of the constant $c$ in the step size on the convergence rates is dicussed in the following remark.
\begin{remark}\label{rem-optimal-c}
The upper bounds on the convergence rate of the expected dual gap function in Theorems~\ref{thm-holder-known-itr}, \ref{thm-hol-conv-unknown-itr}, and \ref{thm-inv-avg} have the following form
\begin{align}
    \frac{f_1(D,a)}{c f_2(T,a)} + \frac{f_3(\sigma, M_\nu, L_\nu, \nu, a)\, c}{f_4(T,a)}, \label{exp_dg_up_bd}
\end{align}
where $f_1(\cdot), f_2(\cdot), f_3(\cdot)$, and $f_4(\cdot)$ are functions with the exact forms depending on the respective theorems. To obtain optimal value of $c$, we differentiate with respect to $c$ and set the derivative to zero, yielding
\begin{align}
    c^* = \sqrt{\frac{f_1(D,a)}{f_3(\sigma, M_\nu, L_\nu, \nu, a)} \cdot \frac{f_4(T,a)}{f_2(T,a)}}. \nonumber
\end{align}
While $f_1$, $f_2$, and $f_4$ are computable using the known quantities such as the step-size parameter $a$, the iteration number $T$, and the set $X$ diameter $D$, the function $f_3$ may be difficult to evaluate in practice due to possibly unknown parameters such as $\sigma$, $L_\nu$, $M_\nu$, and $\nu$. 

From~\eqref{exp_dg_up_bd}, we see that smaller values of $c$ worsen the first term but improve the second, while larger values have the opposite effect. Intuitively, when the noise parameter $\sigma$ is high, a smaller step size (i.e., a smaller $c$) is preferable. If all problem parameters are known, one can use the optimal $c^*$. Finally, we can choose $a$ so that $f_2(T,a) \geq f_4(T,a)$, which ensures $\frac{f_4(T,a)}{f_2(T,a)} \leq 1$, thus obtaining a bound for the the optimal value $c^*$.
\hfill$\square$
\end{remark}

\section{Deterministic Popov Mirror-Prox}\label{sec:deterministic_popov}

In this section, we study the deterministic Popov mirror-prox algorithm and its convergence properties. The method chooses an arbitrary initialization $x_0 = y_0 \in X$ and generates two iterate sequences, ${y_t}$ and ${x_t}$, similar to the stochastic case (cf.~\eqref{stoc-pov1}, \eqref{stoc-pov2}), with the stochastic mapping $\widehat F(\cdot, \xi)$ replaced by the deterministic mapping $F(\cdot)$, i.e.,
\begin{align}
    &y_{t+1} = P_{x_t}(\gamma_t F(y_t)) = \argmin_{z \in X} [\la \gamma_t F(y_{t}) - \nabla \psi(x_{t}), z \ra + \psi(z)] , \label{det-pov1}\\
    &x_{t+1} = P_{x_t}(\gamma_t F(y_{t+1}))  = \argmin_{z \in X} [\la \gamma_t F(y_{t+1}) - \nabla \psi(x_{t}), z \ra + \psi(z)] , \label{det-pov2}
\end{align}
where $\gamma_t > 0$ is the step size of the method and $P_x(\cdot): \R^n \rightarrow X$ is the proximal mapping defined as
\begin{align}
    P_x(\zeta) = \argmin_{z \in X} \{\psi(z) + \la z, \zeta - \nabla\psi(x) \ra\}, \;\; \forall x \in X \; \text{and } \zeta\in \R^n . \label{prox-mapping}
\end{align}


\subsection{Convergence of the dual gap function for mappings with $\nu \in [0,\infty)$}
We present the convergence rate of the deterministic Popov mirror-prox algorithm in~\eqref{det-pov1}--\eqref{det-pov2} with respect to the dual gap function, under the assumptions of a bounded set \( X \) (Assumption~\ref{asum-set2}) and a monotone mapping \( F \) (Assumption~\ref{asum-monotone}). The rate follows from the stochastic case by setting \( \sigma = 0 \), though some constants can be refined. In the deterministic setting, Lemma~\ref{lem-gen-stoc} simplifies as the terms involving \( w_1 \), \( b_t \), and \( \epsilon_t \) vanish for all \( t \geq 0 \). Since the derivations require only minor modifications, we state the result directly as a corollary.
\begin{corollary}\label{cor_det}
    Under Assumptions~\ref{asum-set1}--\ref{asum-set2}, with the constants $D$ and $\bar D$ given as
    \begin{align*}
        D\ge \max_{x,z \in X} B_\psi(z,x) \quad \text{and} \quad \bar D = \begin{cases}
        \frac{3 \times 2^{\nu} L_{\nu}^2 D^{\nu}}{\alpha^{1+\nu}} + \frac{6 M_{\nu}^2}{\alpha}  \quad &\text{when $\nu >0 $,}\\
         \frac{(L_0+M_0)^2}{2 \alpha}  \quad &\text{when $\nu = 0$}. 
        \end{cases} \nonumber
    \end{align*}
    for the iterates of the deterministic Popov mirror-prox, we have:\\
    (i) With known total number of iterations $T \geq 0$ and constant $c>0$,
    \begin{align*}
        &G \left( \frac{1}{T+1} \sum_{t=0}^T y_{t+1} \right) \leq \frac{D + c^2 \bar D}{c \sqrt{T+1}} \ \hbox{when $\gamma_t =\gamma=\frac{c}{\sqrt{T+1}}$ for all $t=0,\ldots T$}, \nonumber\\
        & G \left( \frac{\sum_{t= \lceil \frac{T}{2} \rceil}^T \gamma_t y_{t+1}}{\sum_{t= \lceil \frac{T}{2} \rceil}^T \gamma_t} \right) \leq \frac{2D + 2 \ln(4) c^2 \bar D}{c \sqrt{T+1}} \ \hbox{when $\gamma_t =\frac{c}{\sqrt{t+1}}$ for $t=0,\ldots T$}. \nonumber
    \end{align*}
    (ii) With diminishing step size $\gamma_t = \frac{c}{(t+1)^a}$, where constants $c>0$ and $0<a<1$, the averaged iterate with $\gamma_t$ as weights yields the following rate for all $T \geq 0$,
    \begin{align*}
        G \left( \frac{\sum_{t=0}^T \gamma_t y_{t+1}}{\sum_{t=0}^T \gamma_t} \right) \leq
        \begin{cases}
            \frac{D}{c(T+1)^{1-a}} + \frac{\bar D c }{(1-2a) (T+1)^a} \quad &\text{when $0<a<\frac{1}{2}$,}\\
            \frac{D}{c(T+1)^{1-a}} + \frac{\bar D c (1+\ln(T+1)) }{(T+1)^{1-a}}  \quad &\text{when $a=\frac{1}{2}$,}\\
            \frac{D}{c(T+1)^{1-a}} + \frac{2a \bar D c }{(2a-1) (T+1)^{1-a}}  \quad &\text{when $\frac{1}{2}<a<1$} .
        \end{cases} \nonumber
    \end{align*}
    (iii) Using the same step size $\gamma_t = \frac{c}{(t+1)^a}$ with constants $c>0$ and $0<a<1$, the averaged iterate with $\gamma_t^{-1}$ as weights yields the following rate for all $T \geq 1$,
    \begin{align*}
        G \left(\frac{\sum_{t=0}^T \gamma_t^{-1} y_{t+1}}{\sum_{t=0}^T \gamma_t^{-1}} \right) \leq \frac{4^a (1+a) D}{c T^{1-a}} + \frac{2 \bar D c (1+a)}{T^a} .
    \end{align*}
\end{corollary}
Corollary~\ref{cor_det} shows that the rate with respect to the dual gap function is at best \( O(1/\sqrt{T}) \) when \( M_\nu > 0 \) for any \( \nu > 0 \). This rate follows directly from Corollary~\ref{cor_det}(i) or from Corollary~\ref{cor_det}(iii) by setting \( a = 1/2 \) (cf. Remark~\ref{rem-inv-avg}), corresponding to a step size independent of problem parameters.

\subsection{Convergence rate for mappings with $M_\nu =0$ and $\nu \in (0,1]$}\label{sub-sec-special-case}

For this case, an improved convergence rate can be established with respect to the dual gap function. Additionally, the performance of the algorithm can be evaluated using the residual function as a  merit function, defined by~\cite{dang2015convergence}
\begin{align}
    R_\gamma(x) = \frac{1}{\gamma} [x - P_x (\gamma F(x))], \quad \forall x \in X , \label{residual}
\end{align}
where $\gamma > 0$ is  a step size and $P_x(\cdot)$ is the proximal function (cf.~\eqref{prox-mapping}). For example, when $\psi(\cdot) = \frac{1}{2}\|\cdot\|^2$, then the residual function is
\begin{align}
    R_\gamma(x) = \frac{1}{\gamma} [x- \Pi_X[x-\gamma F(x)]] , \quad \forall x \in X, \nonumber
\end{align}
where $\Pi_X[\cdot]$ denotes the projection operator on the set $X$. Note that $x=x^*$ is a solution of the VI$(X,F)$ (cf.~\eqref{strong-vi}) if and only if $x^* = P_x (\gamma F(x^*))$ \cite[Lemma 1]{dang2015convergence}, which further implies $\|R_\gamma(x^*)\| = 0$ \cite[Lemma 2]{dang2015convergence}. 

Next, we establish a lemma that refines Lemma~\ref{lem-gen-stoc} to be used later in the convergence rate analysis. To do so, we make use of the following auxiliary result.
\begin{lemma}\label{lem_max_val}
    Let $q \geq 0$, $s > 0$, and $\nu \in (0,1)$ be constants. 
    Then, for the function $D(d):=q d^\nu - s d$, we have
    \begin{align*}
        D(\bar d) \leq \max_{d \ge0} D(d) = (1-\nu) \left( \frac{q \nu^\nu}{s^\nu} \right)^{\frac{1}{1-\nu}} \qquad \text{for any $\bar d \ge0$}.
    \end{align*}
\end{lemma}
\begin{proof}
    We differentiate $D(d)$ with respect to $d$, and get the point of inflexion as
   \begin{align*}
        d = \left( \frac{q \nu}{s} \right)^{\frac{1}{1-\nu}} .
    \end{align*}
    Upon substituting the preceding value of $d$ in the function $D(d)$, after some algebraic simplifications the desired relation of the lemma is obtained.
    \hfill$\square$
\end{proof}

Next, we present the generic lemma that makes use of Lemma~\ref{lem_max_val}. While~\cite{dang2015convergence} provides useful intuition for the analysis, our analysis accounts for simplifying additional error terms.
\begin{lemma}\label{lem-det-special}
    Let Assumption~\ref{asum-set1} and  Assumption~\ref{asum-holder}, with $M_\nu = 0$, hold. Then, for 
    the deterministic Popov mirror-prox algorithm, we have for all $T \geq 0$,\\
        i) For $\nu \in (0,1)$, and with some constants $0\leq r<1$ and $w_5>1$, 
    \begin{align}
        &\sum_{t=0}^T \gamma_t \la F(y_{t+1}), y_{t+1} - z \ra + \frac{\alpha r}{2} \left(1-\frac{1}{w_5} \right) \sum_{t=0}^T \gamma_t^2 \|R_{\gamma_t}(x_t)\|^2 \leq B_{\psi}(z,x_0) \nonumber\\
        & + (1-\nu) \!\left[ \left( \frac{2^{2\nu - 1} L_{\nu}^2 \left(r(w_5-1) \alpha + 4 \right) \nu^\nu }{\alpha^{1+\nu}} \right)^{\!\!\frac{1}{1-\nu}} \!\!\!+ \left( \frac{2^{1+\nu} L_{\nu}^2 \nu^\nu}{\alpha^{1+\nu} (1-r)^\nu}   \right)^{\!\!\frac{1}{1-\nu}} \right] \!\sum_{t=0}^{T} \! \gamma_t^{\frac{2}{1-\nu}} ; \nonumber
    \end{align}
    ii) For $\nu = 1$ and $\gamma_t \leq \min \left( \frac{\alpha}{\sqrt{2} L_1 (\sqrt{r(w_5-1) \alpha + 4})} , \frac{\alpha \sqrt{1-r}}{2 L_1} \right)$,
    \begin{align}
        &\sum_{t=0}^T \gamma_t \la F(y_{t+1}), y_{t+1} - z \ra + \frac{\alpha r}{2} \left(1-\frac{1}{w_5} \right) \sum_{t=0}^T \gamma_t^2 \|R_{\gamma_t}(x_t)\|^2 \leq B_{\psi}(z,x_0) . \nonumber
    \end{align}
    iii) For $\nu \in (0,1)$, assuming Assumption~\ref{asum-set2}, $\gamma_t \leq \gamma_{t-1}$, and any $T \geq 1$,
    \begin{align}
        \sum_{t=0}^T \frac{1}{\gamma_t} \la F(y_{t+1}), y_{t+1} \!-\! z \ra \leq \frac{D}{\gamma_T^2} + (1-\nu) \left(2^{\frac{\nu}{1-\nu}} \!+\! 1 \right) \left(\frac{2^{1+\nu} L_\nu^2 \nu^\nu}{\alpha^{1+\nu}} \right)^{\frac{1}{1-\nu}} \sum_{t=0}^{T} \gamma_t^{\frac{2\nu}{1-\nu}} . \nonumber
    \end{align}
\end{lemma}
\begin{proof}
Following the analysis as in the proof of Lemma~\ref{lem-gen-stoc}, and using $\widehat F(z,\xi) = F(z)$ for all $z \in X$ (also implying $b_{t+1} = 0$), and $M_\nu = 0$, we obtain 
\begin{align}
        \gamma_t \la F(y_{t+1})&, y_{t+1} - z \ra \leq B_{\psi}(z,x_t) - B_{\psi}(z,x_{t+1}) \nonumber\\
       &- \left(1- \frac{w_2+w_3}{\alpha} \right) B_{\psi}(x_{t+1},y_{t+1}) +\frac{\gamma_t^2 L_{\nu}^2}{2 w_2} \left(\frac{2}{\alpha} B_\psi(x_t,y_t) \right)^\nu \nonumber\\
       &- B_{\psi}(y_{t+1},x_{t}) + \frac{\gamma_t^2 L_{\nu}^2}{2 w_3} \left(\frac{2}{\alpha} B_\psi(y_{t+1},x_t) \right)^\nu . \label{hol-special-case1}
    \end{align}
Adding $r B_\psi(y_{t+1},x_t)$ on both sides of relation~\eqref{hol-special-case1} with $0 \leq r < 1$ yields
    \begin{align}
        &\gamma_t \la F(y_{t+1}), y_{t+1} - z \ra + r B_\psi(y_{t+1},x_t) \leq B_{\psi}(z,x_t) - B_{\psi}(z,x_{t+1}) \nonumber\\
       & - \left(1- \frac{w_2+w_3}{\alpha} \right) B_{\psi}(x_{t+1},y_{t+1}) +\frac{\gamma_t^2 L_{\nu}^2}{2 w_2} \left(\frac{2}{\alpha} B_\psi(x_t,y_t) \right)^\nu \nonumber\\
       &- (1-r) B_{\psi}(y_{t+1},x_{t}) + \frac{\gamma_t^2 L_{\nu}^2}{2 w_3} \left(\frac{2}{\alpha} B_\psi(y_{t+1},x_t) \right)^\nu . \label{hol1}
    \end{align}
    The second quantity on the left-hand side of \eqref{hol1} can be lower estimated by using the $\alpha$-strong convexity of $B_\psi(y_{t+1},x_t)$ and the definition of $y_{t+1}$ (cf.~\eqref{det-pov1}):
    \begin{align}
        r B_\psi(y_{t+1},x_t) \geq \frac{\alpha r }{2} \|x_t - y_{t+1} \|^2 = \frac{\alpha r }{2} \|x_t - P_{x_t}(\gamma_t F(y_t)) \|^2, \nonumber
    \end{align}
    We add and subtract $P_{x_t}(\gamma_t F(x_t))$ inside the norm of the preceding relation and use the definition of residual function (cf. \eqref{residual}) to obtain
    \begin{align}
        r B_\psi(y_{t+1},x_t) &\geq \frac{\alpha r }{2} \|[x_t - P_{x_t}(\gamma_t F(x_t))] + [P_{x_t}(\gamma_t F(x_t)) - P_{x_t}(\gamma_t F(y_t))] \|^2 \nonumber\\
        & = \frac{\alpha r \gamma_t^2}{2} \|R_{\gamma_t}(x_t)\|^2 + \frac{\alpha r}{2} \|P_{x_t}(\gamma_t F(x_t)) - P_{x_t}(\gamma_t F(y_t))\|^2 \nonumber\\
        & \quad + \alpha r \la \gamma_t R_{\gamma_t}(x_t), P_{x_t}(\gamma_t F(x_t)) - P_{x_t}(\gamma_t F(y_t)) \ra . \label{hol2}
    \end{align}
    We lower estimate the last term on the right hand side of relation~\eqref{hol2} using Young's inequality, with a constant $w_5 > 1$, and obtain
    \begin{align}
        &\alpha r \la \gamma_t R_{\gamma_t}(x_t), P_{x_t}(\gamma_t F(x_t)) - P_{x_t}(\gamma_t F(y_t)) \ra \nonumber\\
        &\geq - \frac{\alpha r \gamma_t^2}{2 w_5} \|R_{\gamma_t}(x_t)\|^2 - \frac{\alpha r w_5}{2} \|P_{x_t}(\gamma_t F(x_t)) - P_{x_t}(\gamma_t F(y_t))\|^2 . \label{hol3}
    \end{align}
    Substituting relation~\eqref{hol3} back into relation~\eqref{hol2} and then combining the resulting relation with equation~\eqref{hol1} yields
    \begin{align}
        &\gamma_t \la F(y_{t+1}), y_{t+1} - z \ra + \frac{\alpha r \gamma_t^2}{2} \left(1-\frac{1}{w_5} \right) \|R_{\gamma_t}(x_t)\|^2 \leq B_{\psi}(z,x_t) - B_{\psi}(z,x_{t+1}) \nonumber\\
        &+ \frac{\alpha r (w_5 - 1)}{2} \|P_{x_t}(\gamma_t F(x_t)) \!-\! P_{x_t}(\gamma_t F(y_t))\|^2 - \left(1 \!-\! \frac{w_2+w_3}{\alpha} \right) B_{\psi}(x_{t+1},y_{t+1}) \nonumber\\
        &+\! \frac{\gamma_t^2 L_{\nu}^2}{2 w_2} \!\left(\frac{2}{\alpha} B_\psi(x_t,y_t) \! \right)^{\!\nu} \!- (1 \!-\! r) B_{\psi}(y_{t+1},x_{t}) \!+\! \frac{\gamma_t^2 L_{\nu}^2}{2 w_3} \!\left(\frac{2}{\alpha} B_\psi(y_{t+1},x_t) \! \right)^{\!\nu} . \label{hol4}
    \end{align}
    The third term on the right-hand side of relation~\eqref{hol4} can be upper estimated using the properties of the prox mapping from~\cite[Lemma 2.1]{nemirovski2004prox} which is
    \begin{align}
        \frac{\alpha r (w_5 - 1)}{2} &\|P_{x_t}(\gamma_t F(x_t)) - P_{x_t}(\gamma_t F(y_t))\|^2 \leq \frac{\alpha r (w_5 - 1) \gamma_t^2}{2 \alpha} \|F(x_t) - F(y_t)\|_*^2 \nonumber\\
        &\leq \frac{r (w_5 - 1) \gamma_t^2 L_\nu^2}{2} \|x_t - y_t\|^{2\nu} \leq \frac{ r (w_5 - 1) \gamma_t^2 L_\nu^2}{2} \left(\frac{2}{\alpha} B_\psi(x_t,y_t)\right)^{\nu} . \nonumber
    \end{align}
    Substituting the preceding relation into equation~\eqref{hol4} yields
    \begin{align}
        &\gamma_t \la F(y_{t+1}), y_{t+1} - z \ra + \frac{\alpha r \gamma_t^2}{2} \left(1-\frac{1}{w_5} \right) \|R_{\gamma_t}(x_t)\|^2 \leq B_{\psi}(z,x_t) - B_{\psi}(z,x_{t+1}) \nonumber\\
        &+\frac{\gamma_t^2 L_{\nu}^2 \left(r(w_5-1) + \frac{1}{w_2} \right)}{2^{1-\nu} \alpha^\nu} \left(B_\psi(x_t,y_t) \right)^\nu - \left(1- \frac{w_2+w_3}{\alpha} \right) B_{\psi}(x_{t+1},y_{t+1}) \nonumber\\
        & + \frac{\gamma_t^2 L_{\nu}^2}{2^{1-\nu} \alpha^\nu w_3} \left(B_\psi(y_{t+1},x_t) \right)^\nu  - (1-r) B_{\psi}(y_{t+1},x_{t}) . \nonumber
    \end{align}
    We select the constants $w_2 = w_3 = \frac{\alpha}{4}$ in the preceding relation, yielding
    \begin{align}
        &\gamma_t \la F(y_{t+1}), y_{t+1} - z \ra + \frac{\alpha r \gamma_t^2}{2} \left(1-\frac{1}{w_5} \right) \|R_{\gamma_t}(x_t)\|^2 \leq B_{\psi}(z,x_t) - B_{\psi}(z,x_{t+1}) \nonumber\\
        &+\frac{\gamma_t^2 L_{\nu}^2 \left(r(w_5-1) + \frac{4}{\alpha} \right)}{2^{1-\nu} \alpha^\nu} \left(B_\psi(x_t,y_t) \right)^\nu - \frac{1}{2} B_{\psi}(x_{t+1},y_{t+1}) \nonumber\\
        & + \frac{2^{1+\nu} \gamma_t^2 L_{\nu}^2}{\alpha^{1+\nu}} \left(B_\psi(y_{t+1},x_t) \right)^\nu  - (1-r) B_{\psi}(y_{t+1},x_{t}) . \nonumber
    \end{align}
    Summing the preceding relation from $t = 0,1, \ldots, T$ for any $T \geq 0$ yields 
    \begin{align}
        &\sum_{t=0}^T \gamma_t \la F(y_{t+1}), y_{t+1} - z \ra + \frac{\alpha r}{2} \left(1-\frac{1}{w_5} \right) \sum_{t=0}^T \gamma_t^2 \|R_{\gamma_t}(x_t)\|^2 \leq B_{\psi}(z,x_0) \nonumber\\
        &- B_{\psi}(z,x_{T+1}) + \frac{\gamma_0^2 L_{\nu}^2 \left(r(w_5-1) + \frac{4}{\alpha} \right)}{2^{1-\nu} \alpha^\nu} \left(B_\psi(x_0,y_0) \right)^\nu - \frac{1}{2} B_{\psi}(x_{T+1},y_{T+1}) \nonumber\\
        &+\sum_{t=0}^{T-1} \left[\frac{\gamma_{t+1}^2 L_{\nu}^2 \left(r(w_5-1) + \frac{4}{\alpha} \right)}{2^{1-\nu} \alpha^\nu} \left(B_\psi(x_{t+1},y_{t+1}) \right)^\nu - \frac{1}{2} B_{\psi}(x_{t+1},y_{t+1}) \right] \nonumber\\
        & + \sum_{t=0}^T \left[ \frac{2^{1+\nu} \gamma_t^2 L_{\nu}^2}{\alpha^{1+\nu}} \left(B_\psi(y_{t+1},x_t) \right)^\nu  - (1-r) B_{\psi}(y_{t+1},x_{t}) \right] . \nonumber
    \end{align}
    With an initialization $x_0 = y_0$, the third quantity on the right-hand side of the preceding relation can be dropped. In addition, we drop the second and the fourth quantities on the right-hand side. Thus, we obtain
    \begin{align}
        &\sum_{t=0}^T \gamma_t \la F(y_{t+1}), y_{t+1} - z \ra + \frac{\alpha r}{2} \left(1-\frac{1}{w_5} \right) \sum_{t=0}^T \gamma_t^2 \|R_{\gamma_t}(x_t)\|^2 \leq B_{\psi}(z,x_0) \nonumber\\
        &+\sum_{t=0}^{T-1} \left[ \underbrace{\frac{\gamma_{t+1}^2 L_{\nu}^2 \left(r(w_5-1) + \frac{4}{\alpha} \right)}{2^{1-\nu} \alpha^\nu} \left(B_\psi(x_{t+1},y_{t+1}) \right)^\nu - \frac{1}{2} B_{\psi}(x_{t+1},y_{t+1}) }_{D_{t,1}(B_\psi(x_{t+1},y_{t+1}))} \right] \nonumber\\
        & + \sum_{t=0}^T \left[ \underbrace{\frac{2^{1+\nu} \gamma_t^2 L_{\nu}^2}{\alpha^{1+\nu}} \left(B_\psi(y_{t+1},x_t) \right)^\nu  - (1-r) B_{\psi}(y_{t+1},x_{t}) }_{D_{t,2}(B_\psi(y_{t+1},x_{t}))} \right] . \label{hol6}
    \end{align}

    \noindent\textit{Case (i):} Now consider $\nu \in (0,1)$. Using Lemma~\ref{lem_max_val}, we obtain global upper estimates of both \( D_{t,1}(B_\psi(x_{t+1}, y_{t+1})) \) and \( D_{t,2}(B_\psi(y_{t+1}, x_t)) \). For the quantity \( D_{t,1}(B_\psi(x_{t+1}, y_{t+1})) \), we equate $d = B_\psi(x_{t+1}, y_{t+1})$, $q = \frac{\gamma_{t+1}^2 L_{\nu}^2 \left(r(w_5-1) + \frac{4}{\alpha} \right)}{2^{1-\nu} \alpha^\nu}$, and $s = \frac{1}{2}$ in Lemma~\ref{lem_max_val}. Similarly for the quantity \( D_{t,2}(B_\psi(y_{t+1}, x_t)) \), we equate $d=B_\psi(y_{t+1}, x_t)$, $q= \frac{2^{1+\nu} \gamma_t^2 L_{\nu}^2}{\alpha^{1+\nu}}$, and $s = (1-r)$. Hence, for $0 \leq t \leq T$,
    \begin{align*}
        &D_{t,1}(B_\psi(x_{t+1},y_{t+1})) \leq (1-\nu) \left( \frac{2^{2\nu - 1} \gamma_{t+1}^2 L_{\nu}^2 \left(r(w_5-1) \alpha + 4 \right) \nu^\nu }{\alpha^{1+\nu}} \right)^{\frac{1}{1-\nu}}, \\
        &D_{t,2}(B_\psi(y_{t+1},x_{t})) \leq (1-\nu) \left( \frac{2^{1+\nu} \gamma_t^2 L_{\nu}^2 \nu^\nu}{\alpha^{1+\nu} (1-r)^\nu}   \right)^{\frac{1}{1-\nu}} .
    \end{align*}
    Substituting the preceding two relations into equation~\eqref{hol6} and applying $\sum_{t=0}^{T-1} \gamma_{t+1}^{\frac{2}{1-\nu}} \leq \sum_{t=0}^{T} \gamma_{t}^{\frac{2}{1-\nu}}$ yields result (i) of the lemma.

    \noindent\textit{Case (ii):} Next, we consider $\nu = 1$ which modifies relation~\eqref{hol6} as
    \begin{align}
        &\sum_{t=0}^T \gamma_t \la F(y_{t+1}), y_{t+1} - z \ra + \frac{\alpha r}{2} \left(1-\frac{1}{w_5} \right) \sum_{t=0}^T \gamma_t^2 \|R_{\gamma_t}(x_t)\|^2 \leq B_{\psi}(z,x_0) \nonumber\\
        &- \sum_{t=0}^{T-1} \left(\frac{1}{2} - \frac{\gamma_{t+1}^2 L_{1}^2 \left(r(w_5-1) \alpha + 4 \right)}{\alpha^2} \right) B_{\psi}(x_{t+1},y_{t+1}) \nonumber\\
        &- \sum_{t=0}^T \left((1-r) - \frac{4 \gamma_t^2 L_{1}^2}{\alpha^{2}} \right) B_{\psi}(y_{t+1},x_{t}) . \label{hol7}
    \end{align}
    To make the last two terms of the preceding relation non-positive, we select
    \begin{align}
        \gamma_t \leq \min \left( \frac{\alpha}{\sqrt{2} L_1 (\sqrt{r(w_5-1) \alpha + 4})} , \frac{\alpha \sqrt{1-r}}{2 L_1} \right) , \label{step-up-bd-lipschitz}
    \end{align}
    with $0 \leq r <1$ and $w_5 > 1$. 
    With the step size rule as in~\eqref{step-up-bd-lipschitz}, the second and the third quantities on the right-hand side of the relation~\eqref{hol7} can be dropped and we obtain the result (ii) of the lemma.

    \noindent\textit{Case (iii):} We start from relation~\eqref{hol-special-case1} and select $w_2 = w_3 = \frac{\alpha}{4}$ and then divide by $\gamma_t^2$ on both sides and obtain
    \begin{align}
        &\frac{1}{\gamma_t} \la F(y_{t+1}), y_{t+1} - z \ra \leq \frac{B_{\psi}(z,x_t)}{\gamma_t^2} - \frac{B_{\psi}(z,x_{t+1})}{\gamma_t^2} - \frac{1}{2 \gamma_t^2} B_{\psi}(x_{t+1},y_{t+1}) \nonumber\\
       &+\frac{2^{1+\nu} L_{\nu}^2}{\alpha^{1+\nu}} \left(B_\psi(x_t,y_t) \right)^\nu - \frac{B_{\psi}(y_{t+1},x_{t})}{\gamma_t^2} + \frac{2^{1+\nu} L_{\nu}^2}{\alpha^{1+\nu}} \left(B_\psi(y_{t+1},x_t) \right)^\nu . \label{inv1}
    \end{align}
    Then, we follow the lines of proof similar to that of Theorem~\ref{thm-inv-avg} by adding and subtracting $\frac{1}{\gamma_{t-1}^2} B_{\psi}(z,x_t)$ on the right-hand side of preceding relation and use non-increasing step sizes, i.e., $\gamma_t \leq \gamma_{t-1}$ and Assumption~\ref{asum-set2} to obtain
    \begin{align}
        \frac{1}{\gamma_t} \la F(y_{t+1}), y_{t+1} - z \ra &\leq \frac{B_{\psi}(z,x_t)}{\gamma_{t-1}^2} - \frac{B_{\psi}(z,x_{t+1})}{\gamma_t^2} + \left( \frac{1}{\gamma_t^2} - \frac{1}{\gamma_{t-1}^2} \right) D \nonumber\\
       & + \frac{2^{1+\nu} L_{\nu}^2}{\alpha^{1+\nu}} \left(B_\psi(x_t,y_t) \right)^\nu - \frac{B_{\psi}(x_{t+1},y_{t+1})}{2 \gamma_t^2} \nonumber\\
       & + \frac{2^{1+\nu} L_{\nu}^2}{\alpha^{1+\nu}} \left(B_\psi(y_{t+1},x_t) \right)^\nu - \frac{B_{\psi}(y_{t+1},x_{t})}{\gamma_t^2} . \nonumber
    \end{align}
    Summing the preceding relation from $t=1, \ldots, T$ for any $T \geq 1$ yields
    \begin{align}
        \sum_{t=1}^T \frac{1}{\gamma_t} \la F(y_{t+1})&, y_{t+1} - z \ra \leq \frac{1}{\gamma_0^2} B_{\psi}(z,x_1) - \frac{1}{\gamma_T^2} B_{\psi}(z,x_{T+1}) + \left( \frac{1}{\gamma_T^2} - \frac{1}{\gamma_{0}^2} \right) D \nonumber\\
       & + \sum_{t=1}^T \left[\frac{2^{1+\nu} L_{\nu}^2}{\alpha^{1+\nu}} \left(B_\psi(x_t,y_t) \right)^\nu - \frac{B_{\psi}(x_{t+1},y_{t+1})}{2 \gamma_t^2} \right] \nonumber\\
       &+ \sum_{t=1}^T \left[\frac{2^{1+\nu} L_{\nu}^2}{\alpha^{1+\nu}} \left(B_\psi(y_{t+1},x_t) \right)^\nu - \frac{B_{\psi}(y_{t+1},x_{t})}{\gamma_t^2} \right] . \label{inv2}
    \end{align}
    Moreover, using $t=0$ in relation~\eqref{inv1} yields
    \begin{align}
        &\frac{1}{\gamma_0} \la F(y_{1}), y_{1} - z \ra \leq \frac{B_{\psi}(z,x_0)}{\gamma_0^2} - \frac{B_{\psi}(z,x_{1})}{\gamma_0^2} - \frac{1}{2 \gamma_0^2} B_{\psi}(x_{1},y_{1}) \nonumber\\
       &+\frac{2^{1+\nu} L_{\nu}^2}{\alpha^{1+\nu}} \left(B_\psi(x_0,y_0) \right)^\nu - \frac{B_{\psi}(y_{1},x_{0})}{\gamma_0^2} + \frac{2^{1+\nu} L_{\nu}^2}{\alpha^{1+\nu}} \left(B_\psi(y_{1},x_0) \right)^\nu.  \label{inv3}
    \end{align}
    Adding relations~\eqref{inv2} and \eqref{inv3}, we obtain
    \begin{align}
        &\sum_{t=0}^T \frac{1}{\gamma_t} \la F(y_{t+1}), y_{t+1} - z \ra \leq \frac{B_{\psi}(z,x_0)}{\gamma_0^2} - \frac{B_{\psi}(z,x_{T+1})}{\gamma_T^2} + \left( \frac{1}{\gamma_T^2} - \frac{1}{\gamma_{0}^2} \right) D \nonumber\\
        & + \sum_{t=0}^{T-1} \left[\frac{2^{1+\nu} L_{\nu}^2}{\alpha^{1+\nu}} \left(B_\psi(x_{t+1},y_{t+1}) \right)^\nu - \frac{B_{\psi}(x_{t+1},y_{t+1})}{2 \gamma_t^2} \right] - \frac{B_{\psi}(x_{T+1},y_{T+1})}{2 \gamma_T^2} \nonumber\\
       & +\frac{2^{1+\nu} L_{\nu}^2 \left(B_\psi(x_0,y_0) \right)^\nu}{\alpha^{1+\nu}}  + \sum_{t=0}^T \left[\frac{2^{1+\nu} L_{\nu}^2}{\alpha^{1+\nu}} \left(B_\psi(y_{t+1},x_t) \right)^\nu - \frac{B_{\psi}(y_{t+1},x_{t})}{\gamma_t^2} \right] . \nonumber
    \end{align}
    The second and fifth terms on the right-hand side of the preceding relation can be dropped. Since we initialize with \( x_0 = y_0 \), the sixth term also vanishes. Applying Assumption~\ref{asum-set2} to the first term and upper estimating the sum of the fourth term from \( t = 0 \) to \( T \) on the right hand side, we obtain
    \begin{align}
        &\sum_{t=0}^T \frac{1}{\gamma_t} \la F(y_{t+1}), y_{t+1} - z \ra \nonumber\\
        &\leq \frac{D}{\gamma_T^2} + \sum_{t=0}^{T} \left[\frac{2^{1+\nu} L_{\nu}^2}{\alpha^{1+\nu}} \left(B_\psi(x_{t+1},y_{t+1}) \right)^\nu - \frac{B_{\psi}(x_{t+1},y_{t+1})}{2 \gamma_t^2} \right] \nonumber\\
       &\quad + \sum_{t=0}^T \left[\frac{2^{1+\nu} L_{\nu}^2}{\alpha^{1+\nu}} \left(B_\psi(y_{t+1},x_t) \right)^\nu - \frac{B_{\psi}(y_{t+1},x_{t})}{\gamma_t^2} \right] . \nonumber
    \end{align}
    We upper estimate the last two terms using Lemma~\ref{lem_max_val}, where for the second term on the right hand side, we equate $d = B_\psi(x_{t+1},y_{t+1})$, $q = \frac{2^{1+\nu} L_{\nu}^2}{\alpha^{1+\nu}}$, and $s = \frac{1}{2 \gamma_t^2}$, whereas for the last term, we equate $d = B_\psi(y_{t+1},x_t)$, $q = \frac{2^{1+\nu} L_{\nu}^2}{\alpha^{1+\nu}}$, and $s = \frac{1}{\gamma_t^2}$, yielding the result (iii) of the lemma.
 \hfill$\square$
\end{proof}
Lemma~\ref{lem-det-special} addresses the special case where the mapping is H\"older continuous. Thus far, we have not invoked the monotonicity of $F$ (Assumption~\ref{asum-monotone}). Moreover, parts (i) and (ii) of Lemma~\ref{lem-det-special} do not rely on the boundedness of the set $X$ (Assumption~\ref{asum-set2}). The result in~\cite[Theorem 1]{chakraborty2024popov} establishes an optimal convergence rate of $O(1/T)$ in terms of the dual gap function using a constant step size $\gamma_t = \gamma = \frac{\alpha}{2L}$ under Lipschitz continuity, i.e., $\nu = 1$. We therefore focus next on the convergence analysis for the dual gap function when $\nu \in (0,1)$.
\begin{theorem}\label{thm-det-special-case1}
    Under Assumptions~\ref{asum-set1}--\ref{asum-set2}, with $M_\nu = 0$, $\nu \in (0,1)$, and constants $\widehat C = (1-\nu) \left[ \left( \frac{2^{1+2\nu} L_{\nu}^2 \nu^\nu }{\alpha^{1+\nu}} \right)^{\frac{1}{1-\nu}} + \left( \frac{2^{1+\nu} L_{\nu}^2 \nu^\nu}{\alpha^{1+\nu}}   \right)^{\frac{1}{1-\nu}} \right]$, $c>0$, and $0<a<1$, for the iterates of the deterministic Popov mirror-prox, we have:\\
    i) With a constant step size $\gamma_t = \frac{c}{(T+1)^a}$ and known number $T\geq 0$ of iterations, the following relation holds
    \begin{align}
        & G \left(\frac{1}{T+1} \sum_{t=0}^T y_{t+1} \right) \leq \frac{D}{c (T+1)^{1-a}} + \frac{\widehat C c^{\frac{1+\nu}{1-\nu}}}{(T+1)^{\frac{a(1+\nu)}{1-\nu}}} . \nonumber
    \end{align}
    ii) If the number $T$ of iterations is not known in advance, then with the  step size $\gamma_t = \frac{c}{(t+1)^a}$, $c>0$, the following relation holds for any $T \geq 0$,
   \begin{align}
        G \left(\frac{\sum_{t=0}^T \gamma_t y_{t+1}}{\sum_{t=0}^T \gamma_t} \right) \leq \begin{cases}
             \frac{D}{c(T+1)^{1-a}} + \frac{\widehat C (1-\nu) c^{\frac{1+\nu}{1-\nu}} }{(1-\nu-2a) (T+1)^{\frac{a(1+\nu)}{1-\nu}}} \quad &\text{for $0<a<\frac{1-\nu}{2}$,}\\
            \frac{D}{c(T+1)^{1-a}} + \frac{\widehat C c^{\frac{1+\nu}{1-\nu}} (1+\ln (T+1)) }{(T+1)^{1-a}}  \quad &\text{for $a=\frac{1-\nu}{2}$,}\\
            \frac{D}{c(T+1)^{1-a}} + \frac{ 2a \widehat C c^{\frac{1+\nu}{1-\nu}}}{(2a+\nu-1) (T+1)^{1-a}}  \quad &\text{for $\frac{1-\nu}{2}<a<1$} .
        \end{cases} \nonumber
    \end{align}
    iii) With $\gamma_t = \frac{c}{(t+1)^a}$, $c>0$, the averaging of iterates with the inverse of the step sizes as weights, and an unknown number $T\ge 1$ of iterations, there holds
    \begin{align}
        &G \left(\frac{\sum_{t=0}^T \frac{y_{t+1}}{\gamma_t}}{\sum_{t=0}^T \frac{1}{\gamma_t}} \right) \leq \frac{4^a (1+a) D}{c T^{1-a}} \nonumber\\
        &+ \begin{cases}
            \frac{2^{\frac{1-\nu-2\nu a}{1-\nu}}  {c^{\frac{1+\nu}{1-\nu}} (1+a)} (1-\nu)^2 \left(2^{\frac{\nu}{1-\nu}} + 1 \right) \left(\frac{2^{1+\nu} L_\nu^2 \nu^\nu}{\alpha^{1+\nu}} \right)^{\frac{1}{1-\nu}} }{(1-\nu-2\nu a) T^{\frac{a(1+\nu)}{1-\nu}}}  \quad &\text{for } 0<a<\frac{1-\nu}{2 \nu},\\
            \frac{{c^{\frac{1+\nu}{1-\nu}}(1+a)} (1-\nu) \left(2^{\frac{\nu}{1-\nu}} + 1 \right) \left(\frac{2^{1+\nu} L_\nu^2 \nu^\nu}{\alpha^{1+\nu}} \right)^{\frac{1}{1-\nu}} (1+\ln(T+1)) }{T^{1+a}} \quad &\text{for }  a=\frac{1-\nu}{2 \nu},\\
            \frac{ 2a\nu {c^{\frac{1+\nu}{1-\nu}} (1+a)} (1-\nu) \left(2^{\frac{\nu}{1-\nu}} + 1 \right) \left(\frac{2^{1+\nu} L_\nu^2 \nu^\nu}{\alpha^{1+\nu}} \right)^{\frac{1}{1-\nu}} }{(2\nu a+\nu-1) T^{1+a}} \quad &\text{for } \frac{1-\nu}{2 \nu}<a<1 .
        \end{cases} \nonumber
    \end{align}
\end{theorem}
\begin{proof}
    Substitute $r = 0$ in the relation of Lemma~\ref{lem-det-special}(i) and then apply the monotonicity of the mapping $F$ (cf. Assumption~\ref{asum-monotone}) to obtain
    \begin{align*}
        \sum_{t=0}^T \gamma_t \la F(z), y_{t+1} - z \ra \leq B_{\psi}(z,x_0) + \widehat C \sum_{t=0}^{T} \gamma_t^{\frac{2}{1-\nu}} ,
    \end{align*}
    where the constant $\widehat C$ is defined in the theorem. The quantity on the left-hand side of the preceding relation can be written as
\begin{align}\label{avg_style}
        \sum_{t=0}^T \gamma_t \la F(z), y_{t+1} - z \ra = \left(\sum_{t=0}^T \gamma_t \right) \left\la F(z), \frac{\sum_{t=0}^T \gamma_t y_{t+1}}{\sum_{t=0}^T \gamma_t} - z \right\ra . 
    \end{align}
    Hence, by combining the preceding two relations, taking the maximum over $z \in X$, and using the compactness of the set $X$ (Assumption~\ref{asum-set2}), we arrive at 
    \begin{align}
        & G \left(\frac{\sum_{t=0}^T \gamma_t y_{t+1}}{\sum_{t=0}^T \gamma_t} \right) \leq \frac{D}{\sum_{t=0}^T \gamma_t} + \frac{\widehat C \sum_{t=0}^{T} \gamma_t^{\frac{2}{1-\nu}}}{\sum_{t=0}^T \gamma_t} . \label{hol10}
    \end{align}
    \noindent
    \textit{Case (i):} We choose a constant step size $\gamma_t = \frac{c}{(T+1)^a}$, with $a<1$ and $c>0$, in the relation~\eqref{hol10} to obtain the desired relation of the theorem.\\
\noindent
    {\it Case (ii):} To obtain this relation, we use $\gamma_t = \frac{c}{(t+1)^a}$ in relation~\eqref{hol10}, with $c>0$ and $a<1$, and then use the results of Lemma~\ref{lem-step-bd2}(i) and (ii) with $p=\nu$.\\
\noindent
    \textit{Case (iii):} We start from Lemma~\ref{lem-det-special} (iii), where we lower estimate the term on its left-hand side using the monotonicity of mapping $F$ (Assumption~\ref{asum-monotone}), and then average the iterates with $\gamma_t^{-1}$ similar to equation~\eqref{avg_style}. Then, we take the maximum over $z \in X$ to obtain the following relation for the dual gap function
    \begin{align}
        G \left(\frac{\sum_{t=0}^T \frac{y_{t+1}}{\gamma_t}}{\sum_{t=0}^T \frac{1}{\gamma_t}} \right) \leq \frac{\frac{D}{\gamma_T^2} + (1-\nu) \left(2^{\frac{\nu}{1-\nu}} + 1 \right) \left(\frac{2^{1+\nu} L_\nu^2 \nu^\nu}{\alpha^{1+\nu}} \right)^{\frac{1}{1-\nu}} \sum_{t=0}^{T} \gamma_t^{\frac{2\nu}{1-\nu}} }{\sum_{t=0}^T \frac{1}{\gamma_t}} . \nonumber
    \end{align}
    With the step size $\gamma_t = \frac{c}{(t+1)^a}$, with $c>0$ and $a<1$, and using Lemma~\ref{lem-step-bd2}(iii) and (iv), with $p=\nu$, from  the preceding relation, we obtain
    \begin{align}
        &G \left(\frac{\sum_{t=0}^T \frac{y_{t+1}}{\gamma_t}}{\sum_{t=0}^T \frac{1}{\gamma_t}} \right) \leq \frac{(1+a) D (1+\frac{1}{T})^{2a}}{c T^{1-a}} \nonumber\\
        &+ \begin{cases}
            \frac{{c^{\frac{1+\nu}{1-\nu}} (1+a)} (1-\nu)^2 \left(2^{\frac{\nu}{1-\nu}} + 1 \right) \left(\frac{2^{1+\nu} L_\nu^2 \nu^\nu}{\alpha^{1+\nu}} \right)^{\frac{1}{1-\nu}} \left(1+\frac{1}{T} \right)^{\frac{1-\nu-2\nu a}{1-\nu}} }{(1-\nu-2\nu a) T^{\frac{a(1+\nu)}{1-\nu}}} \quad &\text{for }  0<a<\frac{1-\nu}{2 \nu},\\
            \frac{{c^{\frac{1+\nu}{1-\nu}}(1+a)} (1-\nu) \left(2^{\frac{\nu}{1-\nu}} + 1 \right) \left(\frac{2^{1+\nu} L_\nu^2 \nu^\nu}{\alpha^{1+\nu}} \right)^{\frac{1}{1-\nu}} (1+\ln(T+1)) }{T^{1+a}} \quad &\text{for }  a=\frac{1-\nu}{2 \nu},\\
            \frac{ 2a\nu {c^{\frac{1+\nu}{1-\nu}} (1+a)} (1-\nu) \left(2^{\frac{\nu}{1-\nu}} + 1 \right) \left(\frac{2^{1+\nu} L_\nu^2 \nu^\nu}{\alpha^{1+\nu}} \right)^{\frac{1}{1-\nu}} }{(2\nu a+\nu-1) T^{1+a}} \quad &\text{for }  \frac{1-\nu}{2 \nu}<a<1 .
        \end{cases} \nonumber
    \end{align}
    Since, $T \geq 1$, we obtain $1+\frac{1}{T} \leq 2$. Using these in the preceding equation gives the stated relation. \hfill$\square$
\end{proof}
Theorem~\ref{thm-det-special-case1}(i) shows that for a monotone and continuous mapping with $\nu \in (0,1)$, a constant step size $\gamma_t = \gamma = {c}/{(T+1)^{\frac{1-\nu}{2}}}$ yields a convergence rate of $O(1/T^{\frac{1+\nu}{2}})$ in terms of the dual gap function. Moreover, if $\widehat C$ is sufficiently small, the convergence can be accelerated by choosing a small $a \ll 1$, and vice versa. Notably, selecting a function $\psi(\cdot)$ with a larger strong convexity constant $\alpha$ can help reduce $\widehat C$. Similar observations apply to Theorem~\ref{thm-det-special-case1}(ii) and (iii). In Theorem~\ref{thm-det-special-case1}(ii), setting $a = \frac{1 - \nu}{2}$ introduces a $\log(T+1)$ dependency in the rate. In contrast, part (iii) avoids this by using an averaging scheme with $\gamma_t^{-1}$ as weights and $a = \frac{1 - \nu}{2} < \frac{1 - \nu}{2\nu}$, still achieving the $O(1/T^{\frac{1+\nu}{2}})$ rate. However, in all cases, the step size remains dependent on the H\"older exponent $\nu$.

In Theorem~\ref{thm-det-special-case1}, the monotonicity of the mapping $F$ and the boundedness of the set $X$ are used. We will now establish the convergence rate of our method without requiring these assumptions, provided a Minty solution exists (cf.~\eqref{minty-vi}).
%
%

\begin{theorem}\label{thm-residual}
    Let Assumption~\ref{asum-set1} and Assumption~\ref{asum-holder}, for $\nu \in (0,1]$ and $M_\nu = 0$, hold. Let a Minty solution $x^*$ of VI$(X,F)$ exists. For the deterministic Popov mirror-prox algorithm we have,  with constants $0< r < 1$ and $w_5 > 1$: \\
    i) For $\nu \in (0,1)$, when the number of iterations $T \geq 0$ is known in advance, then with $\gamma_t = \frac{c}{(T+1)^a}$ for any constant $a < \frac{1}{2}$ and $c>0$, we obtain
    \begin{align}
        &\min_{0 \leq t \leq T} \|R_{\gamma_t}(x_t)\|^2 \leq \frac{B_{\psi}(x^*,x_0)}{\frac{\alpha r}{2} \left(1-\frac{1}{w_5} \right) c^2 (T+1)^{1-2a}} + \frac{\widetilde C c^{\frac{2\nu}{1-\nu}}}{\frac{\alpha r}{2} \left(1-\frac{1}{w_5} \right) (T+1)^{\frac{2\nu a}{1-\nu}}} , \nonumber
    \end{align}
    where  $\widetilde C = (1-\nu) \left[ \left( \frac{2^{2\nu - 1} L_{\nu}^2 \left(r(w_5-1) \alpha + 4 \right) \nu^\nu }{\alpha^{1+\nu}} \right)^{\frac{1}{1-\nu}} + \left( \frac{2^{1+\nu} L_{\nu}^2 \nu^\nu}{\alpha^{1+\nu} (1-r)^\nu}   \right)^{\frac{1}{1-\nu}} \right]$. On the other hand, when the number $T$ of iterations is not known apriori, then with $\gamma_t = \frac{c}{(t+1)^a}$, with $c>0$ and $a < \frac{1}{2}$, the following relation holds, for all $T \geq 0$,
    \begin{align}
        &\min_{0 \leq t \leq T} \|R_{\gamma_t}(x_t)\|^2 \nonumber\\
        &\leq \begin{cases}
            \frac{B_{\psi}(x^*,x_0)}{\frac{\alpha r}{2} \left(1-\frac{1}{w_5} \right) c^2 (T+1)^{1-2a}} + \frac{\widetilde C (1-\nu) c^{\frac{2 \nu}{1-\nu}}}{(1-\nu-2a) \frac{\alpha r}{2} \left(1-\frac{1}{w_5} \right) (T+1)^{\frac{2\nu a}{1-\nu}}} \quad &\text{for $0 < a < \frac{1 - \nu}{2}$,} \\
            \frac{B_{\psi}(x^*,x_0)}{\frac{\alpha r}{2} \left(1-\frac{1}{w_5} \right) c^2 (T+1)^{1-2a}} + \frac{\widetilde C c^{\frac{2 \nu}{1-\nu}} (1+\ln (T+1)) }{\frac{\alpha r}{2} \left(1-\frac{1}{w_5} \right) (T+1)^{1-2a}} \quad &\text{for $a=\frac{1-\nu}{2}$,} \\
            \frac{B_{\psi}(x^*,x_0)}{\frac{\alpha r}{2} \left(1-\frac{1}{w_5} \right) c^2 (T+1)^{1-2a}} + \frac{ 2a \widetilde C c^{\frac{2 \nu}{1-\nu}}}{(2a+\nu-1) \frac{\alpha r}{2} \left(1-\frac{1}{w_5} \right) (T+1)^{1-2a}} \quad &\text{for $\frac{1 - \nu}{2} < a < \frac{1}{2}$} .
        \end{cases} \nonumber
    \end{align}
    ii) For $\nu = 1$ (Lipschitz continuous case), a constant step size selection $\gamma_t = \gamma \leq \min \left( \frac{\alpha}{\sqrt{2} L_1 (\sqrt{r(w_5-1) \alpha + 4})} , \frac{\alpha \sqrt{1-r}}{2 L_1} \right)$ yields, for all $T \geq 0$,
    \begin{align}
        \min_{0 \leq t \leq T} \|R_{\gamma_t}(x_t)\|^2 \leq \frac{B_{\psi}(x^*,x_0)}{\frac{\alpha r}{2} \left(1-\frac{1}{w_5} \right) \gamma^2 (T+1)} . \nonumber
    \end{align}
\end{theorem}
\begin{proof}
\textit{Case (i):} Let $\nu \in (0,1)$. The theorem assumes the existence of the Minty solution $x^*$ of the VI problem (cf.~\eqref{minty-vi}). We begin the analysis from the relation in Lemma~\ref{lem-det-special} and substitute $z = x^*$, the Minty solution. This makes the first term on the left-hand side non-positive, allowing it to be dropped, yielding,
    \begin{align}
        &\sum_{t=0}^T \gamma_t^2 \|R_{\gamma_t}(x_t)\|^2 \leq \frac{B_{\psi}(x^*,x_0)}{\frac{\alpha r}{2} \left(1-\frac{1}{w_5} \right)} + \frac{(1-\nu) \left( \sum_{t=0}^{T} \gamma_t^{\frac{2}{1-\nu}} \right)}{\frac{\alpha r}{2} \left(1-\frac{1}{w_5} \right)} \nonumber\\
        & \times \Bigg[ \left( \frac{2^{2\nu - 1} L_{\nu}^2 \left(r(w_5-1) \alpha + 4 \right) \nu^\nu }{\alpha^{1+\nu}} \right)^{\frac{1}{1-\nu}} + \left( \frac{2^{1+\nu} L_{\nu}^2 \nu^\nu}{\alpha^{1+\nu} (1-r)^\nu}   \right)^{\frac{1}{1-\nu}} \Bigg]  , \nonumber
    \end{align}
    where the constant $0<r<1$ and $w_5>1$. Note that the quantity on the left hand side of the preceding relation can be further lower estimated as $\sum_{t=0}^T \gamma_t^2 \|R_{\gamma_t}(x_t)\|^2 \geq \left(\sum_{t=0}^T \gamma_t^2 \right) \min_{0 \leq t \leq T} \|R_{\gamma_t}(x_t)\|^2$, which when used in the preceding relation yields
    \begin{align}
        &\min_{0 \leq t \leq T} \|R_{\gamma_t}(x_t)\|^2 \leq \frac{B_{\psi}(x^*,x_0)}{\frac{\alpha r}{2} \left(1-\frac{1}{w_5} \right) \sum_{t=0}^T \gamma_t^2} + \frac{\widetilde C \sum_{t=0}^{T} \gamma_t^{\frac{2}{1-\nu}}}{\frac{\alpha r}{2} \left(1-\frac{1}{w_5} \right) \sum_{t=0}^T \gamma_t^2} , \label{hol-spcl1}
    \end{align}
    where $\widetilde C$ is expressed in the theorem. The first relation of result (i) of the theorem can be obtained by choosing a constant step size $\gamma_t = \gamma = \frac{c}{(T+1)^a}$, with constant $a < \frac{1}{2}$ and $T$ is the total number of iterations we know in advance.

    To derive the second part of result (i), we set the step size $\gamma_t = \frac{c}{(t+1)^a}$, with $a < \frac{1}{2}$, substitute it into relation~\eqref{hol-spcl1}, and apply Lemma~\ref{lem-step-bd2}(ii) (with $p=\nu$) together with Lemma~\ref{lem-step-bd2}(v), which yields the desired result.\\

\noindent
\textit{Case (ii):} We start from Lemma~\ref{lem-det-special}(ii) and substitute $z = x^*$ as a Minty solution (assumed to exist in the theorem). We then drop the first quantity on its left hand-side and use a similar analysis as done for the case (i) to obtain
    \begin{align}
        \min_{0 \leq t \leq T} \|R_{\gamma_t}(x_t)\|^2 \leq \frac{B_{\psi}(x^*,x_0)}{\frac{\alpha r}{2} \left(1-\frac{1}{w_5} \right) \sum_{t=0}^T \gamma_t^2 } . \nonumber
    \end{align}
    We choose a constant step size $\gamma_t = \gamma \leq \min \left( \frac{\alpha}{\sqrt{2} L_1 (\sqrt{r(w_5-1) \alpha + 4})} , \frac{\alpha \sqrt{1-r}}{2 L_1} \right)$ for some constant $0< r < 1$ and $w_5 > 1$ to arrive at the stated relation. \hfill$\square$ 
\end{proof}
Theorem~\ref{thm-residual} establishes the convergence rate for the residual function, without requiring the monotonicity of the mapping $F$ (Assumption~\ref{asum-monotone}), but requires the assumption that a Minty solution must exist. The use of the residual function allows us to bypass the requirement  of the set $X$ boundedness (Assumption~\ref{asum-set2}) for $\nu \in (0,1]$, which is not possible when analyzing the method via the dual gap function. From Theorem~\ref{thm-residual}(i), we see that using a constant step size $\gamma_t = \gamma = {c}/{(T+1)^{\frac{1-\nu}{2}}}$, which depends on the H\"older exponent $\nu$, yields a convergence rate of $O(1/T^{\nu})$ in terms of the residual function. Furthermore, knowledge of problem parameters enables better choices of the constant $c$ that defines the step size (see Remark~\ref{rem-optimal-c}). In the Lipschitz continuous case (Theorem~\ref{thm-residual}(ii)), a convergence rate of $O(1/T)$ is achieved for the \textit{residual function} without any assumption on the monotonicity of the VI mapping, which appears to be a new result for the Popov mirror-prox algorithm to the best of our knowledge.

\section{Simulations}\label{sec:simulations}
In this section, we present the simulation results for our algorithm and compare them with the state-of-the-art. We consider the problem of a noisy matrix game, piecewise quadratic functions, and classifying the MNIST and the CIFAR-10 datasets using a ResNet-18 Convolutional Neural Network (CNN) model.

\paragraph{Noisy Matrix Game.}\label{noisy_mat_game}
The problem under consideration here is the stochastic version of the matrix game, whose deterministic counterpart has been studied in \cite{chakraborty2024popov}. The noisy matrix game problem is
\begin{align}
    \text{Player 1:} \quad &\min_{x_1\in \Delta_{2}} \;\; 
    \EXP{\la x_1,A x_1\ra + \la x_1, B x_2\ra + \la p + \xi_1 ,x_1\ra} \nonumber\\
    \text{Player 2:} \quad & \min_{x_2 \in \Delta_{2}} \;\; 
    \EXP{\la x_2, C x_2\ra + \la x_1, D x_2\ra  + \la q + \xi_2 , x_2\ra},  \nonumber 
\end{align}
where $x_1$ and $x_2$ are the decision variables of players 1 and 2, respectively, and $\Delta^2 \subset \R^2$ is the two-dimensional probability simplex, i.e., $x_1 \geq 0$, $x_2 \geq 0$, $\la x_1, e \ra = 1$, and $\la x_2, e \ra = 1$, where $e \in \R^2$ is the all-one vector. The constraint set $\Delta^2$ satisfies Assumptions~\ref{asum-set1} and \ref{asum-set2}. The variables $\xi_1$ and $\xi_2$ are random vectors with a normal distribution $\mathcal{N}(0,\Sigma)$, with $\Sigma=0.4 I$. Taking the gradient of the loss function of each player with respect to its own decision variable and concatenating the gradients into a vector, the game can be written as a stochastic VI problem with the stochastic mapping $\widehat F(\cdot)$ given by
\begin{align}
    \widehat F(x, \xi_1, \xi_2) = \begin{bmatrix}
        A+A^T & B \\
        D^T & C+C^T
    \end{bmatrix} x + \begin{bmatrix}
        p + \xi_1 \\
        q + \xi_2
    \end{bmatrix} , \nonumber 
\end{align}
where $x = [x_1, x_2]^T \in X \equiv \Delta^2 \times \Delta^2$. The expected mapping $F$ is
\begin{align}
    F(x) = \EXP{\widehat F(x, \xi_1, \xi_2)} = \begin{bmatrix}
        A+A^T & B \\
        D^T & C+C^T
    \end{bmatrix} x + \begin{bmatrix}
        p \\
        q
    \end{bmatrix} \nonumber 
\end{align}
since $\xi_1, \xi_2$ have zero mean. 
The Jacobian of the mapping $F$ is
\begin{align}
    \nabla F(x) = \begin{bmatrix}
        A+A^T & B \\
        D^T & C+C^T
    \end{bmatrix} . \nonumber
\end{align}
To satisfy Assumptions~\ref{asum-holder} (with $\nu = 1$, $M_\nu = 0$, and $L_\nu = L_1 = L$) and \ref{asum-monotone}, we construct the matrix $\nabla F(x)$ to have eigenvalues in $[0, L]$, where $L = 10$ is the Lipschitz constant of the expected mapping $F(\cdot)$. 
This is done using eigenvalue decomposition: a diagonal matrix with eigenvalues in $[0, L]$ is combined with an orthogonal matrix $Q$, obtained via the QR decomposition of a random matrix with entries drawn from a uniform distribution over $[0, 4]$. Details are provided in \cite[Section 7]{chakraborty2024random}. Vectors $p$ and $q$ are sampled from $\mathcal{N}(0,1)$. We compare the proposed Popov Mirror-Prox (PMP) algorithm with the Korpelevich Mirror-Prox (KMP) \cite{juditsky2011solving} and the Universal Mirror-Prox (UMP) \cite{klimza2024universal}, which is the Korpelevich method in Euclidean space with adaptive step sizes. Two choices of $\psi(\cdot)$ are used for the Korpelevich and Popov methods: (i)~Euclidean: $\psi(x) = \frac{1}{2} \|x\|_2^2$ and (ii)~Entropic: $\psi(x) = \sum_{j=1}^2 \sum_{i=1}^2 x_j^{(i)}(\ln x_j^{(i)}-1)$. UMP uses the Euclidean case. Note that for the Euclidean case, $\nabla \psi(x) = x$ and $\nabla^2 \psi(x) = I$, whereas for the Entropic case, $\nabla \psi(x) = [\ln x_1^{(1)},\ln x_1^{(2)},\ln x_2^{(1)},\ln x_2^{(2)}]^T$ and $\nabla^2 \psi(x) \succcurlyeq \frac{1}{\max_{j,i} x_j^{(i)}} I \succcurlyeq I$ since the decision variables are in the probability simplex, where $I$ is the identity matrix.
In both cases, $\psi$ is strongly convex with constant $\alpha = 1$.

The KMP algorithm is run for both Euclidean and Entropic cases using the step size $\gamma = \frac{1}{\sqrt{3} L}$ (see \cite[Theorem 1]{juditsky2011solving}), where $L = 10$ in our setup. The UMP algorithm is nearly parameter-free, except it requires the diameter $R$ of the constraint set $\Delta^2 \times \Delta^2$, such that $\max_{x,y \in \Delta^2 \times \Delta^2} \|x-y\| \leq R$. The maximum distance between two points in the set $\Delta^2$ is $\sqrt{2}$. Since our domain is a Cartesian product of two probability simplexes, we set $R = 2\sqrt{2}$. 
We implement PMP using both constant and diminishing step sizes, \textit{assuming no knowledge of problem parameters}. We arbitrarily set $c = 1$ for both Euclidean and Entropic cases and run the algorithm for $400$ iterations. 
For a fair comparison, we also run KMP using the same step-size scheme and averaging method as in PMP. While a similar convergence analysis can be carried out for KMP, yielding comparable rates, we omit the proof since it follows the same structure. For the Euclidean case, the Popov updates are:
\begin{align*}
   &y_{t+1}=\Pi_{\Delta^2 \times \Delta^2}[x_t-\gamma_t \widehat F(y_t,\xi_1^{(t)},\xi_2^{(t)})], \cr
   &x_{t+1} = \Pi_{\Delta^2 \times \Delta^2} [x_t- \gamma_t F(y_{t+1}, \xi_1^{(t+1)}, \xi_2^{(t+1)})],
\end{align*}
where $\xi_1^{(t)}$ and $\xi_2^{(t)}$ are the noise samples and $\widehat F(y_t,\xi_1^{(t)},\xi_2^{(t)})$ is the stochastic mapping used to update both $x_t$ and $y_{t+1}$. The projection on $\Delta^2$ for both players $1$ and $2$ are done using the bisection algorithm for simplex projection~\cite[Algorithm 3]{blondel2014large}. The same projection technique is also used for KMP and UMP. These methods require two stochastic mapping evaluations with different noise samples in the iterate update. For the Entropic case, PMP updates are
%
%
%
\begin{figure}[t]\centering
	\begin{subfigure}{.495\linewidth}
		\includegraphics[width=1\linewidth, height = 0.75\linewidth]
		{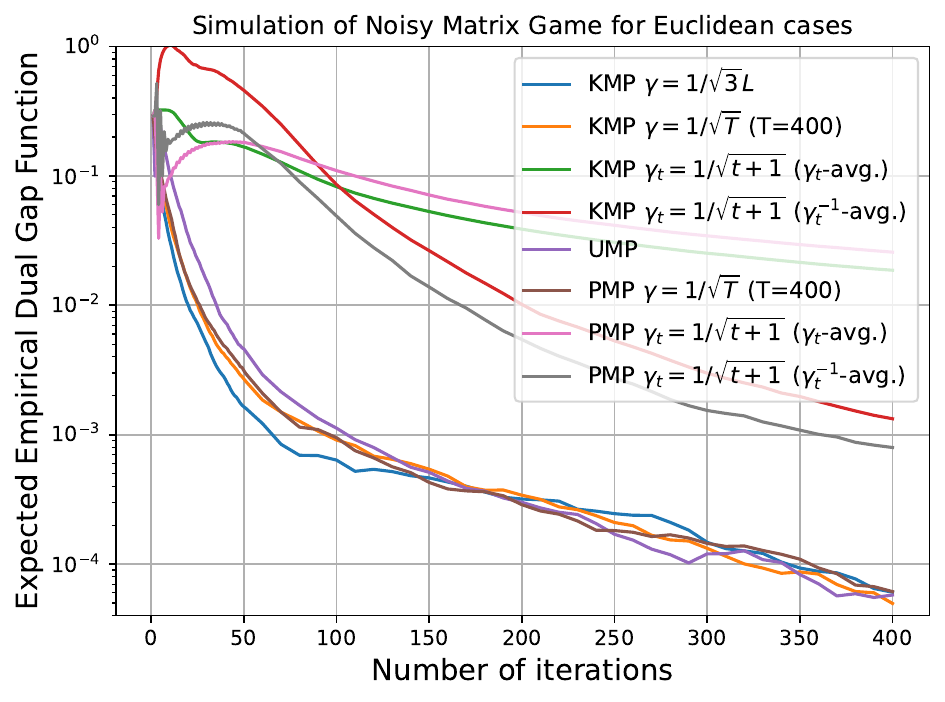}
		\caption{Euclidean case}
		\label{noisy_mat_eu}
	\end{subfigure}
	\begin{subfigure}{.495\linewidth}
		\includegraphics[width=1\linewidth,height = 0.75\linewidth]
		{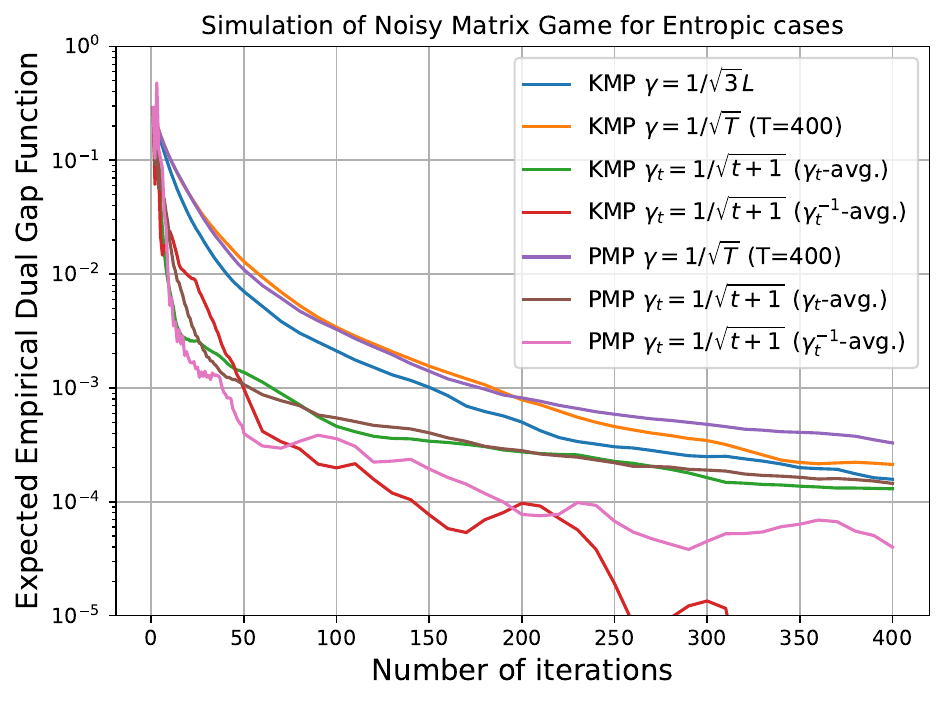}
		\caption{Entropic case}
        \label{noisy_mat_ent}
	\end{subfigure}
 \caption{Convergence of the Universal Mirror-Prox (UMP) algorithm, and the Korpelevich Mirror-Prox (KMP) and Popov Mirror-Prox (PMP) algorithms for Euclidean and Entropic cases, using various step-size schemes. `$\gamma_t$-avg.' and `$\gamma_t^{-1}$-avg.' denotes the averaging of the iterates with step sizes and inverse of the step sizes as weights, respectively.}
\label{fig_mat_game}
\end{figure}
%
%
%
%
%
%
\begin{align}
   &[y_{t+1}]_i = [z_{y,t+1}]_i [x_t]_i\exp(- \gamma_t [\widehat F(y_t,\xi_1^{(t)},\xi_2^{(t)})]_i), \quad i=1,2,3,4, \nonumber\\
   &[x_{t+1}]_i = [z_{x,t+1}]_i [x_t]_i\exp(- \gamma_t [F(y_{t+1}, \xi_1^{(t+1)}, \xi_2^{(t+1)})]_i), \quad i=1,2,3,4 \nonumber ,
\end{align}
where $[z_{y,t+1}]_i$ and $[z_{x,t+1}]_i$ are the normalizing coefficients with $i$ as the coordinate index. For player 1, ($i=1,2$), the normalizing coefficient $[z_{y,t+1}]_i$ is 
\begin{align}
    [z_{y,t+1}]_i = \frac{1}{[x_t]_1\exp(- \gamma_t [\widehat F(y_t,\xi_1^{(t)},\xi_2^{(t)})]_1+[x_t]_2\exp(- \gamma_t [\widehat F(y_t,\xi_1^{(t)},\xi_2^{(t)})]_2}, \nonumber
\end{align}
and for Player 2 ($i = 3, 4$), the same coefficient is of the form
\begin{align}
    [z_{y,t+1}]_i = \frac{1}{[x_t]_3\exp(- \gamma_t [\widehat F(y_t,\xi_1^{(t)},\xi_2^{(t)})]_3+[x_t]_4\exp(- \gamma_y [\widehat F(y_t,\xi_1^{(t)},\xi_2^{(t)})]_4}. \nonumber
\end{align}
The normalizing coefficients $[z_{x,t+1}]_i$ for the second update are computed similarly. These updates ensure that iterates remain positive (given positive initialization), so normalization alone projects them onto the probability simplex, eliminating the need for explicit projection and reducing computational cost. PMP in the Entropic case is implemented accordingly. To empirically estimate the dual gap function, we generate 200,000 random samples in $\Delta^2 \times \Delta^2$, representing the variable $z \in X$ in expression~\ref{dg-gap}. For each algorithm and at every iteration, we substitute the corresponding averaged iterate -- whose performance we aim to evaluate -- for the variable $x$ in expression~\ref{dg-gap}, compute the inner product, and take the maximum value over the $z$ points. The expected dual gap is estimated by averaging over 5 independent runs for each algorithm.
%
%
\begin{figure}[t]\centering
	\begin{subfigure}{.327\linewidth}
		\includegraphics[width=1\linewidth, height = 0.75\linewidth]
		{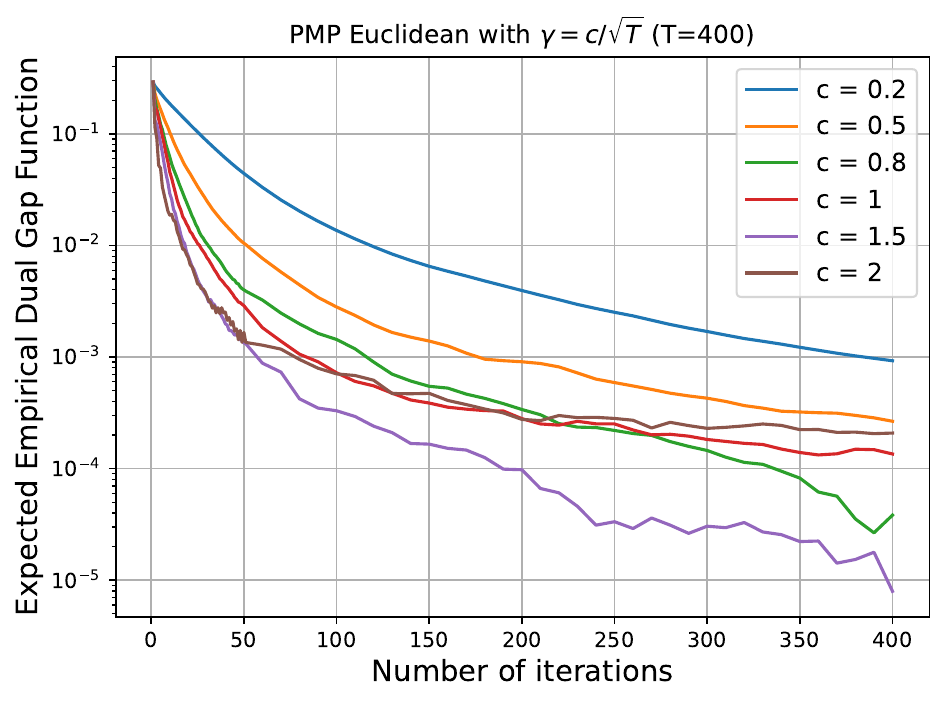}
	\end{subfigure}
	\begin{subfigure}{.327\linewidth}
		\includegraphics[width=1\linewidth, height = 0.75\linewidth]
		{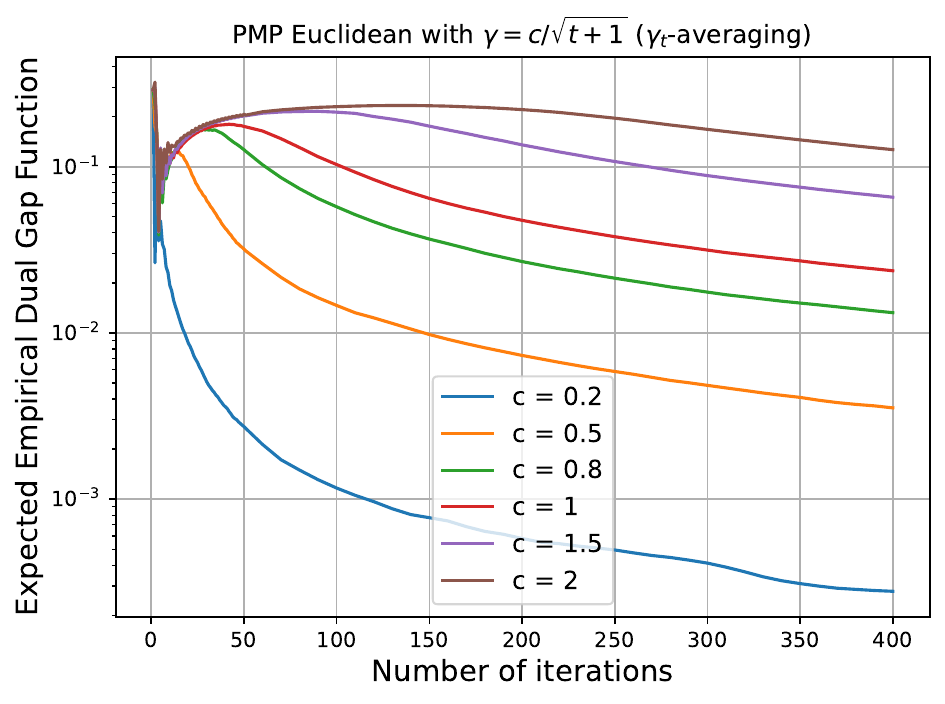}
	\end{subfigure}
    \begin{subfigure}{.327\linewidth}
		\includegraphics[width=1\linewidth, height = 0.75\linewidth]
		{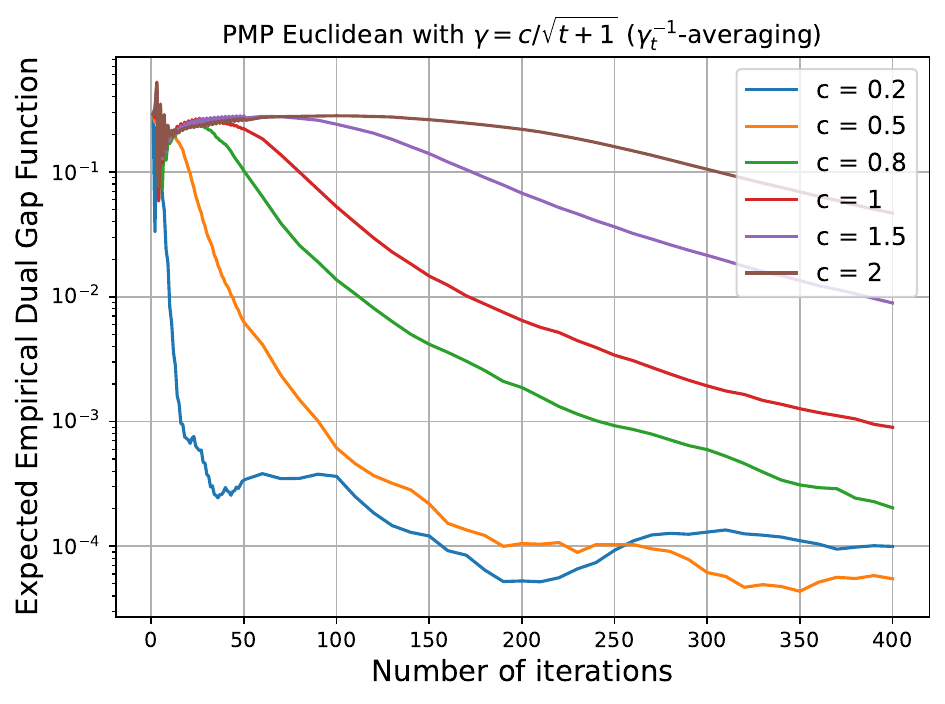}
	\end{subfigure}
    \begin{subfigure}{.327\linewidth}
		\includegraphics[width=1\linewidth, height = 0.75\linewidth]
		{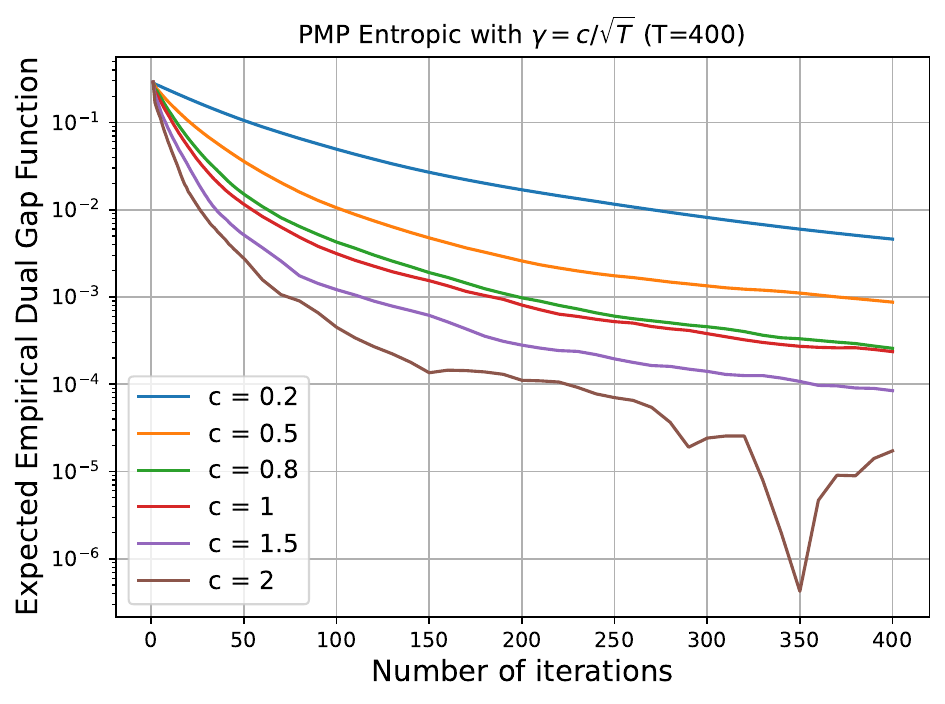}
	\end{subfigure}
    \begin{subfigure}{.327\linewidth}
		\includegraphics[width=1\linewidth, height = 0.75\linewidth]
		{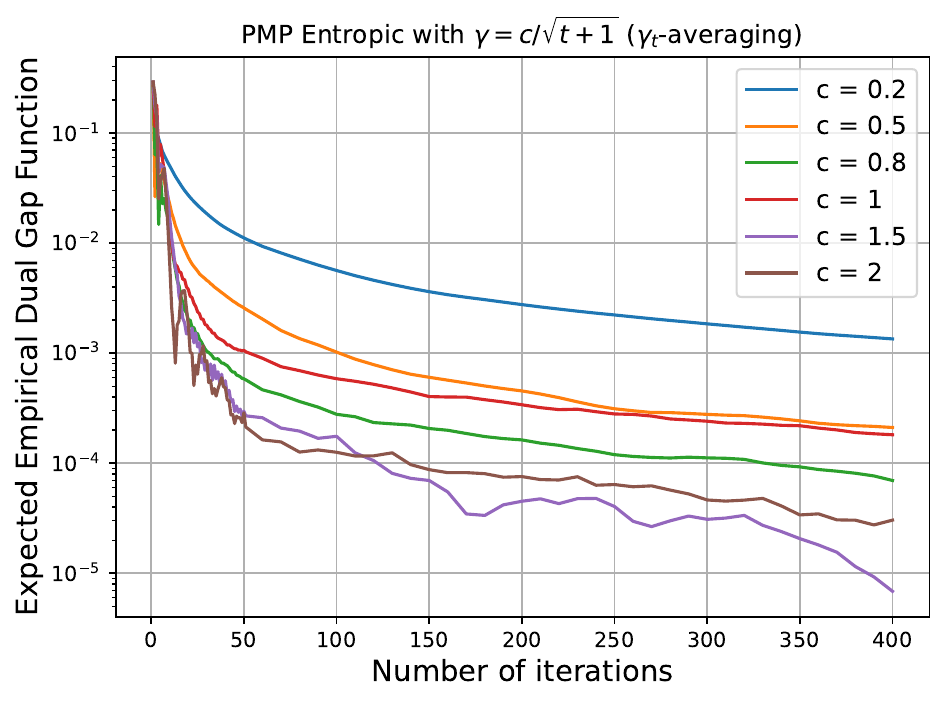}
	\end{subfigure}
    \begin{subfigure}{.327\linewidth}
		\includegraphics[width=1\linewidth, height = 0.75\linewidth]
		{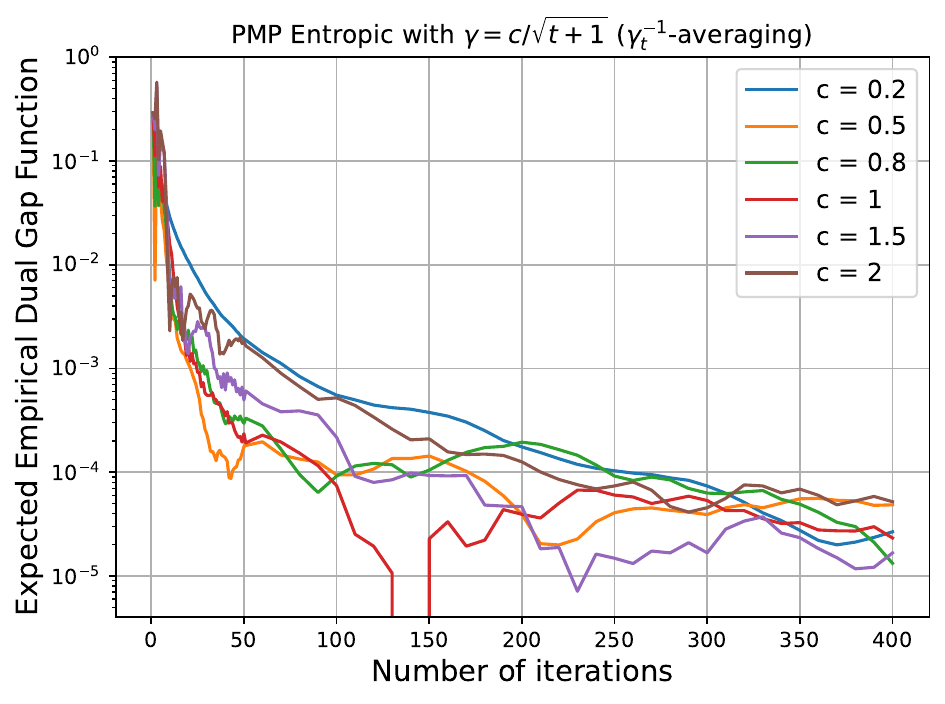}
	\end{subfigure}
 \caption{Convergence rates of the Popov Mirror-Prox (PMP) algorithm for both Euclidean and Entropic cases with different values of the constant $c$ in the step size. The figure presents results for constant step-size and diminishing step-size with iterate averaging with $\gamma_t$ and $\gamma_t^{-1}$ as weights represented as $\gamma_t$-averaging and $\gamma_t^{-1}$-averaging, respectively.}
\label{fig_diff_c}
\end{figure}
%
%
%

We run the PMP and KMP algorithms using our step size schemes with the choice of $c=1$ for both the Entropic and Euclidean cases, without tuning, assuming no knowledge of the problem parameters. Figures~\ref{noisy_mat_eu} and~\ref{noisy_mat_ent} show the convergence of the algorithms with respect to the expected empirical dual gap function for the Euclidean and the Entropic cases, respectively. In Figure~\ref{noisy_mat_eu}, we can see that the constant step size schemes for both KMP and PMP perform better than averaging of iterates with $\gamma_t$ and $\gamma_t^{-1}$ as weights. The UMP method, being adaptive, performs comparable to that of constant step size methods but still requires knowledge of the set diameter, unlike our parameter-free PMP schemes. 
Figure~\ref{noisy_mat_ent} shows that for the Entropic case, the averaging of the iterates with $\gamma_t^{-1}$ as weights for KMP converges the fastest. PMP with $\gamma_t^{-1}$-averaging also performs well. Overall, our PMP method achieves comparable performance while requiring only half of the mapping computations compared to the UMP and the KMP methods.

To explore the effect of tuning $c$, we simulate PMP for both Euclidean and Entropic cases under constant and diminishing step sizes with values of $c \in \{0.2,0.5,0.8,1,1.5,2\}$. The results in Figure~\ref{fig_diff_c}, which indicate that $c = 1.5$ (Euclidean) and $c = 2$ (Entropic) yield the best results for constant step sizes. For the averaging of iterates with $\gamma_t$ and $\gamma_t^{-1}$ as weights, $c = 0.2$ (Euclidean) and $c = 1.5$ (Entropic) works well. These choices of $c$ were guided by the observation that an optimal $c$ exists within a moderate range (see Remark~\ref{rem-optimal-c}). 



\begin{figure}[t]\centering
	\begin{subfigure}{.495\linewidth}
		\includegraphics[width=1\linewidth, height = 0.75\linewidth]
		{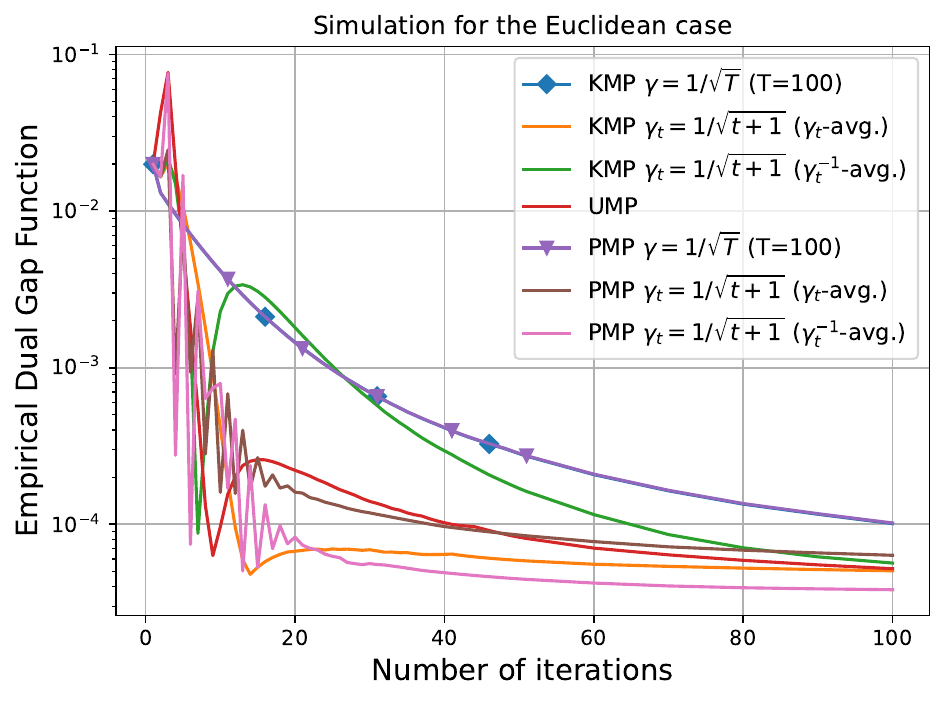}
	\end{subfigure}
	\begin{subfigure}{.495\linewidth}
		\includegraphics[width=1\linewidth,height = 0.75\linewidth]
		{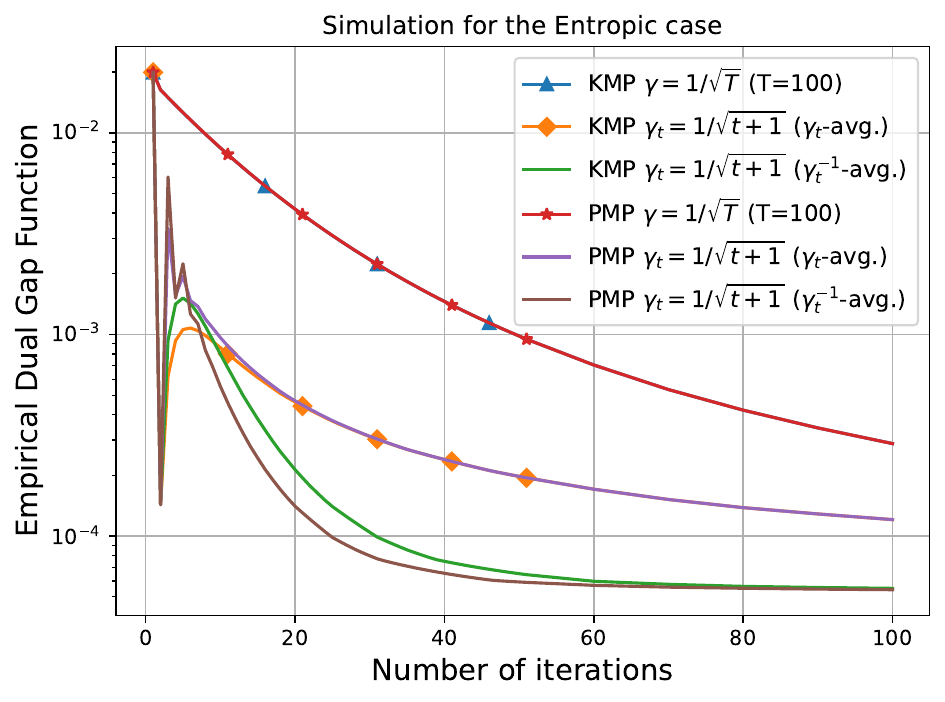}
	\end{subfigure}
 \caption{Convergence of the various algorithms for the deterministic case. Averaging with $\gamma_t$ and $\gamma_t^{-1}$ as weights are stated as `$\gamma_t$-avg.' and `$\gamma_t^{-1}$-avg.' respectively.}
\label{fig_piecewise_det}
\end{figure}


\paragraph{Piecewise Quadratic Functions} 
Consider an optimization problem where the loss function $\ell(\cdot): \Delta^2 \rightarrow \R$ is  the maximum of piecewise quadratic functions:
\begin{align}
    \ell(x) = \max_{i=1,2,3,4} \ell_i(x), \quad \text{where } \ell_i(x) = \left\{\frac{1}{2} \la x, A_i x \ra + \la b_i,x \ra \right\} . \nonumber
\end{align}
The optimization problem under consideration is
\begin{align}
    \min_{x \in \Delta^2} \ell(x) = \min_{x \in \Delta^2} \left\{\max_{i=1,2,3,4} \left\{\frac{1}{2} \la x, A_i x \ra + \la b_i,x \ra \right\} \right\} , \nonumber
\end{align}
where the set $\Delta^2$ is the $2$-dimensional probability simplex satisfying Assumptions~\ref{asum-set1} and~\ref{asum-set2}. We construct the matrices $A_i$'s to be Positive Semi-Definite (PSD) with eigenvalues in the range $[0,L_i]$, where we choose $L_i = i$, for each $i = 1,2,3,4$. Their construction follows the method used in the noisy matrix game experiment. The PSD property ensures that each $i$-th quadratic loss function $\ell_i(x)$ is convex, for all $i = 1,2,3,4$. Since the point-wise maximum of convex functions is convex \cite[Section 3.2.3]{boyd2004convex}, $\ell(x)$ is also convex. The vector $b_1 \in \R^2$ is sampled from $\mathcal{N}(0,1)$. We randomly choose 3 transition points $x_{\text{trans},1}, x_{\text{trans},2}$, and $x_{\text{trans},3}$ in the simplex and compute $b_i = (A_{i+1} - A_i) x_{\text{trans},i-1} + b_{i-1}$, for $i = 2,3,4$, ensuring each piecewise component is active over distinct regions of the domain. The loss function $\ell(x)$ is non-smooth at the transition points. We define an index set $\mathcal{I}(x)$, which denotes the set of active indices of the loss function where the maximum is attained at a particular point $x$, i.e., $\mathcal{I}(x) = \{ i \mid \ell_i(x) = \ell(x)\}$. Here, we need to consider a set-valued mapping $F: \Delta^2 \rightrightarrows \R^2$ which is the subdifferential set for the loss function $\ell(x)$ given by
\begin{align}
    F(x) = \partial \ell(x) = \text{conv} \left \{\frac{1}{2} (A_i + A_i^T) x + b_i \mid i \in \mathcal{I}(x) \right\}, \label{map_piecewise}
\end{align}
where conv denotes the convex hull of the set, i.e., it is given by
\begin{align*}
\left\{ \sum_{i \in \mathcal{I}(x)} \lambda_i \left( \frac{1}{2}(A_i + A_i^\top)x + b_i \right) 
\,\middle|\, \lambda_i \geq 0,\, \sum_{i \in \mathcal{I}(x)} \lambda_i = 1 \right\} .
\end{align*}
As the function $\ell(\cdot)$ is convex and its domain is the entire space, its subdifferential set $\partial \ell(\cdot)$ is non-empty everywhere \cite[Proposition 10.21]{rockafellar2009variational} and also monotone~\cite[Theorem 12.17]{rockafellar2009variational}
\[\la u-v, x-y \ra \geq 0 , \quad \text{for all $x, y \in \R^2$, $u \in \partial \ell(x)$, and $v \in \partial \ell(y)$}.\]
Therefore, the set-valued mapping $F(\cdot)$ is a well defined and monotone over $\R^2$ (hence Assumption~\ref{asum-monotone} is satisfied for such a mapping). Since the constraint set is a simplex which is compact, Assumption~\ref{asum-holder} will be satisfied pointwise for the set-valued mapping $F(\cdot)$ with $\nu = 0$, i.e., a bounded variation.

As done for the matrix game, assuming no knowledge of problem parameters, we consider both Euclidean and Entropic cases, which ensures the function $\psi$ is $\alpha$-strongly convex with $\alpha = 1$. We simulate our PMP method with constant and diminishing step sizes using both $\gamma_t$ and $\gamma_t^{-1}$ as the averaging weights. We compare against KMP with similar step sizes and UMP, where the UMP requires knowledge of the diameter $R = \max_{x,y \in \Delta^2} \|x-y\|$. Here, $R = \sqrt{2}$.


\begin{figure}[t]\centering
	\begin{subfigure}{.495\linewidth}
		\includegraphics[width=1\linewidth, height = 0.75\linewidth]
		{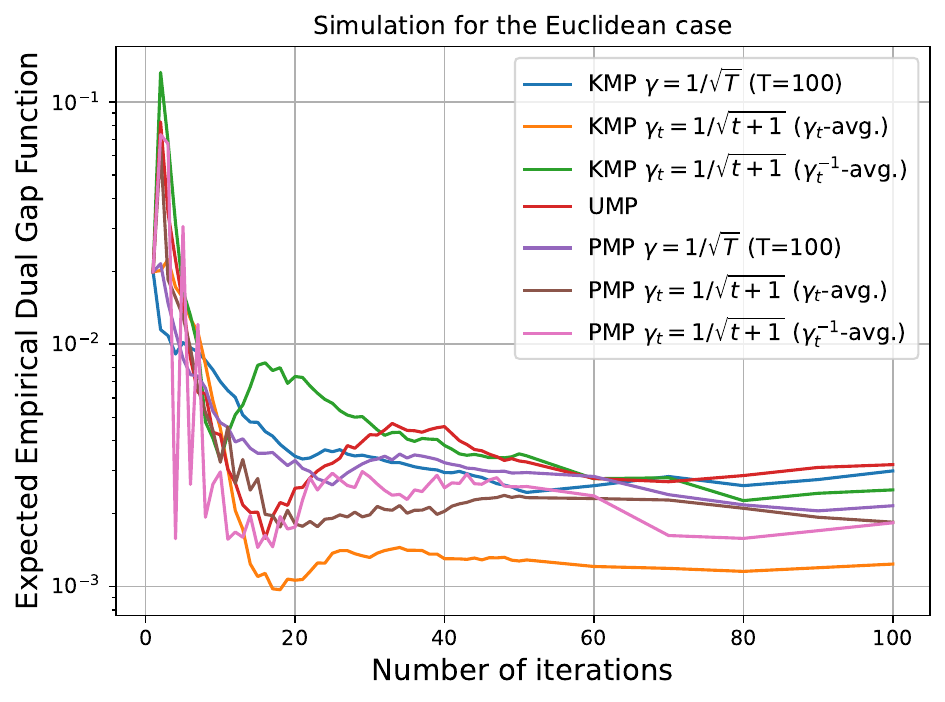}
	\end{subfigure}
	\begin{subfigure}{.495\linewidth}
		\includegraphics[width=1\linewidth,height = 0.75\linewidth]
		{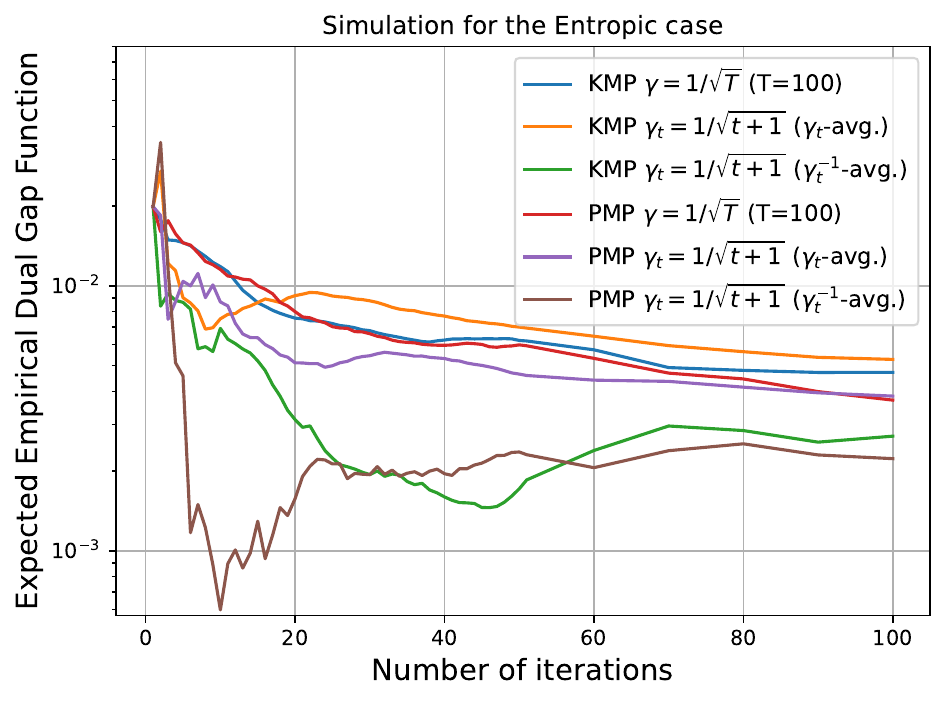}
	\end{subfigure}
 \caption{Convergence of the various algorithms for the stochastic case. `$\gamma_t$-avg.' and `$\gamma_t^{-1}$-avg.' denotes averaging with $\gamma_t$ and $\gamma_t^{-1}$, respectively, as weights.}
\label{fig_piecewise_stoc}
\end{figure}


The dual gap function is empirically estimated as done in the Noisy Matrix Game experiment by randomly sampling $200,000$ points from the simplex $\Delta^2$. Since we sample points from a continuous distribution, with probability $1$, none of the sampling points coincide with the points where the mapping is not single-valued. The projection onto the simplex in the iterate updates is performed as described in the Noisy Matrix Game experiment. Figure~\ref{fig_piecewise_det} shows the convergence behavior of KMP, UMP, and PMP under various step-size schemes. The results demonstrate that PMP with diminishing step sizes, and iterate averaging with $\gamma_t^{-1}$ as weights, achieves the best performance in both Euclidean and Entropic cases.

We have also considered a stochastic version of the previous problem
\begin{align}
    \min_{x \in \Delta^2} \left\{\max_{i=1,2,3,4} \left\{ \EXP{\frac{1}{2} \la x, A_i x \ra + \la b_i + \xi,x \ra} \right\} \right\} , \nonumber
\end{align}
where $\xi \sim \mathcal{N}(0,\Sigma)$, $\Sigma= 0.4I$ is the noise. The problem settings are same as the deterministic case and, hence, the expected value of the mapping is given by relation~\eqref{map_piecewise}. This problem also satisfies all the previous assumptions and the stochastic set-valued map can be written as
\begin{align}
    \widehat F(x,\xi) = \text{conv} \left \{\frac{1}{2} (A_i + A_i^T) x + b_i + \xi \mid i \in \mathcal{I}(x) \right\} . \nonumber
\end{align}
In order to compute the expectation of the empirical dual gap function, we run $5$ different experiments and average over it.

Figure~\ref{fig_piecewise_stoc} shows the convergence of the stochastic algorithms. In the Euclidean case, KMP with diminishing step size and $\gamma_t$ as averaging weights performs better than the other methods, while for the Entropic case, PMP with diminishing step size and $\gamma_t^{-1}$ as averaging weights achieve superior performance. Notably, our step sizes are entirely parameter-free, and our PMP method requires half as many mapping computations as the other two algorithms.


\paragraph{MNIST and CIFAR-10 data classification using ResNet-18 model.}
We evaluate the classification performance of a ResNet-18 CNN (without pretraining) on the MNIST and CIFAR-10 datasets. Let $\{u_i, v_i\}_{i=1}^m$ denote the training data, where $u_i$ is an input image and $v_i \in \{0,1\}^C$ is its one-hot encoded class label. Thus, $v_{ij}=1$ only if $u_i$ belongs to the class $j \in \{1,\ldots,C\}$. MNIST contains grayscale images of size $28 \times 28 \times 1$ representing handwritten digits $0-9$, while CIFAR-10 comprises of RGB colored images of size $32 \times 32 \times 3$ from 10 object classes: airplane, automobile, bird, cat, deer, dog, frog, horse, ship, and truck. Hence, for both the datasets, total number of classes $C = 10$.


\begin{figure}[t]\centering
	\begin{subfigure}{.495\linewidth}
		\includegraphics[width=1\linewidth, height = 0.75\linewidth]
		{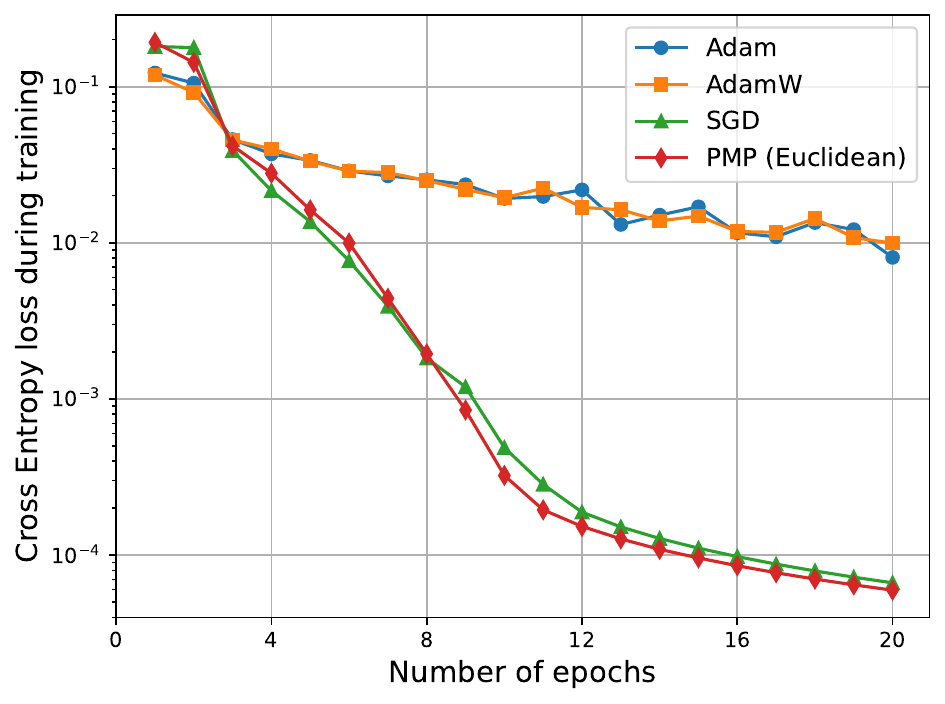}
	\end{subfigure}
	\begin{subfigure}{.495\linewidth}
		\includegraphics[width=1\linewidth,height = 0.75\linewidth]
		{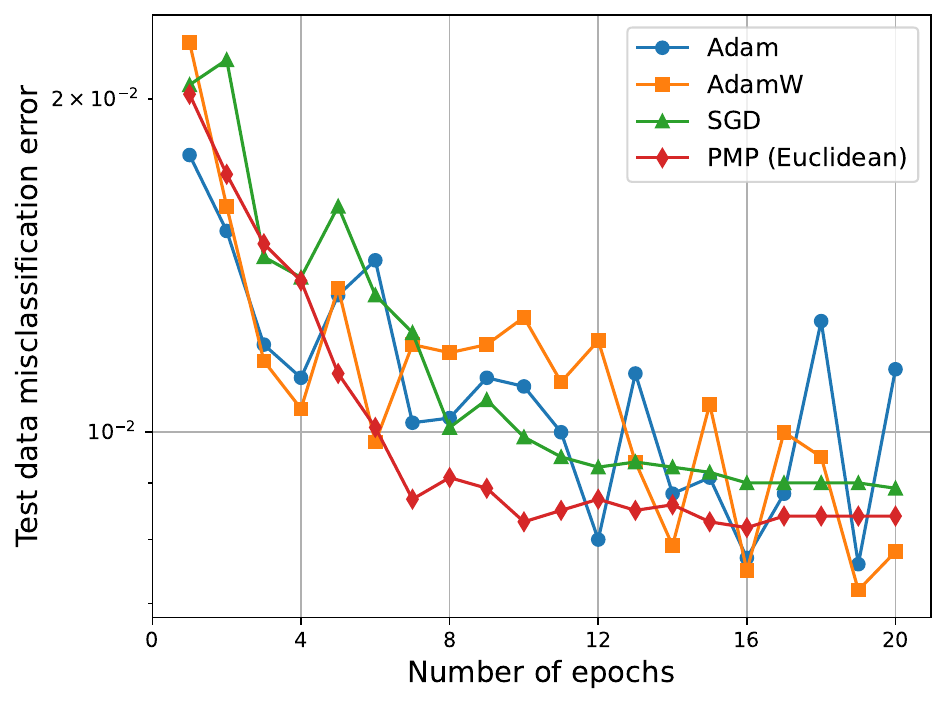}
	\end{subfigure}
 \caption{Training loss and testing misclassification error for Adam, AdamW, SGD, and PMP (Euclidean case) while training ResNet-18 on MNIST dataset.}
\label{fig_mnist}
\end{figure}



\begin{figure}[t]\centering
	\begin{subfigure}{.495\linewidth}
		\includegraphics[width=1\linewidth, height = 0.75\linewidth]
		{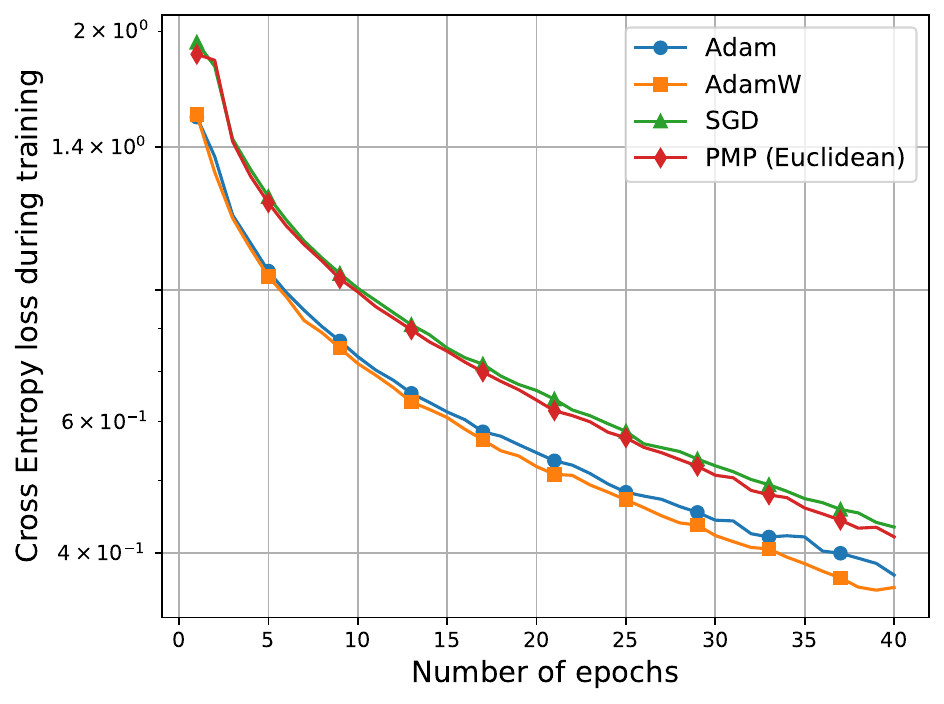}
	\end{subfigure}
	\begin{subfigure}{.495\linewidth}
		\includegraphics[width=1\linewidth,height = 0.75\linewidth]
		{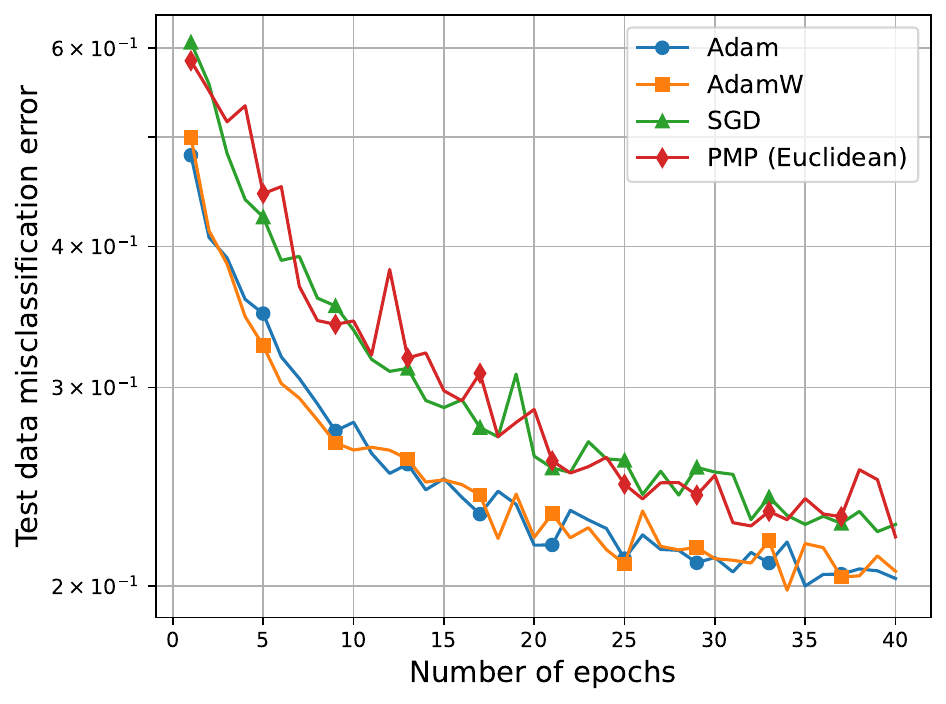}
	\end{subfigure}
 \caption{Training loss and test error for Adam, AdamW, SGD, and PMP (Euclidean case) optimizers while training ResNet-18 on the CIFAR-10 dataset.}
\label{fig_cifar10}
\end{figure}


Let $x \in \R^n$ denote the model parameters and $f(u_i;x) \in \R^C$ the output logits for the input data $u_i$. The model applies a softmax to produce a probability vector $p(u_i;x) \in \Delta^C$, where $\Delta^C$ is $C$-dimensional probability simplex and
\begin{align*}
    p(u_i;x) = \frac{\exp(f_j(u_i,x))}{\sum_{z=1}^C \exp(f_z(u_i,x))} \quad \text{for class $j$.}
\end{align*}
The goal is to minimize the cross-entropy loss between $p(u_i;x)$ and $v_i$. The instantaneous cross-entropy loss for the data $(u_i,v_i)$ is expressed as
\begin{align*}
    \ell_{CE}(p(u_i;x), v_i) = - \sum_{j = 1}^C y_{ij} \log(p_j(u_i;x)) .
\end{align*}
This leads to the Empirical Risk Minimization (ERM) problem:
\begin{align*}
    \min_{x \in \R^n} L(x) := \frac{1}{m} \sum_{i=1}^m \ell_{CE}(p(u_i;x), v_i) ,
\end{align*}
with associated variational inequality mapping $F(x) = \nabla_x L(x)$. While Assumption~\ref{asum-monotone} may not hold, Assumption~\ref{asum-holder} remains valid if the VI mapping is bounded. We implement our PMP algorithm with the Euclidean mirror map $\psi(\cdot) = \frac{1}{2}\|\cdot\|^2$, using a constant step size $\gamma = \frac{c}{\sqrt{T}}$, and compare against Stochastic Gradient Descent (SGD) \cite{robbins1951stochastic, bottou2010large}, Adam \cite{kingma2014adam}, and AdamW \cite{loshchilov2017decoupled}. All methods use a mini-batch size of $128$; the stochastic mapping at iteration $t$ is:
\begin{align*}
    \widehat F(x_t,\{u_i,v_i\}_{i\in \mathcal{I}_t}) = \frac{1}{128} \sum_{i\in \mathcal{I}_t} \nabla_x \ell_{CE}(p(u_i;x_t), v_i) . 
\end{align*}
For PMP, we report training loss and test accuracy using the last iterate; for other methods, we use PyTorch's built-in implementations.

MNIST contains 60,000 training and 10,000 test samples. We run all methods for 20 epochs. Note that a single epoch represents a complete pass over the entire training dataset. We select $\gamma = 0.01$ for PMP and SGD based on the smallest training loss after 20 epochs. Adam and AdamW use PyTorch's default learning rates. The UMP is excluded since the set diameter within which the iterate stays is not known. Figure~\ref{fig_mnist} shows the decay in training loss and test misclassification error over epochs. PMP performs competitively, slightly outperforming SGD on test accuracy and showing strong performance on the relatively simple MNIST dataset.


We next evaluate the more complex CIFAR-10 dataset, that contains 50,000 training and 10,000 test samples. For PMP and SGD, we use a fixed step size $\gamma = 0.2$, selected based on the smallest training loss over 20 epochs. Adam and AdamW use their default step sizes. All methods are run for 40 epochs. Figure~\ref{fig_cifar10} shows the training loss and test error decay over epochs. The Adam and AdamW outperform the other methods, benefiting from momentum and adaptive step sizes. The PMP and SGD show comparable test error performance, while AdamW achieves the smallest training loss. These results suggest that incorporating momentum or adaptive step sizes into the PMP could improve its performance, which are interesting directions for future work, but fall outside the scope of this paper.
\section{Conclusion}\label{sec:conclusion}
In this paper, we analyzed both stochastic and deterministic variants of the Popov Mirror-Prox algorithm and established their optimal convergence rates under various step size schemes. For monotone VIs with polynomial growth and bounded variance \( M_\nu > 0 \), we showed that both variants achieve the optimal rate \( O(1/\sqrt{T}) \) for the dual gap function over compact sets using fully parameter-free step sizes. For deterministic monotone VIs with H\"older continuous mappings, knowledge of the H\"older exponent \( \nu \in (0,1] \) yields an improved rate of \( O(1/T^{\frac{1+\nu}{2}}) \) for the gap function. In the same setting, a rate of \( O(1/T^{\nu}) \) can be achieved for the residual function without requiring the monotonicity of the mapping or the set boundedness assumptions.

\section*{Declarations}


\noindent \textbf{Conflict of interest:} The authors declare that there are no conflicts of interest.

\appendix

\section{Proofs of Preliminary Results}

\subsection{Proof of Lemma~\ref{lem-step-bound}}\label{lem-step-bound-proof}

\begin{proof}
    With $\gamma_t = \frac{c}{\sqrt{t+1}}$, the upper bound on $\sum_{t=\lceil \frac{T}{2} \rceil}^T \gamma_t^2$ can be obtained as
    \begin{align}
        \sum_{t=\lceil \frac{T}{2} \rceil}^T \gamma_t^2 &\leq \int_{t=\lceil \frac{T}{2} \rceil - 1}^{T} \frac{c^2}{t+1} dt \nonumber\\
        & = c^2 (\ln(T + 1) - \ln(\lceil {T}/{2} \rceil)) \nonumber\\
        & = c^2 \ln \left(\frac{T+1}{\lceil {T}/{2} \rceil} \right) \leq c^2 \ln \left(\frac{T+1}{{T}/{2}} \right) \leq c^2 \ln \left(2+\frac{2}{T} \right) \leq c^2 \ln(4) . \nonumber
    \end{align}
    Moreover the lower bound on $\sum_{t=\lceil \frac{T}{2} \rceil}^T \gamma_t$ becomes
    \begin{align}
        &\sum_{t=\lceil \frac{T}{2} \rceil}^T \gamma_t \geq  \sum_{t=\lceil \frac{T}{2} \rceil}^T \frac{c}{\sqrt{T+1}} = \frac{c}{\sqrt{T+1}} \left( T- \left\lceil \frac{T}{2} \right\rceil + 1 \right) \nonumber\\
        &\geq \frac{c}{\sqrt{T+1}} \left( T+1 - \frac{T+1}{2} \right) = \frac{c}{\sqrt{T+1}} \times \frac{T+1}{2} = \frac{c \sqrt{T+1}}{2} . \nonumber \qquad\qquad \;\; \square
    \end{align}
\end{proof}


\subsection{Proof of Lemma~\ref{lem-step-bd2}}\label{lem-step-bd2-proof}
\begin{proof}
    \textit{Part (i):} With $\gamma_t = \frac{c}{(t+1)^a}$, the lower bound on $\sum_{t=0}^T \gamma_t$ for any $T \geq 1$ becomes
    \begin{align}
        \sum_{t=0}^T \gamma_t = \sum_{t=0}^T \frac{c}{(t+1)^a} \geq \sum_{t=0}^T \frac{c}{(T+1)^a} = c (T+1)^{1-a} . \nonumber 
    \end{align}
    For $T=0$, the inequality is replaced by equality with $\gamma_0 = c$, which proves this part.

    \textit{Part (ii):} For any $T \geq 0$, $p \in [0,1)$, and $\gamma_t = \frac{c}{(t+1)^a}$, we see
    \begin{align*}
        \sum_{t=0}^T \gamma_t^{\frac{2}{1-p}} \leq \left.\frac{c^{\frac{2}{1-p}}}{(t+1)^{\frac{2a}{1-p}}} \right|_{t=0} + \int_{t=0}^{T} \frac{c^{\frac{2}{1-p}}}{(t+1)^{\frac{2a}{1-p}}} dt = c^{\frac{2}{1-p}} \left( 1 + \int_{t=0}^{T} \frac{1}{(t+1)^{\frac{2a}{1-p}}} dt \right) . 
    \end{align*}
    Now we simplify the preceding relation for different ranges of the value of $a$. \\
    (1) When $0<a<\frac{1-p}{2}$, then
    \begin{align*}
        &\sum_{t=0}^T \gamma_t^{\frac{2}{1-p}} \leq c^{\frac{2}{1-p}} \left( 1 + \frac{(1-p)\left( (T+1)^{\frac{1-p-2a}{1-p}} - 1 \right)}{1-p-2a} \right)\\
        &= c^{\frac{2}{1-p}} \left( \frac{(1-p) (T+1)^{\frac{1-p-2a}{1-p}}}{1-p-2a} - \frac{2a}{1-p-2a} \right) \leq \frac{(1-p) c^{\frac{2}{1-p}} (T+1)^{\frac{1-p-2a}{1-p}}}{1-p-2a} .
    \end{align*}
    (2) When $a = \frac{1-p}{2}$, then
    \begin{align*}
        \sum_{t=0}^T \gamma_t^{\frac{2}{1-p}} \leq c^{\frac{2}{1-p}} \left( 1 + \int_{t=0}^{T} \frac{1}{t+1} dt \right) = c^{\frac{2}{1-p}} ( 1+ \ln(T+1) ) .
    \end{align*}
    (3) When $\frac{1-p}{2}<a<1$, then
    \begin{align*}
        \sum_{t=0}^T \gamma_t^{\frac{2}{1-p}} \leq c^{\frac{2}{1-p}} \left( 1 + \frac{(1-p)\left( 1 - (T+1)^{\frac{1-p-2a}{1-p}} \right)}{2a+p-1} \right) \leq \frac{2a c^{\frac{2}{1-p}}}{2a+p-1} .
    \end{align*}

    \textit{Part (iii):} With $p \in (0,1)$, the proof follows the same analysis as that of part (ii).

    \textit{Part (iv):} For any $T \geq 0$, we see that
    \begin{align}
        \sum_{t=0}^T \frac{1}{\gamma_t} = \sum_{t=0}^T \frac{(t+1)^a}{c} \geq \int_{t=0}^T \frac{(t+1)^a}{c} dt = \frac{(T+1)^{1+a} - 1}{c(1+a)} . \label{inv_step_sum}
    \end{align}
    Next, we claim that
    \begin{align}
        (T+1)^{1+a} - 1 \geq T^{1+a} , \label{itr-bound}
    \end{align}
    with some constant $0<a<1$. In order to show that, we consider a function $\phi(T) = (T+1)^{1+a} - 1 - T^{1+a}$. We see that $\phi(0) = 0$, and the derivative of $\phi(T)$  with respect to $T$ is given as
    \begin{align}
        \frac{d\phi(T)}{dT} = (1+a) \left( (T+1)^a - T^a \right) , \nonumber
    \end{align}
    which is positive for any $T \geq 0$. Hence, we conclude that $\frac{d\phi(T)}{dT} > 0$ implying that $\phi(T)$ is an increasing function. Moreover, since we obtained $\phi(0) = 0$, hence the relation~\eqref{itr-bound} holds. Substituting relation~\eqref{itr-bound} back to \eqref{inv_step_sum}, we obtain the result of part (iv) for any $T \geq 0$.

    \textit{Part (v):} The lower bound result on $\sum_{t=0}^T \gamma_t^2$ can be obtained as follows
    \begin{align}
        \sum_{t=0}^T \gamma_t^2 = \sum_{t=0}^T \frac{c^2}{(t+1)^{2a}} \geq \sum_{t=0}^T \frac{c^2}{(T+1)^{2a}} = c^2 (T+1)^{1-2a} , \nonumber
    \end{align}
    which holds for any $T \geq 0$, similar to the result of part (i), and for any $0<a< \frac{1}{2}$. 
    \hfill$\square$
\end{proof}


\subsection{Proof of Lemma~\ref{lem-juditsky}}\label{lem-juditsky-proof}
\begin{proof}
We let $x = h_{t+1}$ and $y = h_t$ in Lemma~\ref{lem-3pt}, and obtain
\begin{align}
    B_\psi(z,h_{t+1}) + B_\psi(h_{t+1},h_t) - B_\psi(z,h_{t}) = \la \nabla \psi(h_t) - \nabla \psi(h_{t+1}), z - h_{t+1} \ra . \label{eq-3pt-lem}
\end{align}
By the update equation for $h_{t+1}$ and the optimality condition, we have
\begin{align*}
    \la \gamma_t b_{t+1} - \nabla \psi(h_t) + \nabla \psi(h_{t+1}), z - h_{t+1} \ra \geq 0 ,
\end{align*}
which can be written as $\la \nabla \psi(h_t) - \nabla \psi(h_{t+1}), z - h_{t+1} \ra \leq \gamma_t \la b_{t+1}, z - h_{t+1} \ra$.
By combining this relation with~\eqref{eq-3pt-lem}, we obtain
\begin{align*}
    B_\psi(z,h_{t+1}) + B_\psi(h_{t+1},h_t) - B_\psi(z,h_{t}) \leq \gamma_t \la b_{t+1}, z - h_{t+1} \ra . 
\end{align*}
Adding and subtracting the quantity $\gamma_t \la b_{t+1}, h_t \ra$ on the right-hand side of the preceding relation and reorganizing the terms, we further obtain
\begin{align}
    \gamma_t \la b_{t+1}, h_{t} - z \ra \leq B_\psi(z,h_{t}) - B_\psi(z,h_{t+1}) - B_\psi(h_{t+1},h_t) + \gamma_t \la b_{t+1}, h_t - h_{t+1} \ra . \label{rel-err-bd1}
\end{align}
The fourth term on the right-hand side of relation~\eqref{rel-err-bd1} can be upper bounded using Young's inequality and  using relation~\eqref{breg-lb} to find that
\begin{align*}
    \gamma_t \la b_{t+1}, h_t - h_{t+1} \ra \leq \frac{\gamma_t^2 \e_{t+1}^2}{2 \alpha} + \frac{\alpha}{2} \|h_{t} - h_{t+1} \|^2 \leq \frac{\gamma_t^2 \e_{t+1}^2}{2 \alpha} + B_\psi(h_{t+1},h_t) ,
\end{align*}
where $\e_{t+1} = \|b_{t+1}\|_*$ as defined in \eqref{epsilon-def}. Substituting the preceding relation back into relation~\eqref{rel-err-bd1}, we obtain the desired result. 
\hfill$\square$
\end{proof}



%
%

\bibliographystyle{spmpsci}      
\bibliography{references}   


%
%

\end{document}